\documentclass[a4paper,10pt,leqno]{amsart}
\usepackage[bf,wide, cthms]{gewoehn}

\newcommand{\E}{\mathbf{E}}
\renewcommand{\GL}{\mathbf{GL}_3}
\newcommand{\ty}[1]{\text{\Small\ensuremath{#1}}}
\newcommand{\crit}{\ensuremath{\sr T}}
\newcommand{\gl}{\mathfrak{gl}_3}

\begin{document}

\title[Components of spaces of curves on flat surfaces]{Components of spaces of curves with constrained curvature on flat surfaces}
\author{Nicolau C.~Saldanha}
\author{Pedro Z\"{u}hlke}
\maketitle

\begin{abstract}
	Let $S$ be a complete flat surface, such as the Euclidean plane. We obtain direct characterizations of the connected components of the space of all curves on $S$ which start and end at given points in given directions, and whose curvatures are constrained to lie in a given interval, in terms of all parameters involved. Many topological properties of these spaces are investigated. Some conjectures of L.\thinspace E.\thinspace Dubins are proved.
\end{abstract}

\setcounter{tocdepth}{1}
\tableofcontents

\setcounter{section}{-1}

\section{Introduction}\label{S:introduction}
To abbreviate the notation, we shall identify $\R^2$ with $\C$ throughout. A curve $\ga\colon [0,1]\to \C$ is called \tdef{regular} if its derivative is continuous and never vanishes. Its \tdef{unit tangent} is then defined as
\begin{equation*}
	\ta_\ga\colon [0,1]\to \Ss^1,\quad \ta_\ga(t)=\frac{\dot\ga(t)}{\abs{\dot\ga(t)}}.
\end{equation*}
Lifting $\ga$ to the unit tangent bundle $UT\C\equiv \C\times \Ss^1$, we obtain its \tdef{frame} 
\begin{equation}\label{E:frame}
	\Phi_\ga\colon [0,1]\to \C\times \Ss^1,\quad \Phi_\ga(t)=(\ga(t),\ta_\ga(t)).
\end{equation}
Let $P=(p,w),~Q=(q,z)\in \C\times \Ss^1$ and consider the spaces of curves 
\begin{equation}\label{E:spaces}
	\begin{alignedat}{9}
		\sr S(P,Q)&=\set{\ga\colon [0,1]\overset{C^r}{\longrightarrow} \C}{\text{\,$\ga$ is regular, }\Phi_\ga(0)=P\text{ and }\Phi_\ga(1)=Q}\\
		\Om UT\C(P,Q)&=\set{\om\colon [0,1]\overset{C^{r-1}}{\longrightarrow} UT\C}{\,\om(0)=P\text{ and }\om(1)=Q}
	\end{alignedat}
\end{equation}
endowed with the $C^r$ (resp.~$C^{r-1}$) topology ($1\leq r\in \N$). In 1956, S.~Smale proved that the map 
\begin{equation*}
	\Phi\colon \sr S(P,Q)\to \Om UT\C(P,Q),\quad \ga\mapsto \Phi_\ga
\end{equation*}
is a weak homotopy equivalence (that is, it induces isomorphisms on homotopy groups).  Actually, Smale's theorem (\cite{Smale}, Theorem C) is much more general in that it holds for any manifold, not just $\C$. Using standard results on Banach manifolds which were discovered later, one can conclude that the spaces in \eqref{E:spaces} are in fact homeomorphic, and that the value of $r$ is unimportant.

Given a regular plane curve $\ga$, an \tdef{argument} of $\ta_\ga$ is a continuous function $\theta_\ga\colon [0,1]\to \R$ such that $\ta_\ga=e^{i\theta_\ga}$. The \tdef{total turning} of $\ga$ is defined to be $\theta_\ga(1)-\theta_\ga(0)$; note that this is independent of the choice of $\theta_\ga(0)$. It is easy to see that $\Om UT\C(P,Q)$ is homotopy equivalent to $\Om \Ss^1 (w,z)$. The latter possesses infinitely many connected components, one for each $\theta_1$ satisfying $e^{i\theta_1}=z\bar w=zw^{-1}$, all of which are contractible.  Therefore, the components of $\sr S(P,Q)$ are all contractible as well, and two curves in $\sr S(P,Q)$ lie in the same component if and only if they have the same total turning. This generalizes the Whitney-Graustein theorem (\cite{WhiGra}, Theorem 1) to non-closed curves.

The main purpose of this work is to investigate the topology of subspaces of $\sr S(P,Q)$ obtained by imposing constraints on the curvature of the curves. 

\begin{defn}\label{D:planecurve}
	Let $-\infty\leq \ka_1<\ka_2\leq  +\infty$ and $r\in \se{2,3,\dots,\infty}$. For $P=(p,w),~Q=(q,z)\in \C\times \Ss^1$, define $\sr C_{\ka_1}^{\ka_2}(P,Q)$ to be the set of all $C^r$ regular curves $\ga\colon [0,1]\to \C$ such that:
\begin{enumerate}
	\item [(i)] $\Phi_\ga(0)=P$ and $\Phi_\ga(1)=Q$;
	\item [(ii)]  The curvature $\ka_\ga$ of $\ga$ satisfies $\ka_1<\ka_\ga(t)<\ka_2$ for each $t\in [0,1]$.
\end{enumerate}
Let this set be furnished with the $C^r$ topology.
\end{defn}

Condition (i) means that $\ga$ starts at $p$ in the direction of $w$ and ends at $q$ in the direction of $z$.  In this notation, $\sr S(P,Q)$ becomes $\sr C_{-\infty}^{+\infty}(P,Q)$.  The connected components of $\sr C_{-\ka_0}^{+\ka_0}(P,Q)$ ($\ka_0>0$) were first studied by L.\,E.~Dubins in \cite{Dubins1}. His main result (Theorem 5.3, slightly rephrased) implies that there may exist curves with the same total turning which are not homotopic within this space.

\begin{uthm}[Dubins]\label{P:Dubinsintro}
	Let $x>0$ and $Q_x=(x,1)$,\,$O=(0,1)\in \C\times \Ss^1$. Let $\eta \in \sr C_{-1}^{+1}(O,Q_x)$ be the line segment parametrized by $\eta(t)=xt$. Then the concatenation of $\eta$ with a  figure eight curve lies in the same connected component of $\sr C_{-1}^{+1}(O,Q_x)$ as $\eta$ if and only if $x>4$.
\end{uthm}

By a ``figure eight'' curve it is meant any closed curve of total turning 0 whose curvature takes values in $(-1,1)$, such as the one depicted in \fref{F:spreading}\?(d). 


Naturally, we always have the following decomposition of $\sr C_{\kappa_1}^{\kappa_2}(P,Q)$ into closed-open subspaces:
\begin{equation*}
	\sr C_{\ka_1}^{\ka_2}(P,Q)=\bigsqcup_{\theta_1}\sr C_{\kappa_1}^{\kappa_2}(P,Q;\theta_1),
\end{equation*}
where $\sr C_{\kappa_1}^{\kappa_2}(P,Q;\theta_1)$ consists of those curves in $\sr C_{\kappa_1}^{\kappa_2}(P,Q)$ which have total turning equal to $\theta_1$ and the union is over all $\theta_1\in \R$ satisfying $e^{i\theta_1}=z\bar w$.

If $\ka_1\ka_2\geq 0$, it will be shown that each $\sr C_{\kappa_1}^{\kappa_2}(P,Q;\theta_1)$ is either empty or a contractible connected component of $\sr C_{\kappa_1}^{\kappa_2}(P,Q)$.\footnote{In determining the sign of $\ka_1\ka_2$, we adopt the convention that $0(\pm \infty)=0$.} 
If $\ka_1\ka_2<0$, then $\sr C_{\kappa_1}^{\kappa_2}(P,Q;\theta_1)$ is never empty, and it is a contractible connected component provided that $\abs{\theta_1}\geq \pi$. However, 
the remaining subspace $\sr C_{\kappa_1}^{\kappa_2}(P,Q;\theta_1)$ with $\abs{\theta_1}<\pi$ may not be contractible, nor even connected, as implied by Dubins' theorem. It turns out that one can obtain simple and explicit characterizations of its components in terms of $\ka_1,\,\ka_2,\,P$ and $Q$ by using a homeomorphism with a space of the form $\sr C_{-1}^{+1}(P_0,Q_0;\theta_1)$ and an elementary geometric construction (see Figure \ref{F:S0}).

\begin{figure}[ht]
	\begin{center}
		\includegraphics[scale=.32]{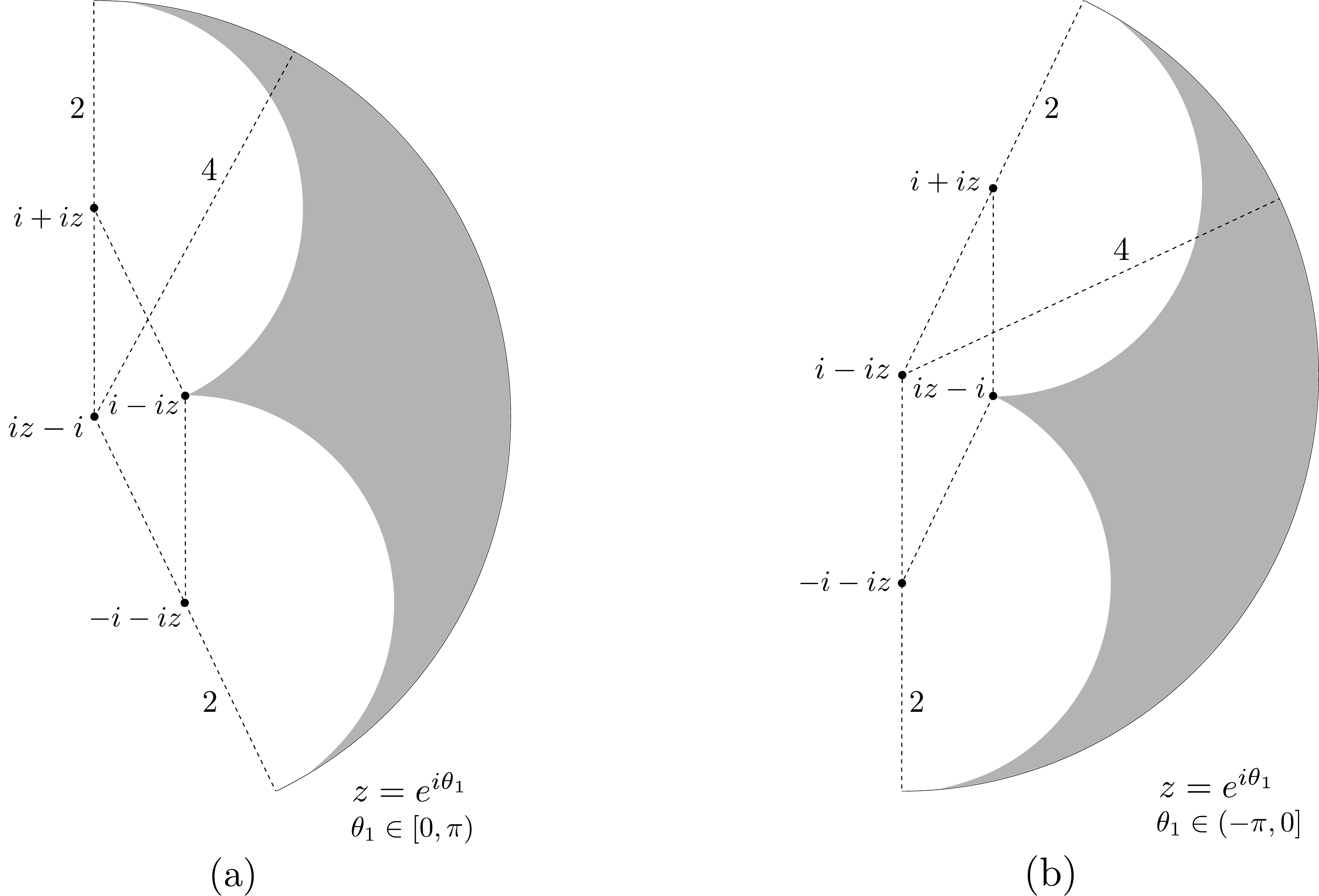}
		\caption{Let $\theta_1\in \R$ be fixed, $z=e^{i\theta_1}$ and $Q=(q,z)$. Then $\sr C_{-1}^{+1}(Q;\theta_1)$ is disconnected if and only if $\abs{\theta_1}<\pi$ and $q$ lies in the gray region. The region contains the arc of circle of radius 4, but not the arcs of circle of radius 2. Figure (a) depicts the case $\theta_1\in [0,\pi)$, and (b) the case $\theta_1\in (-\pi,0]$ (here $\theta_1\approx \pm 26^\circ$). The  theorem of Dubins stated above corresponds to the case where $\theta_1=0$ and $q\in \R$.}
		\label{F:S0}
	\end{center}
\end{figure}

This paper is close in spirit to Dubins' \cite{Dubins1}, and some of his conjectures will be settled; it is \textsl{not} assumed, however, that the reader is familiar with his work. In the sequel \cite{SalZueh2} to this article, we determine the homotopy type of $\sr C_{\kappa_1}^{\kappa_2}(P,Q)$. Granted the results described above, the only remaining task is the determination of the homotopy type of the exceptional subspace $\sr C_{\ka_1}^{\ka_2}(P,Q;\theta_1)\subs \sr C_{\ka_1}^{\ka_2}(P,Q)$ with $\abs{\theta_1}<\pi$ ($\ka_1\ka_2<0$), containing the curves in the latter of least total turning.  It is proved in \cite{SalZueh2} that this subspace may be homotopy equivalent to an $n$-sphere for any $n\in \se{0,1,\dots,\infty}$ (recall that $\Ss^\infty$ is contractible). The value of $n$ can be determined in terms of all parameters by first reducing to the case where $\ka_1=-1$, $\ka_2=+1$ through the homeomorphism mentioned above, and then using a construction extending the one depicted in Figure \ref{F:S0} (which only tells whether $n=0$ or not).

\subsection*{Outline of the sections}
Many useful constructions, such as the concatenation of elements of $\sr C_{\ka_1}^{\ka_2}(P,Q)$ and $\sr C_{\ka_1}^{\ka_2}(Q,R)$, yield curves which need not be of class $C^2$. To avoid having to smoothen curves all the time, we work with curves which have a continuously varying unit tangent at all points, but whose curvatures are defined only almost everywhere. The resulting spaces, denoted by $\sr L_{\kappa_1}^{\kappa_2}(P,Q)$, are defined in \S \ref{S:plane}, where it will also be seen that the set inclusion $\sr C_{\kappa_1}^{\kappa_2}(P,Q)\to \sr L_{\kappa_1}^{\kappa_2}(P,Q)$ is a homotopy equivalence with dense image and that these spaces are homeomorphic.


Let $O=(0,1)\in \C\times \Ss^1$ denote the canonical element of $UT\C$, and let us denote $\sr C_{\kappa_1}^{\kappa_2}(O,Q)$ simply by $\sr C_{\kappa_1}^{\kappa_2}(Q)$. Using Euclidean motions, dilatations and a construction called \tdef{normal translation} (see Figure \ref{F:translation} on p.~\pageref{F:translation}), we obtain in \tref{T:normalized} an explicit homeomorphism between any space $\sr C_{\kappa_1}^{\kappa_2}(P_0,Q_0)$ and a space of one of the following types: $\sr C_{0}^{+\infty}(Q)$, $\sr C_{1}^{+\infty}(Q)$ or $\sr C_{-1}^{+1}(Q)$, according as $\ka_1\ka_2=0$, $\ka_1\ka_2>0$ or $\ka_1\ka_2<0$, respectively. Moreover, this homeomorphism preserves the total turning of curves up to sign.  Among these three, $\sr C_{-1}^{+1}(Q)$ has the most interesting topological properties.

We call a regular curve $\ga\colon [0,1]\to \C$ \tdef{condensed}, \tdef{critical} or \tdef{diffuse}, according as its \tdef{amplitude}
	\begin{equation*}
		\om=\sup_{t\in [0,1]}\theta_\ga(t)-\inf_{t\in [0,1]}\theta_\ga(t)
	\end{equation*}
satisfies $\om<\pi$, $\om=\pi$ or  $\om>\pi$, respectively. Let $Q=(q,z)\in \C\times \Ss^1$ and $\theta_1$ be such that $e^{i\theta_1}=z$. Let $\sr U_c,\,\sr U_d\subs \sr C_{-1}^{+1}(Q;\theta_1)$ denote the subspaces consisting of all condensed (resp.~diffuse) curves. Both are open and $\sr U_d\neq \emptyset$, since we may always concatenate a curve in $\sr C_{-1}^{+1}(Q;\theta_1)$ with a curve of total turning 0 (an eight curve, as in Figure \ref{F:spreading}\?(e) on p.~\pageref{F:spreading}). Clearly, $\sr U_c$ must be empty if $\abs{\theta_1}\geq \pi$, but it may also be empty otherwise, depending on $Q$. We determine exactly when this occurs in \S\ref{S:condensed}. 

A condensed curve may be viewed as the graph of a function with respect to some axis. This leads to a direct, albeit involved, proof that $\sr U_c$ is contractible, when nonempty. In fact, if the curvatures are allowed to be discontinuous and to take values in the {closed} interval $[-1,1]$, then one can exhibit a contraction of the subspace of condensed curves to the unique curve of minimal length (\tdef{Dubins path}) in the corresponding space. This is also done in \S\ref{S:condensed}. 

In \S\ref{S:diffuse} an indirect proof that $\sr U_d$ is contractible is obtained. If $\ga$ is diffuse, then we can ``graft'' straight lines segment onto $\ga$, as illustrated in Figure \ref{F:grafting}, p.~\pageref{F:grafting}. Such a segment can be deformed so that in the end an eight curve of large radius traversed a number $n$ of times has been attached to it. These eights are then spread along the curve, as in Figure \ref{F:spreading}\?(f). If $n\in \N$ is large enough, then the whole process can be carried out within $\sr C_{-1}^{+1}(Q)$. The result is a curve whose curvature is uniformly small, and hence easily deformable.

In \S\ref{S:critical} we determine when the set $\sr T$ of all critical curves in $\sr C_{-1}^{+1}(Q;\theta_1)$ is empty. The main result in this section is that  $\sr T=\bd \sr U_c=\bd \sr U_d$. When $\sr T\neq \emptyset$, a finer analysis of how $\bd \sr U_c$ and $\bd \sr U_d$ fit together is required to determine the homeomorphism class of $\sr C_{-1}^{+1}(Q;\theta_1)$. This problem will be treated in \cite{SalZueh2}.

In \tref{T:criterion} we obtain various characterizations of the connected components of $\sr C_{-1}^{+1}(Q;\theta_1)$. Perhaps the simplest one is the following: this space is disconnected if and only if $\abs{\theta_1}<\pi$ and $q$ lies in the region illustrated in Figure \ref{F:S0}, or, equivalently, its subset $\sr T$ is empty, but $\sr U_c$ is not. In this case, it has exactly two components, $\sr U_c$ and $\sr U_d$, which are contractible. As mentioned previously, this is sufficient to determine explicitly the components of any space $\sr C_{\kappa_1}^{\kappa_2}(P_0,Q_0)$ with $\ka_1\ka_2<0$.

In \S\ref{S:locallyconvex} it is established that when $\ka_1\ka_2\geq 0$, the space $\sr C_{\kappa_1}^{\kappa_2}(P,Q)$ has one connected component for each realizable total turning, and they are all contractible. The set of possible total turnings can be described in terms of all parameters using normal translation and elementary geometry. The detailed solution to this problem is not carried out to shorten the paper, but it can be found in the earlier unpublished version \cite{SalZuehold}.  

In \S\ref{S:flat} these results are extended to spaces of curves with constrained curvature on any complete flat surface $S$ (orientable or not), using the fact that if $S$ is connected then it must be the quotient of $\C$ by a group of isometries.

Even though we have imposed that the curvatures should lie in an open interval, the main results obtained here have analogues for spaces (defined in \S\ref{S:plane}) where the curvature is constrained to lie in $[\ka_1,\ka_2]$. For $\ka_1=-\ka_2$, this is the class that Dubins actually worked with in \cite{Dubins} and \cite{Dubins1}. The necessary modifications in the statements and proofs are sketched in \S\ref{S:final}, where we also prove some conjectures appearing in \cite{Dubins1} and discuss a few additional conjectures on curves of minimal length.

\subsection*{Related work} The problem treated here and in \cite{SalZueh2} for flat surfaces can be generalized to any smooth (or even $C^2$) surface $S$ equipped with a Riemannian metric: If $u,v$ are elements of its unit tangent bundle $UTS$, then one can study the space $\sr CS_{\ka_1}^{\ka_2}(u,v) $ of curves on $S$ whose lift to $UTS$ joins $u$ to $v$ and whose geodesic curvature takes values in $(\ka_1,\ka_2)$. When $S$ is nonorientable, only the unsigned curvature makes sense, so in this case we require that $\ka_2=-\ka_1>0$ (cf. \S \ref{S:flat} below). This topic is largely unexplored, and even the problem of determining when $\sr CS_{\ka_1}^{\ka_2}(u,v)\neq \emptyset$ is open (and probably difficult). The topology of these spaces is very closely related to the geometry of $S$.

A special case which has been more intensively studied is that of the space of \tdef{nondegenerate} curves on $S$, i.e., curves of nonvanishing curvature. In our notation, this corresponds to $\sr CS_{0}^{+\infty}(u,v)\du \sr CS_{-\infty}^{0}(u,v)$. There is also an obvious generalization to higher-dimensional manifolds, obtained by replacing the (geodesic) curvature by the generalized curvature of a curve $\ga\colon [0,1]\to M^n$. To say that the latter does not vanish is equivalent to requiring that the first $n$ (covariant) derivatives of $\ga$ at $\ga(t)$ span the tangent space to $M$ at this point, for each $t\in [0,1]$. Some papers treating this problem, especially for spaces of closed curves on the simplest manifolds, such as $\R^n$, $\Ss^n$ or $\RP^n$, include \cite{Anisov}, \cite{Feldman}, \cite{Feldman1}, \cite{KheSha}, \cite{KheSha1}, \cite{Little}, \cite{Little1}, \cite{MosSad} \cite{Saldanha3}, \cite{SalSha}, \cite{ShaSha}, \cite{ShaTop}. Most of these are concerned with obtaining characterizations of the connected components of the corresponding spaces. 

In \cite{Saldanha3} the homotopy type of spaces of (not necessarily closed) nondegenerate curves on $\Ss^2$ is determined, and in \cite{SalZueh} the connected components of spaces of closed curves on $\Ss^2$ with curvature in an arbitrary interval $(\ka_1,\ka_2)$ are characterized. In the sequel \cite{SalZueh2} we determine the homotopy type of $\sr CS_{\ka_1}^{\ka_2}(u,v)$ for any flat surface $S$ in terms of $\ka_1,\ka_2$ and $u,v\in UTS$. Many of the ideas appearing in the present paper (normal translation, diffuse vs.~condensed, grafting, curvature spreading, etc.) appear in \cite{Saldanha3} or \cite{SalZueh} in some form as well, although sometimes the connection is only heuristical.


\section{Spaces of plane curves}\label{S:plane}
\subsection*{Basic terminology} 
Let $\ga\colon [a,b]\to \C$ be a regular curve. 
The \tdef{unit normal} $\no=\no_\ga\colon [a,b]\to \Ss^1$ is given by $\no=i\ta$, where $i\in \C$ denotes the imaginary unit and $\ta=\ta_\ga$ is the unit tangent to $\ga$. 
The \tdef{arc-length parameter} $s$ of $\ga$ is defined by 
\[
s(t)=\int_a^t \abs{\dot{\ga}(\tau)}\?d\tau\quad (t\in [a,b])
\]
and $L=\int_a^b \abs{\dot \ga(\tau)}\?d\tau$ is the \tdef{length} of $\ga$. Assuming $\ga$ is twice differentiable, its \tdef{curvature} $\ka=\ka_\ga$ is given by
\begin{equation}\label{E:curvature1}
	\ka(s)=\big\langle\ta'(s),\no(s)\big\rangle\quad (s\in [0,L]).
\end{equation}
In terms of a general parameter,
\begin{equation}\label{E:curvature}
	\ka=\frac{1}{\abs{\dot\ga}}\big\langle\dot\ta,\no\big\rangle=\frac{1}{\abs{\dot\ga}^2}\big\langle\ddot\ga,\no\big\rangle=\frac{\det(\dot\ga,\ddot \ga)}{\abs{\dot\ga}^3}.
\end{equation}
[We denote derivatives with respect to arc-length by a $\vphantom{\ga}'$ (prime) and derivatives with respect to other parameters by a $\dot{\vphantom{\ga}}$ (dot).] Notice that the curvature at each point is not altered by an orientation-preserving reparametrization of the curve, while its sign changes if the reparametrization is orientation-reversing. It follows from \eqref{E:curvature1} that if $\theta_\ga\colon [0,L]\to \R$ is an argument of $\ta$, then 
\begin{equation}\label{E:rate}
	\ka(s)=\theta_\ga'(s).
\end{equation}

	The following example illustrates one reason why it is more convenient to require that curvatures lie in an open interval, as in \dref{D:planecurve}.
\begin{exm}\label{E:closed}
	Consider the space of all $C^2$ regular curves $\ga\colon [0,1]\to \C$ whose curvatures are restricted to lie in $[-1,1]$ and which satisfy $\Phi_\ga(0)=(1,i)$, $\Phi_\ga(1)=(i,-1)$, where we have identified $UT\C$ with $\C\times \Ss^1$. The arc $\al$ of the unit circle given by $t\mapsto \exp(\frac{\pi i}{2}t)$ ($t\in [0,1]$) is a curve in this space. In fact, it is not hard to see that $\al$ is an isolated point, i.e., its connected component does not contain any other curve. 
\end{exm}
In contrast, the spaces $\sr C_{\ka_1}^{\ka_2}(P,Q)^r$ are Banach manifolds (for $r\neq \infty$). Still, some useful constructions, such as the concatenation of curves, lead out of this class of spaces. To avoid having to smoothen curves all the time, we shall work with another class of spaces, which possess the additional advantage of being Hilbert manifolds.

\subsection*{The group structure of $UT\C$}The group of all orientation-preserving isometries of $\C$ (i.e., proper Euclidean motions) acts simply transitively on $UT\C$.  An element of this group is thus uniquely determined by where it maps $(0,1)\in UT\C$, and may be identified with this image. Therefore, $UT\C$ carries a natural Lie group structure as a semidirect product $\C\sd \Ss^1$, wherein the operation is:
\begin{equation*}
\qquad	(p,w)\cdot(q,z)=(p+wq,wz)\qquad (p,q\in \C,~w,z\in \Ss^1).
\end{equation*}
Accordingly, viewed as a one-parameter family of Euclidean motions, the frame $\Phi_\ga$ of a regular curve $\ga\colon [0,1]\to \C$ operates on $\C$ through
\begin{equation}\label{E:group}
	\Phi_\ga(t)a=\ga(t)+\ta(t)a\qquad (a\in \C,~t\in [0,1]).
\end{equation}
If we identify the Lie algebra of $UT\C$ with $\C\times \R$, then the bracket operation is given by
\begin{equation}\label{E:bracket}
\qquad	[(a,\theta)\?,\?(b,\vphi)]=\big(i(\theta b-\vphi a),0\big)\qquad (a,b\in \C,~\theta,\vphi\in \R).
\end{equation}
We can also realize $UT\C$ as a matrix group if we identify
\begin{equation*}
	P=(p,w)\text{\ \,with\ \,}\begin{pmatrix}
	\cos\theta & -\sin\theta & x \\
	\sin\theta & \cos\theta & y \\
	0 & 0 & 1
	\end{pmatrix},\text{\ where\ \,$p=x+iy$,\ $w=e^{i\theta}$.}
\end{equation*}
Then $\Phi_\ga$ corresponds to the map
\begin{equation}\label{E:frame2}
	\Phi_\ga\colon [0,1]\to \GL,\quad \Phi_\ga(t)=\begin{pmatrix}
	\cos\theta_\ga(t) & -\sin\theta_\ga(t) & \ga_1(t) \\
	\sin\theta_\ga(t) & \phantom{-}\cos\theta_\ga(t) & \ga_2(t) \\
	0 & 0 & 1
	\end{pmatrix},
\end{equation}
where $\theta_\ga\colon [0,1]\to \R$ is an argument of $\ta_\ga$ and $\ga(t)=\ga_1(t)+i\ga_2(t)$.\footnote{Notice that the first column of $\Phi_\ga$ gives the coordinates of $\ta_\ga$, the second the coordinates of $\no_\ga$ and the third the coordinates of $\ga$. This justifies our terminology ``frame'' for $\Phi_\ga$.} Moreover, under this identification the Lie algebra $\mathfrak{a}$ of $UT\C$ becomes a subalgebra of $\gl$,  generated by
\begin{equation}\label{E:AandB}
	A=\begin{pmatrix}
	0 & -1 & 0 \\
	1 & 0 & 0 \\
	0 & 0 & 0
	\end{pmatrix},\ \ B=\begin{pmatrix}
	0 & 0 & 1 \\
	0 & 0 & 0 \\
	0 & 0 & 0
	\end{pmatrix}\text{\ \ and\ \ }C=\begin{pmatrix}
	0 & 0 & 0 \\
	0 & 0 & 1 \\
	0 & 0 & 0
	\end{pmatrix}.
\end{equation}
The expression for the bracket in \eqref{E:bracket} can be easily derived from this.  

\subsection*{Spaces of admissible curves}Suppose now that $\ga\colon [0,1]\to \C$ is not only regular, but also smooth. Let $\ka$ denote its curvature and $\sig=\abs{\dot\ga }$ its speed. Using \eqref{E:rate}, we deduce that
\begin{equation*}
	\dot\Phi_\ga=\abs{\dot\ga}\begin{pmatrix}
	-\ka\sin\theta_\ga & -\ka\cos\theta_\ga & \cos\theta_\ga \\
	\phantom{-}\ka\cos\theta_\ga & -\ka\sin\theta_\ga & \sin\theta_\ga \\
	0 & 0 & 0
	\end{pmatrix}=\Phi_\ga\La_\ga, \text{\ \ where\ \ }\La_\ga=\sig\begin{pmatrix}
	0 & -\ka & 1 \\
	\ka & 0 & 0 \\
	0 & 0 & 0
	\end{pmatrix}.
\end{equation*}
Let $\mathfrak{h}\subs \mathfrak{a}$ denote the half-plane 
\begin{equation}\label{E:hp}
	\mathfrak{h}=\set{aA+bB}{a\in \R,~b>0}.
\end{equation}
The map $\La_\ga\colon [0,1]\to \mathfrak{h}$ is called the \tdef{logarithmic derivative} of $\ga$. The crucial observation for us is that $\Phi_\ga$ (and hence $\ga$) is uniquely determined as the solution of an initial value problem 
\begin{equation}\label{E:ivp}
	\Phi(0)=P\in UT\C,\ \ \dot\Phi=\Phi\La,\ \ \text{where}\ \ \La\colon [0,1]\to \mathfrak{h},\ \ \La=\sig\begin{pmatrix}
	0 & -\ka & 1 \\
	\ka & 0 & 0 \\
	0 & 0 & 0
	\end{pmatrix}.
\end{equation}
Equivalently, $\ga$ is uniquely determined by $P=\Phi_\ga(0)$ and the pair of functions $\ka\colon [0,1]\to \R$ and $\sig\colon [0,1]\to \R^+$. Our preferred class of spaces is obtained by relaxing the requirements that $\sig$ and $\ka$ be smooth.

Let $h=h_{0,+\infty}\colon (0,+\infty)\to \R$ be the smooth diffeomorphism
\begin{equation*}\label{E:thereason}
	h(t)=t-t^{-1}.
\end{equation*} 
More generally, for each pair $\ka_1<\ka_2\in \R$, let $h_{\ka_1,\,\ka_2} \colon (\ka_1,\ka_2)\to \R$ be the smooth diffeomorphism
\begin{equation*}\label{E:hfunctions}
	h_{\ka_1,\,\ka_2}(t)=(\ka_1-t)^{-1}+(\ka_2-t)^{-1}
\end{equation*}
and, similarly, set
\begin{alignat*}{10}
&h_{-\infty,+\infty}\colon \R\to \R\qquad  & &h_{-\infty,+\infty}(t)=t \\
 &h_{-\infty,\ka_2}\colon (-\infty,\ka_2)\to \R\qquad &  &h_{-\infty,\ka_2}(t)=t+(\ka_2-t)^{-1} \\
 &h_{\ka_1,+\infty}\colon (\ka_1,+\infty)\to \R \qquad & &h_{\ka_1,+\infty}(t)=t+(\ka_1-t)^{-1}.
\end{alignat*}
\begin{rem}\label{R:preservesL2ness}
	All of these functions are monotone increasing, hence so are their inverse functions. Also, if $\hat{\ka}\in L^2[0,1]$, then $\ka=h_{\ka_1,\ka_2}^{-1}\circ \hat\ka\in L^2[0,1]$ as well. This is obvious if $(\ka_1,\ka_2)$ is bounded, and if one of $\ka_1,\ka_2$ is infinite then it is a consequence of the fact that $h_{\ka_1,\ka_2}^{-1}(t)$ diverges linearly to $\pm \infty$ with respect to $t$.
\end{rem}

In all that follows, $\E$ denotes the separable Hilbert space $L^2[0,1]\times L^2[0,1]$. The $(i,j)$-entry of a matrix $A$ will be denoted by $A^{(i,j)}$.
\begin{defn}
Let $-\infty\leq \ka_1<\ka_2\leq +\infty$ and $P\in UT\C$. A curve $\ga\colon [0,1]\to \C$, $\ga=\ga_1+i\ga_2$, will be called \tdef{$(\ka_1,\ka_2)$-admissible} if $\ga_1=\Phi^{(1,3)},\,\ga_2=\Phi^{(2,3)}$ for $\Phi\colon [0,1]\to UT\C$ satisfying \eqref{E:ivp}, with
\begin{equation}\label{E:Sobolev}
	\sig=h^{-1}\circ \hat{\sig},\quad \ka=\?h^{-1}_{\ka_1,\,\ka_2}\circ \hat{\ka},\quad (\hat\sig,\hat\ka)\in \E.
\end{equation}
When it is not important to keep track of the bounds $\ka_1,\ka_2$, we will simply say that $\ga$ is \tdef{admissible}.
\end{defn}

The differential equation \eqref{E:ivp} has a unique solution $\Phi$ for any $(\hat \sig,\hat \ka)\in \E$ and $P\in UT\C$. This follows from Theorem C.3 on p.~386 of \cite{Younes}, using the fact that $\sig, \ka\in L^2[0,1]\subs L^1[0,1]$. Moreover, $\Phi$ is absolutely continuous (see p.~385 of \cite{Younes}) and defined over all of $[0,1]$  (since $\C$ is complete). The resulting maps $\ta\colon [0,1]\to \Ss^1$, $\no\colon [0,1]\to \Ss^1$ and $\ga\colon [0,1]\to \C$, obtained from the first, second and third columns of $\Phi$, respectively, are thus absolutely continuous. It follows from \eqref{E:ivp} that 
\begin{equation}\label{E:together}
	\dot\ga=\sig \ta,\quad \dot\ta=\sig\ka\no\quad\text{and}\quad\dot\no=-\sig\ka\ta.
\end{equation}
Furthermore, if $\Psi$ denotes the $2\times 2$ matrix obtained from $\Phi$ by discarding its third column and line, then $\Psi\colon [0,1]\to \SO_2$, as one sees by differentiating $\Psi\Psi^T$, using \eqref{E:ivp} and noting that  $\Psi(0)\in \SO_2$. Hence, $\no=i\ta$. Differentiation of $\abs{\ta}^2$ yields that 
\begin{equation*}
	\abs{\no(t)}=\abs{\ta(t)}=\abs{\ta(0)}=1\ \ \text{for all $t\in [0,1]$}.
\end{equation*}
Comparing \eqref{E:together}, it is thus natural to define $\ta_\ga=\ta$, $\no_\ga=\no$, $\Phi_\ga=\Phi$, and to call $\sig$ and $\ka$ the \tdef{speed} and \tdef{curvature} of $\ga$, respectively, even though $\sig,\ka\in L^2[0,1]$. With this definition, $\ta_\ga$, $\no_\ga$, $\Phi_\ga$ and any argument $\theta_\ga=\arg\circ \ta_\ga$ are  absolutely continuous functions, as remarked above. Although $\dot \ga=\sig\ta_\ga$ is defined only almost everywhere on $[0,1]$, if we reparametrize $\ga$ by arc-length then it becomes a regular curve, since $\ga'=\ta_\ga$. Instead of thinking of $\ga$ as corresponding to a pair of $L^2$ functions, it is more helpful to regard $\ga$ as a regular curve whose curvature is defined only a.e.. In fact, all of the concrete examples of admissible curves considered below are piecewise $C^2$ curves.

\begin{defn}\label{D:loose}
Let $-\infty\leq \ka_1<\ka_2\leq +\infty$. For $P\in UT\C$, define $\sr L_{\ka_1}^{\ka_2}(P,\cdot)$ to be the set of all $(\ka_1,\ka_2)$-admissible curves $\ga\colon [0,1]\to \C$ with $\Phi_\ga(0)=P$. This set is identified with $\E$ via the correspondence $\ga\leftrightarrow (\hat \sig,\hat \ka)$, thus furnishing $\sr L_{\ka_1}^{\ka_2}(P,\cdot)$ with a trivial Hilbert manifold structure.
\end{defn}
The `$\sr L$' is intended to remind one of $L^2$ functions.

\begin{lem}\label{L:submersion}
	Let $-\infty\leq \ka_1<\ka_2\leq +\infty$ and $P\in UT\C$. Then 
	\begin{equation*}
		F\colon \sr L_{\ka_1}^{\ka_2}(P,\cdot)\to UT\C,\quad \ga\mapsto \Phi_\ga(1),
	\end{equation*}
	is a submersion. Consequently, it is an open map.
\end{lem}
\begin{proof}
	Let $\de>0$, $r\in (-\de,\de)$ and $\hat\sig(r),\,\hat\ka(r) \in L^2[0,1]$ be one-parameter families of functions; set $\sig(r)=h^{-1}\circ \hat\sig(r),~\ka(r)=h_{\ka_1,\ka_2}^{-1}\circ \hat\ka(r)$. Define a corresponding family of curves $\La(r)\colon [0,1]\to \mathfrak{h}$ by
	\begin{equation*}
		\La(r)=\sig(r)\begin{pmatrix}
	0 & -\ka(r) & 1 \\
	\ka(r) & 0 & 0 \\
	0 & 0 & 0
	\end{pmatrix}.
	\end{equation*}
	Denoting derivatives with respect to $t$ (resp.~$r$) by a $\dot{\phantom{a}}$ (resp.~${\phantom{a}}'$), let $\Phi(r)\colon [0,1]\to UT\C$, $t\mapsto \Phi(r)(t)$, be the solution of $\dot\Phi(r)=\Phi(r)\La(r)$.  A straightforward computation shows that 
	\begin{equation*}
		\Phi'(r)(t)\big[\Phi(r)(t)\big]^{-1}=\int_0^t\Phi(r)(\tau)\, \La'(r)(\tau) \,\big[\Phi(r)(\tau)\big]^{-1}\,d\tau\quad (r\in (-\de,\de),~t\in [0,1]).
	\end{equation*}
	Let $\La'(0)$ consist of three smooth narrow bumps at times $t=t_0$, $t=t_1$ and $t=t_2$, with each $t_i\in (0,1)$ close to 1.
Let $\Psi = \Phi(0)$;
setting $r=0$ in the previous expression, we deduce that
	\begin{equation*}
\Psi(1)^{-1}\Phi'(0)(1)\approx \sum_{i=1}^3
\big[\Psi(t_i)^{-1}\Psi(1)\big]^{-1}\,\La'(0)(t_i)\,\big[\Psi(t_i)^{-1}\Psi(1)\big].
	\end{equation*}
	Since each $\La(r)$ is a curve in the open convex cone
	\begin{equation*}
		\set{aA+bB}{a\in \R,~b>0\text{ and }\ka_1b<a<\ka_2b},
	\end{equation*} 
	we can make $\La'(0)(t_i)$ assume any value in the vector subspace $\mathfrak{v}$
generated by $A$ and $B$ (with $A,B$ as in \eqref{E:AandB}). Another computation using the fact that $\sig(0)>0$ a.e.~shows that the planes $\mathfrak{v}$ and $\big[\Psi(t_i)^{-1}\Psi(1)\big]^{-1}\mathfrak{v}\big[\Psi(t_i)^{-1}\Psi(1)\big]$ are transversal for small $1-t_i$, with the angle between them proportional to $(1-t_i)+o(1-t_i)$. Hence, any vector in $\mathfrak{a}$ can be written in the form $\Psi(1)^{-1}\Phi'(0)(1)$ for a suitable choice of $\La'(0)$, which shows that $F$ is a submersion. 
\end{proof}		
	
\begin{defn}\label{D:Lurvespace}
	Let $-\infty\leq \ka_1<\ka_2\leq +\infty$ and $P,Q\in UT\C$. Define $\sr L_{\ka_1}^{\ka_2}(P,Q)$ to be the subspace of $\sr L_{\ka_1}^{\ka_2}(P,\cdot)$ consisting of all $\ga\in \sr L_{\ka_1}^{\ka_2}(P,\cdot)$ such that $\Phi_\ga(1)=Q$.
\end{defn}

It follows from \lref{L:submersion} that $\sr L_{\ka_1}^{\ka_2}(P,Q)$ is a closed submanifold of codimension 3 in $\sr L_{\ka_1}^{\ka_2}(P,\cdot)\equiv \E$; the proof that $\sr L_{\ka_1}^{\ka_2}(P,Q)$ is always nonempty is postponed to \S\ref{S:diffuse}.

The following lemmas contain all the results on infinite-dimensional manifolds that we shall use.

\begin{lem} Let $\sr M,\,\sr N$ be (infinite-dimensional) separable Banach manifolds. Then:\label{L:Hilbert}
	\begin{enumerate}
		\item [(a)] $\sr M$ is locally path-connected. In particular, its connected and path components coincide. 
		\item [(b)]  If $F\colon \sr M\to \sr N$ is a weak homotopy equivalence,  then $F$ is homotopic to a homeomorphism. 
		\item [(c)] Let $\E$ and $\mathbf{F}$ be separable Banach spaces. Suppose $i\colon \mathbf{F}\to \E$ is a bounded, injective linear map with dense image and $\sr M\subs \E$ is a smooth closed submanifold of finite codimension.  Then $\sr N=i^{-1}(\sr M)$ is a smooth closed submanifold of\, $\mathbf{F}$ and $i\colon \sr N\to \sr M$ is a homotopy equivalence.
	\end{enumerate}
\end{lem}
\begin{proof} 
Part (a) is obvious. Part (b) follows from Theorem 15 in \cite{Palais}, Theorem 9 in \cite{BurKui} and cor.~3 in \cite{Henderson}. Part (c) is Theorem 2 in \cite{BurSalTom}. 
\end{proof}

\begin{lem}\label{L:net}
	Let $\E$ be a separable Hilbert space, $D\subs \E$ a dense vector subspace,  $L\subs \E$ a submanifold of finite codimension and $U$ an open subset of $L$.  If $K$ is a finite simplicial complex and $f\colon \abs{K}\to U$ a continuous map, then $f$ is homotopic within $U$ to a map $\abs{K}\to D\cap U$.
	\end{lem}
	\begin{proof}
		See \cite{SalZueh}, lemma 1.10.
	\end{proof}

\begin{cor}\label{L:dense}
	Let  $\ka_1<\ka_2$ and $P,Q\in UT\C$. Then the subset of all smooth curves in $\sr L_{\ka_1}^{\ka_2}(P,Q)$ is dense in the latter.
\end{cor}
\begin{proof}
	Take $\E=L^2[0,1]\times L^2[0,1]$, $D=C^\infty[0,1]\times C^\infty[0,1]$ and  $U$ an open subset of $L=\sr L_{\kappa_1}^{\kappa_2}(P,Q)$. Then it is a trivial consequence of \lref{L:net} that $D\cap U\neq \emptyset$ if $U\neq \emptyset$.
\end{proof}

\begin{lem}\label{L:CtoLinclusion}
	Let $(\ka_1,\ka_2)\subs (\bar \ka_1,\bar \ka_2)$ and $P,Q\in UT\C$. Then 
	\begin{equation}\label{E:injection}
		j\colon \sr C_{\kappa_1}^{\kappa_2}(P,Q)^r\to \sr L_{\bar\kappa_1}^{\bar\kappa_2}(P,Q),\quad \ga\mapsto (\hat\sig,\hat\ka),
	\end{equation}
	where $\hat\sig=h\circ \abs{\dot \ga}$ and $\hat\ka=h_{\bar\ka_1,\bar\ka_2}\circ \ka_\ga$, is a continuous injection for all $r\geq 2$. Furthermore, the actual curve in $\C$ corresponding to $j(\ga)\in \sr L_{\bar\kappa_1}^{\bar\kappa_2}(P,Q)$ is $\ga$ itself.
\end{lem}
\begin{proof}
	The curve corresponding to the right side of \eqref{E:injection} in $\sr L_{\bar\ka_1}^{\bar\ka_2}(P,\cdot)$ is the solution of \eqref{E:ivp} with
	\begin{equation*}
			\sig=h^{-1}\circ \hat \sig=\abs{\dot\ga} \text{\quad and\quad }\ka=h^{-1}_{\bar\ka_1,\,\bar\ka_2}\circ \hat\ka =\ka_\ga.
	\end{equation*}
	By uniqueness, this solution must equal $\ga$. In particular, $j$ is injective and its image is indeed contained in $\sr L_{\bar\kappa_1}^{\bar\kappa_2}(P,Q)$. Continuity of $j$ is clear: If $\eta$ is $C^r$-close to $\ga$, then $\sig_{\eta}$ (resp.~$\ka_{\eta}$) is $C^1$-close (resp.~$C^0$-close) to $\sig_\ga$ (resp.~$\ka_\ga$), hence $j(\eta)$ is close to $j(\ga)$ in the $L^2$-norm.
\end{proof}

\begin{cor}\label{L:smoothie}
	Let $\ka_1<\ka_2$, $P,Q\in UT\C$ and $\sr U\subs \sr L_{\ka_1}^{\ka_2}(P,Q)$ be open. Let $K$ be a finite simplicial complex and $f\colon \abs{K}\to \sr U$ be a continuous map. Then there exists a continuous $g\colon \abs{K}\to \sr U$ such that:
	\begin{enumerate}
		\item [(i)] $f\iso g$ within $\sr U$.
		\item [(ii)]  $g(a)$ is a smooth curve for all $a\in K$.
		\item [(iii)] All derivatives of $g(a)$ with respect to $t$ depend continuously on $a\in K$.
	\end{enumerate}
	In particular, the set inclusion $j\colon \sr C_{\ka_1}^{\ka_2}(P,Q)\inc \sr L_{\ka_1}^{\ka_2}(P,Q)$ in \eqref{E:injection} induces surjections $\pi_k(j^{-1}(\sr U))\to \pi_k(\sr U)$ for all $k\in \N$.
\end{cor}
\begin{proof}
	Parts (i) and (ii) follow immediately from \lref{L:net} by setting $\E=L^2[0,1]\times L^2[0,1]$, $D=C^\infty[0,1]\times C^\infty[0,1]$,  $L=\sr L_{\kappa_1}^{\kappa_2}(P,Q)$ and $U=\sr U$. The image of the function $g=H_2\colon \abs{K}\to \sr U$ constructed in the proof of \lref{L:net} (see \cite{SalZueh}, lemma 1.10) is contained in a finite-dimensional vector subspace of $D$, namely, the one generated by all $\te{v}_{ij}$, so (iii) also holds.
\end{proof}

\begin{lem}\label{L:C^2}
	Let $\ka_1<\ka_2$ and $P,Q\in UT\C$. Then the inclusion $j\colon \sr C_{\ka_1}^{\ka_2}(P,Q)^r\to \sr L_{\ka_1}^{\ka_2}(P,Q)$  of \eqref{E:injection} is a  homotopy equivalence for any $r\in \N$, $r\geq 2$. Consequently, $\sr C_{\kappa_1}^{\kappa_2}(P,Q)^r$ is homeomorphic to $\sr L_{\kappa_1}^{\kappa_2}(P,Q)$ for any $r\in \N$, $r\geq 2$.
\end{lem}
\begin{proof}Let $\E=L^2[0,1]\times L^2[0,1]$, let $\tbf{F}=C^{r-1}[0,1]\times C^{r-2}[0,1]$ (where $C^k[0,1]$ denotes the set of all $C^k$ functions $[0,1]\to \R$, with the $C^k$ norm) and let $i\colon \tbf{F}\to \E$ be set inclusion. Setting $\sr M=\sr L_{\ka_1}^{\ka_2}(P,Q)$, we conclude from \lref{L:Hilbert}\?(c) that $i\colon \sr N=i^{-1}(\sr M)\inc \sr M$ is a homotopy equivalence. We claim that $\sr N$ is homeomorphic to $\sr C_{\ka_1}^{\ka_2}(P,Q)^r$, where the homeomorphism $G$ is obtained by associating a pair $(\hat \sig,\hat \ka)\in \sr N$ to the curve $\ga$ obtained by solving \eqref{E:ivp}, with $\sig$ and $\ka$ as in \eqref{E:Sobolev}. The lemma will follow from this and the easily verified commutativity of 
\begin{equation*}
	\xymatrix{
		\sr N \ar[dr]_{i}\ar[r]^{G\qquad } & \sr C_{\kappa_1}^{\kappa_2}(P,Q)^r \ar[d]^{j}  \\
		 	&	\sr L_{\kappa_1}^{\kappa_2}(P,Q) 
	}
\end{equation*}
	
	Suppose first that $\ga\in \sr C_{\ka_1}^{\ka_2}(P,Q)^r$. Then $\abs{\dot\ga}$ (resp.~$\ka$) is a function $[0,1]\to \R$ of class $C^{r-1}$ (resp.~$C^{r-2}$). Hence, so are $\hat{\sig}=h\circ \abs{\dot\ga}$ and $\hat \ka=h_{\ka_1}^{\ka_2}\circ \ka$, since $h$ and $h_{\ka_1}^{\ka_2}$ are smooth. Moreover, if $\ga,\,\eta\in \sr C_{\ka_1}^{\ka_2}(P,Q)^r$ are close in $C^r$ topology, then $\hat \ka_\ga$ is $C^{r-2}$-close to $\hat\ka_\eta$ and $\hat \sig_\ga$ is $C^{r-1}$-close to $\hat \sig_\eta$.
	
	Conversely, if $(\hat{\sig},\hat{\ka})\in \sr N$, then $\sig=h^{-1}\circ \hat \sig$ is of class $C^{r-1}$ and $\ka=(h_{\ka_1}^{\ka_2})^{-1}\circ \hat \ka$ of class $C^{r-2}$. Since all functions on the right side of \eqref{E:ivp} are of class (at least) $C^{r-2}$, the solution $\ta=\ta_\ga$ to this initial value problem is of class $C^{r-1}$. Moreover, $\dot\ga=\sig\ta$, hence the velocity vector of $\ga$ is seen to be of class $C^{r-1}$. We conclude that $\ga$ is a curve of class $C^r$. Further, continuous dependence on the parameters of a differential equation shows that the correspondence $(\hat \sig,\hat \ka)\mapsto \ta_\ga$ is continuous. Since $\ga$ is obtained by integrating $\sig\ta_\ga$, we deduce that  the map $(\hat\sig,\hat\ka)\mapsto \ga$ is likewise continuous.
	
	The last assertion of the lemma follows from \lref{L:Hilbert}\?(b).
\end{proof}

 \begin{defn}
Let $P=(p,w),Q=(q,z)\in \C\times \Ss^1$. Given $\theta_1\in \R$ satisfying $e^{i\theta_1}=z\bar w$, we denote by $\sr L_{\kappa_1}^{\kappa_2}(P,Q;\theta_1)$ the subspace of $\sr L_{\kappa_1}^{\kappa_2}(P,Q)$ consisting of all curves which have total turning equal to $\theta_1$. When $P=(0,1)$, the space ${\sr L}_{\ka_1}^{\ka_2}(P,Q)$ (resp.~$\sr L_{\kappa_1}^{\kappa_2}(P,Q;\theta_1)$) will be denoted simply by ${\sr L}_{\ka_1}^{\ka_2}(Q)$ (resp.~$\sr L_{\kappa_1}^{\kappa_2}(Q;\theta_1)$).
\end{defn}

Notice that $(0,1)\in \C\times \Ss^1$ corresponds to the identity element in the group structure of $UT\C$. It will be proved in \S\ref{S:diffuse} that $\sr L_{\kappa_1}^{\kappa_2}(P,Q;\theta_1)$ is never empty if $\ka_1\ka_2<0$, but may be empty if $\ka_1\ka_2\geq 0$, depending on the value of $\theta_1$.

The next two results allow us to reparametrize a family of curves to better suit our needs.

\begin{lem}\label{L:reparametrize}
	Let $\sr M=\sr L_{\kappa_1}^{\kappa_2}(P,Q)$ or $\sr M=\sr C_{\kappa_1}^{\kappa_2}(P,Q)$. Let $A$ be a topological space and $A\to \sr M$, $a\mapsto \ga_a$, be a continuous map.  Then there exists a homotopy $\ga_a^r\colon [0,1] \to \sr M$, $r\in [0,1]$, such that for any $a\in A$:
	\begin{enumerate}
		\item [(i)] $\ga_a^0=\ga_a$ and $\ga_a^1$ is parametrized so that $\vert\dot{\ga}_a^1(t)\vert$ is independent of $t$.
		\item [(ii)]  $\ga_a^r$ is an orientation-preserving reparametrization of $\ga_a$, for all $r\in [0,1]$. 
	\end{enumerate} 
\end{lem}
\begin{proof}
	Let $s_a(t)=\int_0^t\abs{\dot\ga_a(\tau)}\,d\tau$ be the arc-length parameter of $\ga_a$, $L_a$ its length and $\tau_a\colon [0,L_a]\to [0,1]$ the inverse function of $s_a$. Define $\ga_a^r\colon [0,1]\to M$ by:
	\begin{equation*}
\qquad		\ga_a^r(t)=\ga_a\big((1-r)t+r\tau_a(L_at)\big) \qquad (r,t\in [0,1],~a\in A).
	\end{equation*}
Then $\ga_a^r$ is the desired homotopy.
\end{proof}

\begin{cor}\label{C:reparametrize}
	Let $\sr M=\sr L_{\kappa_1}^{\kappa_2}(P,Q)$ or $\sr C_{\kappa_1}^{\kappa_2}(P,Q)$. Let $A$ be a topological space and $f\colon \Ss^0\times A \to \sr M$ a continuous map such that, for all $a\in A$, $f(1,a)$ is an orientation-preserving reparametrization of $f(-1,a)$. Then $f$ admits a continuous extension $F\colon A\times [-1,1]\to \sr M$ with the property that $f(r,a)$ is an orientation-preserving reparametrization of $f(-1,a)$ for all $r\in [-1,1]$, $a\in A$.\qed
\end{cor}

It is assumed in \lref{L:reparametrize} and \cref{C:reparametrize} that the parametrizations of all curves have domain $[0,1]$, but it is clearly possible to have these intervals depend (continuously) on $a$. A typical application is to reparametrize all curves in a family by arc-length, not just proportionally to arc-length as in \lref{L:reparametrize}.

\subsection*{Spaces of curves with curvature in a closed interval}

\begin{defn}\label{D:main}
Let $P,Q\in UT\C$ and $-\infty< \ka_1<\ka_2< +\infty$. Define $\hat{\sr L}_{\ka_1}^{\ka_2}(P,Q)$ to be the set of all $C^1$ regular plane curves $\ga\colon [0,1]\to \C$ satisfying:
\begin{enumerate}
	\item [(i)] $\Phi_\ga(0)=P$ and $\Phi_\ga(1)=Q$;
	\item [(ii)] $\ka_1\leq \frac{\theta(s_1)-\theta(s_2)}{s_1-s_2}\leq \ka_2$ for any $s_1\neq s_2\in [0,L]$. (Here the parameter is the arc-length of $\ga$, $L$ is its length and $\theta\colon [0,L]\to \R$ an argument of $\ta_\ga$.)
\end{enumerate}
Condition (ii) implies that $\theta$ is a Lipschitz function. In particular, it is absolutely continuous, and its derivative $\ka_\ga$ lies in $L^2$, since it is bounded. We give this set the topology induced by the following distance function $d$: Given $\ga,\eta\in \hat{\sr L}_{\kappa_1}^{\kappa_2}(P,Q)$, set
\begin{equation*}
	d(\ga,\eta)=\norm{\ga-\eta}_2+\norm{\dot\ga-\dot\eta}_2+\norm{\ka_\ga-\ka_\eta}_2.
\end{equation*}
For $P=(0,1)\in \C\times \Ss^1$, we will denote $\hat{\sr L}_{\ka_1}^{\ka_2}(P,Q)$ simply by $\hat{\sr L}_{\ka_1}^{\ka_2}(Q)$.
\end{defn}

\begin{urem}
	 This definition is essentially due to L.\,E.\,Dubins, who studied paths of minimal length, now called \tdef{Dubins paths}, in $\hat{\sr L}_{-\ka_0}^{+\ka_0}(P,Q)$ ($\ka_0>0$). Such shortest paths always exist, but may not be unique in some special cases (see Proposition 1 and the corollary to Theorem I of \cite{Dubins}). His main result states that any Dubins path is the concatenation of at most three pieces, each of which is either a line segment or an arc of circle of radius $\frac{1}{\ka_0}$ (see Theorem I of \cite{Dubins} for the precise statement). Dubins paths and variations thereof have many applications in engineering and are the subject of a vast literature. The space $\hat{\sr L}_{\kappa_1}^{\kappa_2}(P,Q)$ will play a minor role in our investigations. Its topology has been chosen to ensure that the following result holds.
\end{urem}	

\begin{lem}\label{L:inclusions}
	Let $(\ka_1,\ka_2)\subs [\bar \ka_1,\bar \ka_2]\subs (\bar{\bar \ka}_1,\bar{\bar \ka}_2)$ and $P,Q\in UT\C$. Then the set inclusions $\sr C_{\kappa_1}^{\kappa_2}(P,Q) \to \hat{\sr L}_{\bar \kappa_1}^{\bar \kappa_2}(P,Q)$  and $\hat{\sr L}_{\bar \kappa_1}^{\bar \kappa_2}(P,Q)\to \sr L_{\bar{\bar{\kappa}}_1}^{\bar{\bar{\kappa}}_2}(P,Q)$ are continuous injections.
\end{lem}
\begin{proof}
	The proof is a straightforward verification, which will be left to the reader.
\end{proof}


\section{Normal translation}\label{S:normal}The \tdef{radius of curvature} $\rho$ of an admissible curve $\ga$ is given by $\rho=\frac{1}{\ka}$; when $\ka(t)=0$, it is to be understood that $\rho(t)=\infty$ (unsigned infinity). An analogue of the following construction has already appeared in \cite{SalZueh}. It can be used to uniformly shift the radii of curvature of a family of curves.

\begin{defn}
Let $\ga\colon [0,1]\to \C$ be admissible and $u\in \R$. The \tdef{normal translation} $\ga_u$ of $\ga$ by $u$ is the curve given by
\begin{equation*}
	\ga_u(t)=\ga(t)+u\no(t)\quad (t\in [0,1]).
\end{equation*}
\end{defn}

Observe that the normal translation $\al_u$ of a circle $\al$ of radius of curvature $\rho\in \R\ssm \se{0}$ is a circle of radius of curvature $\rho-u$ for any $u$ in the component of $\R\ssm \se{\rho}$ containing $0$ (see Figure \ref{F:translation}). The following lemma generalizes this to arbitrary curves.


\begin{lem}\label{L:normal}
	Let $\ga\in \sr L_{\kappa_1}^{\kappa_2}(P,Q)$ be parametrized proportionally to arc-length and let $\ta,\ka,\rho$ denote its unit tangent, curvature and radius of curvature, respectively. Suppose  $u\in \R$ satisfies $1-uk>0$ for all $k\in (\ka_1,\ka_2)$ and set
	\begin{equation}\label{E:barka}
		\bar\ka_i=\frac{\ka_i}{1-u\ka_i}\qquad (i=1,2).
	\end{equation}
	Then the normal translation $\ga_u$ of $\ga$ by $u$ has the following properties: 
	\begin{enumerate}
		\item [(a)] $\ga_u\in \sr L_{\bar\ka_1}^{\bar\ka_2}(\bar P,\bar Q)$ for $\bar P=(p+iuw,w)$, $\bar Q=(q+iuz,z)$ and its unit tangent $\bar\ta$ satisfies $\bar \ta(t)=\ta(t)$ for each $t\in [0,1]$. In particular, $\ga$ and $\ga_u$ have the same total turning.
		\item [(b)]  $(\ga_u)_{-u}=\ga$.
		\item [(c)] If $\eta$ is a reparametrization of $\ga$, then $\eta_u$ is a reparametrization of $\ga_u$.
		\item [(d)]  For almost every $t\in [0,1]$, the curvature $\bar\ka$ of $\ga_u$ is given by:
		\begin{equation*}
			\bar\ka(t)=\frac{\ka(t)}{1-u\ka(t)}
		\end{equation*}
		and its radius of curvature $\bar\rho$ by:
		\begin{equation*}
			\bar\rho(t)=\rho(t)-u.
		\end{equation*}
	\end{enumerate}
\end{lem}
	In eq.~\eqref{E:barka} above, it should be understood that $\bar\ka_i=-\frac{1}{u}$ if $\ka_i$ is infinite and that $\bar\ka_i=\pm \infty$ has the same sign as $\ka_i$ if $1-u\ka_i=0$.
	
\begin{figure}[ht]
	\begin{center}
		\includegraphics[scale=.40]{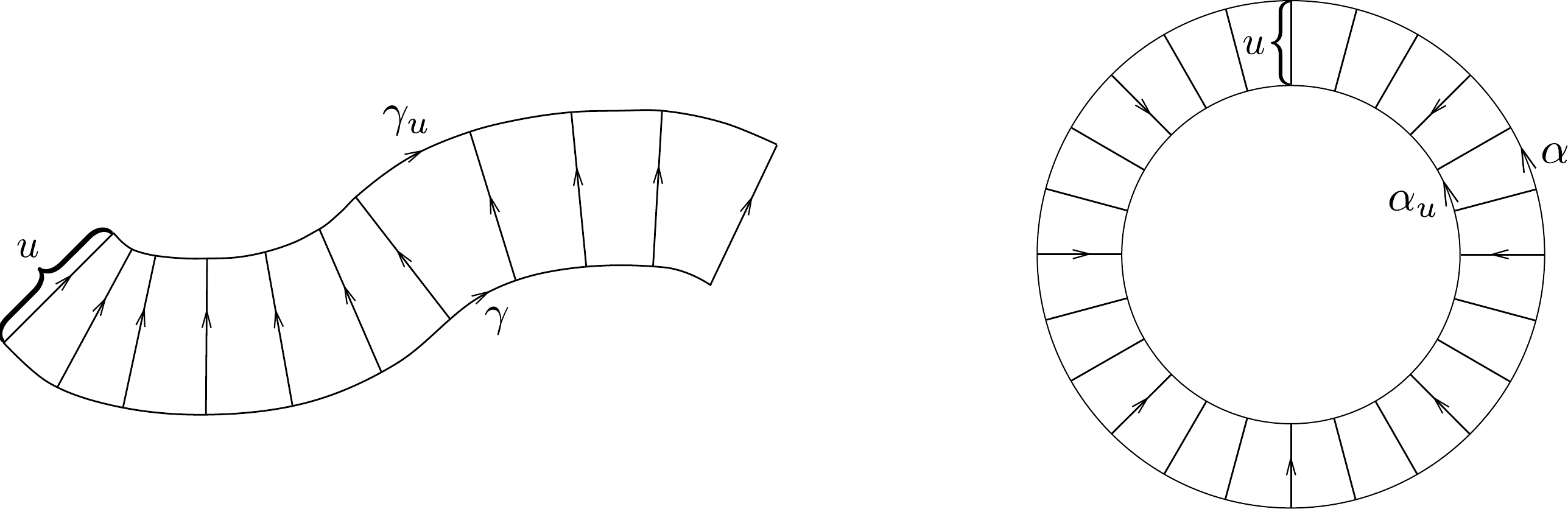}
		\caption{The normal translation of a general curve $\ga$ and of a circle $\al$.}
		\label{F:translation}
	\end{center}
\end{figure}
	\begin{proof}
	Let $\theta_\ga\colon [0,1]\to \R$ be an argument of $\ta=\ta_\ga$ and define  $\Psi\colon [0,1]\to \GL$ by 
	\begin{equation}\label{E:third}
		\Psi=\begin{pmatrix}
		\cos\theta_\ga & -\sin\theta_\ga & \ga_1-u\sin\theta_\ga \\
		\sin\theta_\ga & \cos\theta_\ga & \ga_2+u\cos\theta_\ga \\
		0 & 0 & 1
		\end{pmatrix}.
	\end{equation}
	Let $L$ be the length of $\ga$. Since $\ga$ is parametrized proportionally to arc-length,  a straightforward calculation shows that $\Psi$ satisfies $\dot\Psi=\Psi\La$, for 
	\begin{equation}\label{E:lambdao}
		\La\colon [0,1]\to \mathfrak{a}\subs \gl,\quad
		\La=L\begin{pmatrix}
		0 & -\ka & 1-u\ka \\
		\ka & 0 & 0 \\
		0 & 0 & 0
		\end{pmatrix}.
	\end{equation}
	By hypothesis, the image of $\La$ is contained in the half-plane $\mathfrak{h}$ of \eqref{E:hp}. Comparing the third column of \eqref{E:third} with the definition of $\ga_u$, we deduce that $\Psi$ is the frame of $\ga_u$. Further, looking at the first and second columns we deduce that $\bar\ta=\ta$ and $\bar\no=\no$. That $\Phi_{\ga_u}(0)=\bar P$ and $\Phi_{\ga_u}(1)=\bar Q$ then follows immediately from the definition. This establishes (a) except for the fact that $\ga_u$ is $(\bar\ka_1,\bar\ka_2)$-admissible, which will be proved below. 
	
	Part (b) is an easy verification:
		\begin{equation*}
			(\ga_u)_{-u}=\ga_u-u\bar\no=(\ga+u\no)-u\no=\ga.
		\end{equation*}
	 
	 Part (c) is obvious. 
	 
	 We know that the curvature $\bar\ka$ of $\ga_u$ is given by the quotient of $\La^{(2,1)}$ by $\La^{(1,3)}$, that is, 
	\begin{equation*}
		\bar\ka=\frac{\ka}{1-u\ka}=\frac{1}{\rho-u}=\frac{1}{\bar\rho}.
	\end{equation*}
	This proves (d). 

	It is straightforward to check that $u\in \R$ satisfies $1-uk>0$ for all $k\in (\ka_1,\ka_2)$ if and only if $u$ lies in the maximal closed interval $J$ containing $0$ and not containing any number of the form $\frac{1}{k}$ for $k\in (\ka_1,\ka_2)$. More explicitly:
	\begin{enumerate}
		\item [(i)] If $0\leq \ka_1<\ka_2$ then $J=(-\infty,\rho_2]$;
		\item [(ii)]  If $\ka_1<0<\ka_2$ then $J=[\rho_1,\rho_2]$;
		\item [(iii)] If $\ka_1<\ka_2\leq 0$ then $J=[\rho_1,+\infty)$.
	\end{enumerate}
	By \eqref{E:lambdao}, $\abs{\dot\ga_u}=L(1-u\ka)$. To establish that $\ga_u$ is $(\bar\ka_1,\bar\ka_2)$-admissible, it suffices to prove that:
	\begin{alignat}{10}\label{E:needs}
		&h_{0,+\infty}\circ (1-u\ka)=(1-u\ka)-(1-u\ka)^{-1}\in L^2[0,1]\qquad\text{and}\\\label{E:needs2}
		&h_{\bar\ka_1,\bar\ka_2}\circ \bar\ka\in L^2[0,1]. 
	\end{alignat}
	By \rref{R:preservesL2ness}, $\ka\in L^2[0,1]$, hence so does $(1-u\ka)$. Moreover, $(1-u\ka)^{-1}$ is bounded unless $u$ is one of the endpoints of $J$, but we claim that even in this case $(1-u\ka)^{-1}\in L^2[0,1]$. Suppose for concreteness that $u=\rho_2\in \bd J$. If $\ka_2=+\infty$ ($\rho_2=0$) then there is nothing to prove, and otherwise
	\begin{equation}\label{E:lies}
		(1-u\ka)^{-1}=(1-\rho_2\ka)^{-1}=\ka_2(\ka_2-\ka)^{-1}.
	\end{equation}
	Now by hypothesis $\ga\in \sr L_{\kappa_1}^{\kappa_2}(P,Q)$, therefore
	\begin{equation*}
		h_{\ka_1,\ka_2}\circ \ka=(\ka_1-\ka)^{-1}+(\ka_2-\ka)^{-1}\in L^2[0,1].
	\end{equation*}
	This implies that both
	\begin{equation}\label{E:integ}
		(\ka_1-\ka)^{-1}\in L^2[0,1]\ \ \text{and}\ \ (\ka_2-\ka)^{-1}\in L^2[0,1],
	\end{equation}  since as one of them increases in absolute value, the other one decreases. Consequently, \eqref{E:lies} lies in $L^2[0,1]$ and \eqref{E:needs} follows from Minkowski's inequality.

	The proof of \eqref{E:needs2} involves the tedious consideration of several cases, because it depends on which of the four $h$ functions defined on p.~\pageref{E:hfunctions} is used, both for $(\ka_1,\ka_2)$ and $(\bar\ka_1,\bar\ka_2)$. Assume first that $\ka_1<\ka_2$ are both finite. If $u\nin \bd J$, then $\bar\ka_1,\bar\ka_2$ are also finite, so
	\begin{alignat*}{9} 
		h_{\bar\ka_1,\bar\ka_2}\circ \bar\ka=(1-u\ka)(1-u\ka_1)(\ka_1-\ka)^{-1}+(1-u\ka)(1-u\ka_2)(\ka_2-\ka)^{-1}.
	\end{alignat*}
	Since $\ka\in (\ka_1,\ka_2)$ is bounded, this is a sum of two functions in $L^2[0,1]$ by \eqref{E:integ}, hence it lies in $L^2[0,1]$. If $u$ is an endpoint $\rho_i$ of $J$ then $\bar\ka_i$ is infinite. For instance, if $u=\rho_2$ then
	\begin{alignat*}{9} 
		h_{\bar\ka_1,\bar\ka_2}\circ \bar\ka=h_{\bar\ka_1,+\infty}\circ \bar\ka=(1-\rho_2\ka)(1-\rho_2\ka_1)(\ka_1-\ka)^{-1}+\ka_2\ka(\ka_2-\ka)^{-1}.
	\end{alignat*}
	Because $\ka$ is bounded, we conclude from \eqref{E:integ} that \eqref{E:needs2} holds in this case also.
	
	If one of the $\ka_i$, say $\ka_2$, is infinite, then the hypothesis that $\ga$ is admissible implies that
	\begin{equation*}
		h_{\ka_1,+\infty}\circ \ka=(\ka_1-\ka)^{-1}+\ka\in L^2[0,1].
	\end{equation*}
	As above, it follows that each of the summands lies in $L^2[0,1]$. If $u\neq \rho_1$ then $\bar\ka_1=\frac{\ka_1}{1-u\ka_1}$, $\bar\ka_2=-\frac{1}{u}$ are both finite, and
	\begin{equation*}
		h_{\bar\ka_1,\bar\ka_2}\circ \bar\ka=(1-u\ka_1)(1-u\ka)(\ka_1-\ka)^{-1}-u(1-u\ka).
	\end{equation*}
	Observe that $(1-u\ka)(\ka_1-\ka)^{-1}\in L^2[0,1]$ because as $\ka$ increases to $+\infty$, $(\ka_1-\ka)^{-1}$ remains bounded, while as $\ka\to \ka_1$, obviously $(1-u\ka)$ remains bounded. Thus, \eqref{E:needs2} holds. We leave the similar verification in the remaining cases to the reader.
\end{proof}

\begin{rem}\label{R:badspeed}
	The necessity of reparametrizing an admissible curve by arc-length before applying normal translation stems from the fact that the product of two $L^2$ functions need not be of class $L^2$: For a general parametrization the speed of $\ga_u$ is given by $\sig(1-u\ka)$, where $\sig$, $\ka$ are the speed and curvature of $\ga$. Hence, $\ga_u$ need not be admissible. This has no serious consequences because of \lref{L:reparametrize}.
\end{rem}

The next result greatly simplifies the study of the spaces $\sr L_{\kappa_1}^{\kappa_2}(P,Q)$. In all that follows the notation $X\home Y$ means that $X$ is homeomorphic to $Y$.

\begin{thm}\label{T:normalized}
	Let $P=(p,w),~Q=(q,z)\in \C\times \Ss^1$, $-\infty\leq \ka_1<\ka_2\leq +\infty$ and $\rho_i=\frac{1}{\ka_i}$.
	\begin{enumerate}
		\item [(a)] Suppose $\ka_1<0<\ka_2$. If at least one of $\ka_1,\ka_2$ is finite, then $\sr L_{\ka_1}^{\ka_2}(P,Q)\home \sr L_{-1}^{+1}(Q_1)$ for
		\begin{alignat*}{9}
			Q_1&=\Big(\tfrac{2}{\rho_2-\rho_1}\bar w\big[(q-p)+\tfrac{i}{2}(\rho_1+\rho_2)(z-w)\big]\,,\,z\bar w \Big).
		\end{alignat*}
		\item [(b)] Suppose $0<\ka_1<\ka_2$. Then $\sr L_{\kappa_1}^{\kappa_2}(P,Q)\home \sr L_{1}^{+\infty}(Q_2)$ for 
		\begin{equation*}
			Q_2=\Big(\tfrac{\bar w}{\rho_1-\rho_2}\big[(q-p)+i\rho_2(z-w)\big]\,,\,z\bar w \Big).
		\end{equation*}
		\item [(c)] Suppose $0=\ka_1<\ka_2$. Then $\sr L_{\kappa_1}^{\kappa_2}(P,Q)\home \sr L_{0}^{+\infty}(Q_3)$ for 
		\begin{equation*}
			Q_3=\big(\?\bar w\?[(q-p)+i\rho_2(z-w)]\,,\,z\bar w \?\big).
		\end{equation*}
		\item [(d)] Suppose $\ka_1<\ka_2<0$. Then $\sr L_{\kappa_1}^{\kappa_2}(P,Q)\home \sr L_{1}^{+\infty}(Q_4)$ for 
		\begin{equation*}
			Q_4=\Big(\tfrac{\bar z}{\rho_1-\rho_2}\big[(q-p)+i\rho_1(z-w)\big]\,,\,w \bar z \Big).
		\end{equation*}
		\item [(e)] Suppose $\ka_1<\ka_2=0$. Then $\sr L_{\kappa_1}^{\kappa_2}(P,Q)\home \sr L_{0}^{+\infty}(Q_5)$ for 
		\begin{equation*}
			Q_5=\big(\?\bar z\?[(q-p)+i\rho_1(z-w)]\,,\,w\bar z \?\big).
		\end{equation*}
	\end{enumerate}
	In cases \tup(a\tup)--\tup(c\tup) \tup(resp.~\tup(d\tup)--\tup(e\tup)\tup), the total turning of the image of a curve under the homeomorphism is equal \tup(resp.~opposite\tup) to that of the original curve.
\end{thm}


\begin{proof} Suppose first that $\ka_1<0<\ka_2$ and let $k\in (\ka_1,\ka_2)$ be arbitrary. If $\rho_1+\rho_2\leq 0$, then 
	\begin{equation*}
		1-\big(\tfrac{\rho_1+\rho_2}{2}\big)k> 1-\big(\tfrac{\rho_1+\rho_2}{2}\big)\ka_1=\tfrac{1}{2}\big(1-\rho_2\ka_1\big)\geq \tfrac{1}{2}>0
	\end{equation*}
	and if $\rho_1+\rho_2\geq 0$, then
	\begin{equation*}
		1-\big(\tfrac{\rho_1+\rho_2}{2}\big)k> 1-\big(\tfrac{\rho_1+\rho_2}{2}\big)\ka_2=\tfrac{1}{2}\big(1-\rho_1\ka_2\big)\geq \tfrac{1}{2}>0.
	\end{equation*}
	Consequently, $u=\frac{\rho_1+\rho_2}{2}$ satisfies the hypothesis of \lref{L:normal}. Let 
	\begin{equation*}
		\ka_0=\frac{2}{\rho_2-\rho_1}.
	\end{equation*}
Note that $0<\ka_0<+\infty$; in the notation of \lref{L:normal}, $-\ka_0=\bar\ka_1$ and $\ka_0=\bar\ka_2$. Define a map $F\colon \sr L_{\kappa_1}^{\kappa_2}(P,Q)\to \sr L_{-\ka_0}^{+\ka_0}(\bar P,\bar Q)$ by letting $F(\ga)$ be the translation by $u$ of its reparametrization (still with domain $[0,1]$) by a multiple of arc-length. This is continuous by \lref{L:reparametrize}. In fact, it is a homotopy equivalence: There is a similarly defined map $G\colon \sr L_{-\ka_0}^{+\ka_0}(\bar P,\bar Q)\to \sr L_{\kappa_1}^{\kappa_2}(P,Q)$ using translation by $-u$, and $GF(\ga)$ is just a reparametrization of $\ga$, by \lref{L:normal}\?(b) and \lref{L:normal}\?(c).

	Let $T\colon \C \to \C$ be the dilatation $x \mapsto \ka_0 x$. If $\ga\in \sr L_{-\ka_0}^{+\ka_0}(\bar P,\bar Q)$, then $T\circ \ga$ lies in 
	\begin{equation*}
		\sr L_{-1}^{+1}(\te{P},\te{Q}),\ \ \text{where}\ \ \te{P}=\big(\ka_0(p+\tfrac{\rho_1+\rho_2}{2} iw),w\big),~\te{Q}=\big(\ka_0(q+\tfrac{\rho_1+\rho_2}{2}i z),z\big)
	\end{equation*}
	and the correspondence $\ga\mapsto T\circ \ga$ yields a homeomorphism between these two spaces. Write $\te{P}=(\te p,w)\in \C\times \Ss^1$ and let $E\colon \C\to \C$ be the Euclidean motion given by $E(x)=\bar w(x-\te p)$. Then the map $\ga\mapsto E\circ \ga$ is a homeomorphism from $\sr L_{-1}^{+1}(\te{P},\te{Q})$ onto $\sr L_{-1}^{+1}(Q_1)$, with $Q_1$ as in the statement. The composition of all of these maps yields a homotopy equivalence $\sr L_{\kappa_1}^{\kappa_2}(P,Q)\to \sr L_{-1}^{+1}(Q_1)$, which is homotopic to a homeomorphism by \lref{L:Hilbert}\?(b).
	
	The proofs of parts (b) and (c) are analogous, so only a brief outline will be provided. In part (b) we first use normal translation by $\rho_2$, and then compose with the dilatation $x\mapsto \frac{x}{\rho_1-\rho_2}$ and an Euclidean motion; in part (c) the dilatation is not necessary. Parts (d) and (e) follow from (b) and (c), respectively, by reversing the orientation of all curves in the corresponding space. 
	
	By \lref{L:normal}\?(a), the normal translations used in establishing (a)--(c) preserve the total turning of a curve. Clearly, so do dilatations and Euclidean motions, while a reversal of orientation changes the sign of the total turning. This proves the last assertion of the theorem.
\end{proof}

\begin{rem}\label{R:respect}
	Normal translations, and hence also the homotopy equivalences constructed in \tref{T:normalized}, do not generally respect inequalities between lengths. This is clear from Figure \ref{F:translation}: Two circles of the same radius $r>0$ but different orientations are mapped to circles of radii equal to $r\pm u$ under normal translation by $u\in (0,r)$. See also the remarks at the end of \S\ref{S:final}.
\end{rem}
A more concise version of \tref{T:normalized} is the following; recall that $0(\pm \infty)=0$ by convention.
\begin{cor}\label{C:normalized}
	Let $P,\,Q\in UT\C$. Then $\sr L_{\kappa_1}^{\kappa_2}(P,Q)$ is homeomorphic to a space of type $\sr L_{-1}^{+1}(Q_0)$, $\sr L_{0}^{+\infty}(Q_0)$ or $\sr L_{1}^{+\infty}(Q_0)$, according as $\ka_1\ka_2<0$, $\ka_1\ka_2=0$ or $\ka_1\ka_2>0$, respectively.\qed
\end{cor}

Out of the three possibilities, the spaces of type $\sr L_{\ka_1}^{\ka_2}(P,Q)$ with $\ka_1\ka_2<0$ are the ones with the most interesting topological properties. We deal with the two remaining cases in \S\ref{S:locallyconvex}.  

\begin{rem}\label{R:obvious}
	We may replace $\sr L$ with $\sr C$ throughout in the statement of \tref{T:normalized}. In fact, the difficulty indicated in \rref{R:badspeed} disappears in this case, so the proof is simpler because it is not necessary to reparametrize the curves by arc-length before applying normal translation. This yields explicit homeomorphisms of the corresponding spaces, without relying on \lref{L:Hilbert}\?(b). Because the curves in a space of type $\hat{\sr L}_{\ka_1}^{\ka_2}$ are $C^1$ regular by definition, this simpler proof also works for this class (except that here $\ka_1<\ka_2$ must be finite); see \tref{T:normalizedbar} for the precise statement.

\end{rem}  

\section{Topology of $\sr U_c$}\label{S:condensed}

\begin{defn}
	Let $\ga\colon [0,1]\to \C$ be a regular curve and $\theta\colon [0,1]\to \R$ be an argument of $\ta_\ga$. The \tdef{amplitude} of $\ga$ is given by
	\begin{equation*}
		\om=\sup_{t\in [0,1]}\theta(t)-\inf_{t\in [0,1]}\theta(t).
	\end{equation*}
	We call $\ga$ \tdef{condensed}, \tdef{critical} or \tdef{diffuse} according as $\om<\pi$, $\om=\pi$ or  $\om>\pi$. 
\end{defn}

Our main objective now is to understand the topology of $\sr L_{\kappa_1}^{\kappa_2}(P,Q)$ when $\ka_1\ka_2<0$. By \cref{C:normalized}, no generality is lost in assuming that $\ka_1=-1$, $\ka_2=+1$ and $P=(0,1)\in \C\times \Ss^1$

\begin{defn}
Let $Q=(q,z)\in \C\times \Ss^1$ and $\theta_1\in \R$ satisfy $e^{i\theta_1}=z$. We denote by $\sr U_c$, $\sr U_d$ and $\sr T$ the subspaces of $\sr L_{-1}^{+1}(Q;\theta_1)$ consisting of all condensed, diffuse and critical curves, respectively.
\end{defn}


\begin{thm}\label{T:condensed}
	The subspace $\sr U_c\subs \sr L_{-1}^{+1}(Q;\theta_1)$ consisting of all condensed curves is either empty or homeomorphic to $\E$, hence contractible.
\end{thm}

Recall that $\E$ denotes the separable Hilbert space. In what follows a function $\phi$ of a real variable will be called \tdef{increasing} (resp.~\tdef{decreasing}) if $x<y$ (resp.~$x>y$) implies that $\phi(x)\leq \phi(y)$. The previous theorem will be derived as a corollary of the following result.

\begin{prop}\label{P:excavator}
	Let $\ka_0>0$ and $\hat{\sr U}_c\subs \hat{\sr L}_{-\ka_0}^{+\ka_0}(Q;\theta_1)$ be the subspace consisting of all condensed curves. If $\hat{\sr U}_c\neq \emptyset$, then there exists a continuous $H\colon [0,1]\times \hat{\sr U}_c\to \hat{\sr U}_c$ such that for all $\ga\in \hat{\sr U}_c$:
	\begin{enumerate}
		\item [(i)] $H(1,\ga)=\ga$ and $H(0,\ga)=\ga_0$ is independent of $\ga$. 
		\item [(ii)]  The amplitude of $\ga_s=H(s,\ga)$ is an increasing function of $s\in [0,1]$.
		\item [(iii)] The length of $\ga_s=H(s,\ga)$ is an increasing function of $s\in [0,1]$.
	\end{enumerate}
	In particular, $\hat{\sr U}_c$ is contractible. Moreover, $\ga_0$ is the unique curve of minimal length in $\hat{\sr L}_{-\ka_0}^{+\ka_0}(Q)$.
\end{prop}

We believe that this proposition and its proof may be useful for other purposes which are not pursued here, e.g., for calculating the minimal length of curves in $\hat{\sr L}_{-\ka_0}^{+\ka_0}(Q)$. We shall first describe the effect of $H$ on a single curve $\ga\in  \hat{\sr U}_c$ and then derive its main properties separately as lemmas. First we record two results which will be used to show that $H(0,\ga)$ is independent of $\ga$.

\begin{figure}[ht]
	\begin{center}
		\includegraphics[scale=.335]{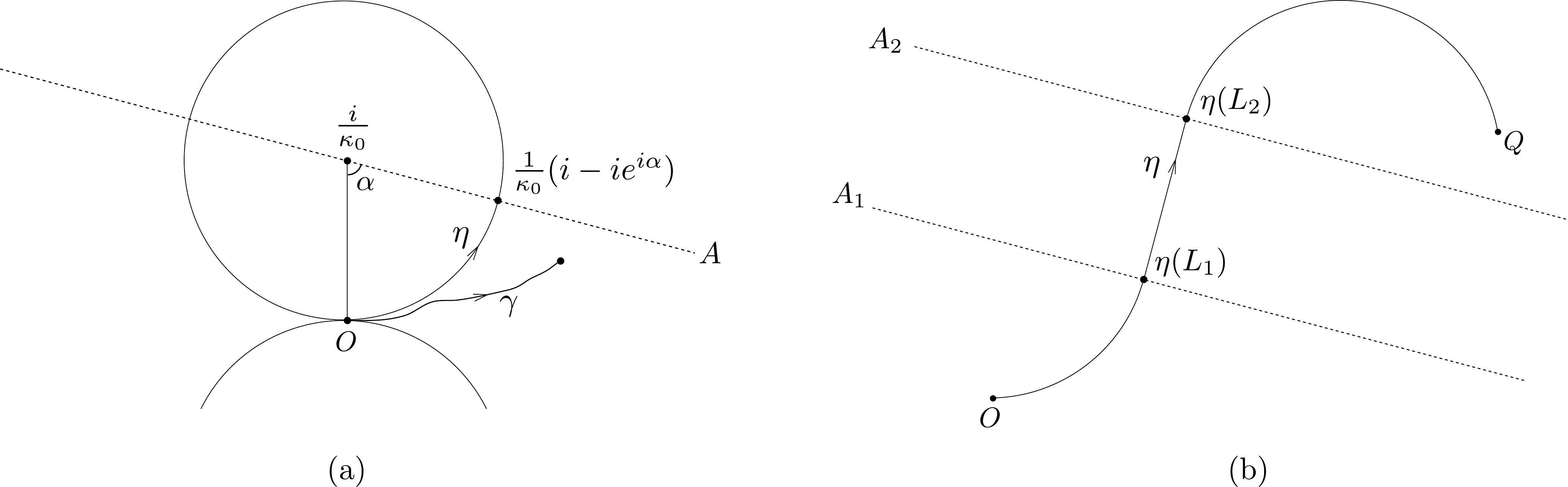}
		\caption{An illustration of \lref{L:faster} and \cref{C:Dubinspath}.}
		\label{F:faster}
	\end{center}
\end{figure}

\begin{lem}\label{L:faster}
	Let $Q=(q,z)\in \C\times \Ss^1$, $\ga\in \hat{\sr L}_{-\ka_0}^{+\ka_0}(Q)$ and $L$ be the length of $\ga$. Suppose that $q$ lies on the line through $\frac{i}{\ka_0}$ having direction $-ie^{i\alpha}$, for some $\alpha\in [0,\pi)$. Then $L\geq \frac{\alpha}{\ka_0}$ and equality holds if and only if $\ga$ is a reparametrization of the arc of the circle centered at $\frac{i}{\ka_0}$ joining $0$ to $\frac{1}{\ka_0}(i-ie^{i\alpha})$.
\end{lem}

\begin{proof}
	We lose no generality in assuming that $\ka_0=1$. If $\al=0$, there is nothing to prove, so suppose $\al\in (0,\pi)$. Let $\ga\colon [0,L]\to \C$ be parametrized by arc-length, and let $\eta\colon [0,\alpha]\to \C$ be given by
	\begin{equation*}
		\eta(s)=\int_0^se^{i\sig}\,d\sig=i-ie^{is}\quad (s\in [0,\alpha]),
	\end{equation*}
	so that $\eta$ is the parametrization by arc-length of the arc of circle described in \lref{L:faster}, see Figure \ref{F:faster}\?(a). Set
	\begin{alignat*}{10}
		f\colon [0,L]&\to \R,&\quad f(s)&=\gen{\ga(s)-i,e^{i\alpha}}, \\
		g\colon [0,\alpha]&\to \R,&\quad g(s)&=\gen{\eta(s)-i,e^{i\alpha}}.
	\end{alignat*}
	Let $A$ denote the line in the statement. Note that $f(s)=0$ if and only if $\ga(s)\in A$. We need to prove that $f(s)<0$ for all $s\in [0,\alpha)\cap [0,L]$. Let $\theta_\ga$ be the argument of $\ta_\ga$ satisfying $\theta_\ga(0)=0$. Then
	\begin{equation}\label{E:faster}
		\begin{alignedat}{10}
		f'(s)&=\big\langle e^{i\theta_\ga(s)},e^{i\alpha}\big\rangle=\cos(\al-\theta_\ga(s))\text{\ \ and\ \ }g'(s)&=\gen{e^{is},e^{i\alpha}}=\cos(\al-s)
		,
	\end{alignedat}
	\end{equation}
	
	We have $f(0)=g(0)$. Since $g(s)<0$ for all $s\in [0,\alpha)$, it suffices to establish that $f'(s)\leq g'(s)$ for all $s\in [0,\alpha]\cap [0,L]$. By the definition of $\hat{\sr L}_{-1}^{+1}(Q)$, $\theta_\ga$ is 1-Lipschitz. Hence, $\abs{\theta_\ga(s)}\leq s$ for all $s\in [0,L]$. Consequently,
	\begin{equation*}
		\alpha-s\leq \alpha-\theta_\ga(s)\leq \alpha+s\text{\ \ for all $s\in [0,L]$}.
	\end{equation*}
	In particular, $\alpha-\theta_\ga(s)\in [0,2\pi]$ for all $s\in [0,\alpha]\cap [0,L]$. Since the cosine is decreasing over $[0,\pi]$, it follows immediately from \eqref{E:faster} that if $\alpha-\theta_\ga(s)\leq \pi$, then $f'(s)\leq g'(s)$. On the other hand, if  $\alpha-\theta_\ga(s)\in [\pi,2\pi]$, then from $\alpha-\theta_\ga(s)\leq \alpha+s$, we obtain that
	\begin{equation*}
		\cos(\alpha-\theta_\ga(s))\leq \cos(\alpha+s)\leq \cos(\alpha-s),
	\end{equation*}
	the latter inequality coming from $\alpha\in (0,\pi)$ and $s\in [0,\alpha]$. Thus, $f'(s)\leq g'(s)$ in this case also. We conclude that $f(s)\leq g(s)<0$ for all $s\in [0,\alpha)\cap [0,L]$. In particular, $L\geq \al$, as $\ga(L)\in A$. 
	
	If $f(\al)=g(\al)=0$, then we must have $f'=g'$, that is, $\theta_\ga(s)=s$ for all $s\in [0,\alpha]$. Thus, in this case, $\ga|_{[0,\al]}$ is a reparametrization of $\eta|_{[0,\al]}$.
\end{proof}

\begin{cor}\label{C:Dubinspath}
	Suppose that $\eta\in \hat{\sr L}_{-\ka_0}^{+\ka_0}(Q)$ is a concatenation of an arc of circle of curvature $\pm \ka_0$, a line segment, and another arc of circle of curvature $\pm \ka_0$, where some of these may be degenerate and both arcs have length less than $\frac{\pi}{\ka_0}$. Then $\eta$ is the unique curve in $\hat{\sr L}_{-\ka_0}^{+\ka_0}(Q)$ of minimal length.
\end{cor}
This result should be compared to Proposition 9 in \cite{Dubins}. Their proofs are essentially the same.
\begin{proof}
	Let $\eta\colon [0,L]\to \C$ be parametrized by arc-length, with $\eta|_{[0,L_1]}$, $\eta|_{[L_1,L_2]}$ and $\eta|_{[L_2,L]}$ corresponding to the first arc, line segment and second arc, respectively (see Figure \ref{F:faster}\?(b)). Let $A_i$ be the line perpendicular to $\eta'(L_i)$ passing through $\eta(L_i)$, $i=1,2$. Notice that $A_1$ and $A_2$ are parallel (or equal). Suppose that $\ga\colon [0,M]\to \C$ is another curve in $\hat{\sr L}_{-\ka_0}^{+\ka_0}(Q)$, parametrized by arc-length. Let 
	\begin{equation*}
		M_1=\inf\set{s\in [0,M]}{\ga(s)\in A_1},\quad M_2=\sup\set{s\in [0,M]}{\ga(s)\in A_2}.
	\end{equation*}
	By \lref{L:faster}, we have $M_1\geq L_1$ and $M-M_2\geq L-L_2$. It is clear that $M_2-M_1\geq L_2-L_1$ since any path joining a point of $A_1$ to a point of $A_2$ must have length greater than or equal to the distance between these lines. Hence, $M\geq L$.  Furthermore, if equality holds, then $M_1=L_1$, $M-M_2=L-L_2$ and $M_2-M_1=L_2-L_1$. By \lref{L:faster}, the two former equalities imply that $\ga|_{[0,M_1]}=\eta|_{[0,L_1]}$ and  $\ga|_{[M_2,M]}=\eta|_{[L_2,L]}$. The condition $M_2-M_1=L_2-L_1$ then implies that $\ga|_{[M_1,M_2]}$ must coincide with the line segment $\eta|_{[L_1,L_2]}$.
\end{proof}

\begin{rem}\label{R:condembed}
	Notice that a condensed curve must be an embedding of $[0,1$]. In fact, its image is the graph of a function of $x$, after a suitable choice of the $x$-axis. 	
\end{rem}

	\begin{cons}\label{C:excavator}
		Let $\ga\in \hat{\sr U}_c$, $\theta\colon [0,1]\to \R$ be the argument of $\ta_\ga$ satisfying $\theta(0)=0$. A number $\vphi\in (-\frac{\pi}{2},\frac{\pi}{2})$ will be called an \tdef{axis} of $\ga$ if $\gen{\ta_\ga(t),e^{i\vphi}}>0$ for all $t\in [0,1]$. Since $\ga$ is condensed, the set of all axes of $\ga$ is an open interval. The most natural axis, and the center of this interval, is
	 \begin{equation}\label{E:thetabar}
		\bar\vphi_\ga=\frac{1}{2}\Big(\sup_{t\in [0,1]}\theta(t)+\inf_{t\in [0,1]}\theta(t)\Big).
	 \end{equation}
	Let $\vphi$ be any axis of $\ga$. Rotating around the origin through $\vphi$ and  writing $\ga(t)=(x(t),y(t))$ in terms of the new $x$- and $y$-axes, the hypothesis that $\gen{\ta_\ga,e^{i\vphi}}>0$ becomes equivalent to the fact that $\dot x$ is bounded and positive over $[0,1]$. Let 
	\begin{equation*}
		\ga(x)=(x,y(x))\quad (x\in [0,b])
	\end{equation*}
	be the reparametrization of $\ga$ by $x$ and define
	\begin{equation*}
		f\colon [0,b]\to \R \quad\text{by}\quad f(x)=\dot y(x).
	\end{equation*}
	Let $f_s\colon [0,b]\to \R$ $(s\in [0,1])$ be a family of absolutely continuous functions and set
	\begin{equation*}
		\ga_s(x)=\Big(x,\int_0^xf_s(u)\?du\Big)\quad (x\in [0,b]).
	\end{equation*}  
	A straightforward computation shows that the curvature of $\ga_s$ is given by 
	\[
	\ka_{\ga_s}(x)=\frac{\dot f_s(x)}{[1+f_s(x)^2]^{\frac{3}{2}}}\quad (x\in [0,b]).
	\]
	Therefore, $\ga_s$ lies in $\hat{\sr L}_{-\ka_0}^{+\ka_0}(Q;\theta_1)$ if and only if $f_s$ satisfies:
	\begin{enumerate}
		\item [(i)] $\vert \dot f_s(x)\vert \leq \ka_0[1+f_s(x)^2]^{\frac{3}{2}}$ for almost every $x\in [0,b]$ (that is, $\ka_{\ga_s}\in [-\ka_0,+\ka_0]$ a.e.);
		\item [(ii)] $f_s(0)=r_0:=\dot y(0)$ and $f_s(b)=r_b:=\dot y(b)$ (that is, $\ta_{\ga_s}(0)=\ta_\ga(0)$ and $\ta_{\ga_s}(b)=\ta_{\ga}(b)$);
		\item [(iii)] $\int_0^b f_s(x)\?dx=A_1:=y(b)-y(0)$ (that is, $\ga_s(b)=\ga(b)$).
	\end{enumerate}
	 We will now produce a homotopy of $f=f_1$ through absolutely continuous functions satisfying (i)--(iii). 
	 \begin{figure}[ht]
		\begin{center}
			\includegraphics[scale=.12]{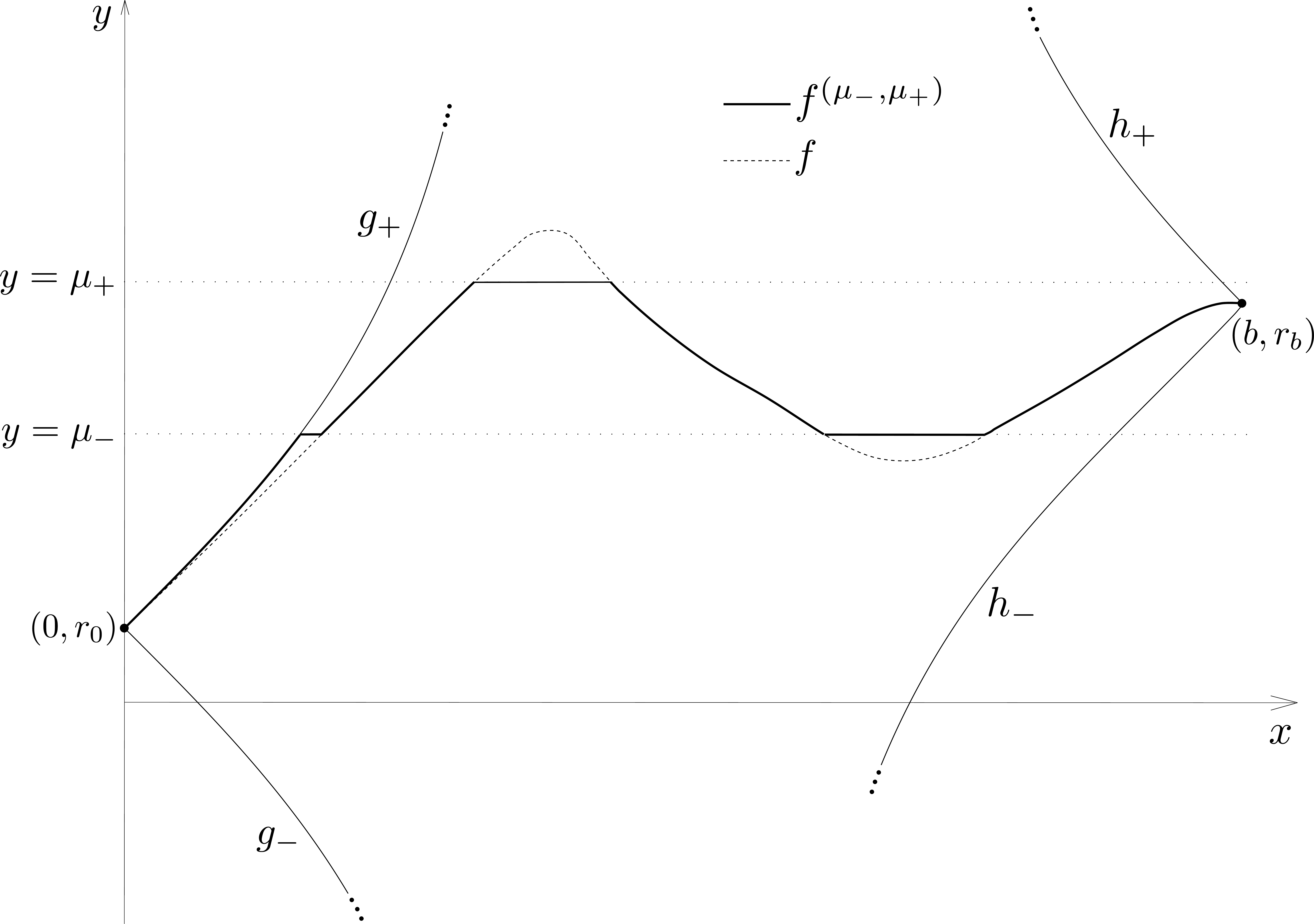}
			\caption{An illustration of \cref{C:excavator}.}
			\label{F:excavator}
		\end{center}
	\end{figure}
	 Define 
	\begin{equation}\label{E:gpm}
		\al_{\pm}=\mp \frac{r_0}{\sqrt{1+r_0^2}},\quad g_{\pm}(x)=\pm \frac{\ka_0x-\al_{\pm}}{\sqrt{1-(\ka_0x-\al_\pm)^2}}\quad \text{for}\ \ x\in \Big(\frac{\al_\pm-1}{\ka_0},\frac{\al_\pm+1}{\ka_0}\Big)
	\end{equation}
	(see Figure \ref{F:excavator}) and, similarly,
	\begin{equation}\label{E:hpm}
		\be_{\pm}=\ka_0b\pm \frac{r_b}{\sqrt{1+r_b^2}},\quad h_{\pm}(x)=\mp \frac{\ka_0x-\be_{\pm}}{\sqrt{1-(\ka_0x-\be_\pm)^2}}\quad\text{for}\ \  x\in \Big(\frac{\be_\pm-1}{\ka_0},\frac{\be_\pm+1}{\ka_0}\Big).
	\end{equation}
	The functions $g_\pm$ are the solutions of the differential equations $\dot g=\pm \ka_0(1+g^2)^{\frac{3}{2}}$ with $g(0)=r_0$. Similarly, $h_{\pm}$ are the solutions of the differential equations $\dot h=\mp \ka_0(1+h^2)^{\frac{3}{2}}$ with $h(b)=r_b$. Extend their domains to all of $\R$ by setting
	\begin{equation*}
		g_{\pm}(x)=\pm \infty \text{ \ \ if \ \ $x\geq \frac{\al_{\pm}+1}{\ka_0}$} \quad\text{and}\quad 	g_{\pm}(x)= \mp \infty  \text{ \ \ if \ \  $x\leq \frac{\al_{\pm}-1}{\ka_0}$}
	\end{equation*}
	and similarly for $h_{\pm}$.
	Since the curvature of $\ga=\ga_1$ takes values in $[-\ka_0,+\ka_0]$, condition (i) applied to $f=f_1$ gives:
	\begin{equation}\label{E:ineq}
		g_-(x),\,h_-(x)\leq  f(x)\leq  g_+(x),\,h_+(x)\quad \text{for all $x\in [0,b]$}.
	\end{equation}
	 Let 
	\begin{equation}\label{E:mpmm}
		m_-=\inf_{x\in [0,b]}f(x),\quad m_+=\sup_{x\in [0,b]}f(x)\quad \text{and}\quad \De=\set{(\mu_-,\mu_+)\in [m_-,m_+]}{\mu_-\leq \mu_+}.
	\end{equation} 
	For $(\mu_-,\mu_+)\in \De$, let $f^{(\mu_-,\mu_+)}\colon [0,b]\to \R$ be given by
	\begin{equation}\label{E:median}
		f^{(\mu_-,\mu_+)}(x)=\midd\big(h_-(x)\,,\,g_-(x)\,,\,\mu_-\,,\,f(x)\,,\,\mu_+\,,\,g_+(x)\,,\,h_+(x)\big)
	\end{equation}
	(cf.~Figure \ref{F:excavator}). The functions $f^{(\mu_-,\mu_+)}$ automatically satisfy conditions (i) and (ii). Define $A\colon \De\to \R$ to be the area under the graph of $f^{(\mu_-,\mu_+)}$:
	\begin{equation*}
		A(\mu_-,\mu_+)=\int_0^bf^{(\mu_-,\mu_+)}(x)\?dx.
	\end{equation*}
	It is immediate from \eqref{E:median} that:
	\begin{enumerate}
		\item [(A)] $A$ is increasing as a function of $\mu_-$ (resp.~$\mu_+$);
		\item [(B)] $A$ is a Lipschitz function of $(\mu_-,\mu_+)$. In fact, 
		\begin{equation*}
		\abs{A(\mu_-+u,\mu_++v)-A(\mu_-,\mu_+)}\leq b(\abs{u}+\abs{v}).
	\end{equation*}
	\end{enumerate}  
	By (A), for each $s\in [0,1]$, the set 
	\begin{equation*}
		\set{(\mu_-,\mu_+)\in \De}{A(\mu_-,\mu_+)=A_1\text{\ and\ }\mu_+-\mu_-=(m_+-m_-)s}
	\end{equation*} is an interval of the latter line in the $(\mu_-,\mu_+)$-plane. Let $(\mu_-(s),\mu_+(s))$ be the coordinates of the center of this interval. By (B), $\mu_-(s)$ and $\mu_+(s)$ are continuous (even Lipschitz), and (A) implies that $\mu_-$ is a decreasing, while $\mu_+$ is an increasing function of $s\in [0,1]$. The functions
	\begin{equation*}
		f_s\colon [0,b]\to \R,\quad f_s=f^{(\mu_-(s),\mu_+(s))}
	\end{equation*}	
	satisfy all of conditions (i)--(iii) by construction. We repeat their definition for convenience:
	\begin{equation}\label{E:f_s}
		\begin{alignedat}{10}
				f_s(x)&=\midd\big(h_-(x)\,,\,g_-(x)\,,\,\mu_-(s)\,,\,f(x)\,,\,\mu_+(s)\,,\,g_+(x)\,,\,h_+(x)\big), 	\\
				\ga_s(x)&=\Big(x,\int_0^xf_s(u)\?du\Big) \quad (x\in [0,b]).
		\end{alignedat} %
	\end{equation}
	We will denote $\mu_+(0)=\mu_-(0)$ by $\mu_0$. The monotonicity of $\mu_-$,\,$\mu_+$ implies that
	\begin{equation}\label{E:sigmas}
		\mu_-(s)\leq \mu_0\leq \mu_+(s)\quad\text{ for all $s\in [0,1]$}.\pushQED{\qed} \qedhere\popQED
	\end{equation}
	\end{cons}
	
	\begin{rem}\label{R:geom}
		We deduce from \eqref{E:ineq} and \eqref{E:f_s} that 
		\begin{equation*}
			f_0=\midd(h_-,g_-,\mu_0,g_+,h_+).
		\end{equation*}
The graph of $f_0$  is composed of at most three parts: a piece of the graph of $g_-$ or $g_+$, a piece of the graph of the constant function $y=\mu_0$, and a piece of the graph of $h_-$ or $h_+$. The corresponding curve $\ga_0$ is thus the concatenation of an arc of circle of curvature $\pm \ka_0$, a line segment and another arc of circle of curvature $\pm \ka_0$, though some of these may degenerate to a point. It is an immediate consequence of \cref{C:Dubinspath} that $\ga_0$ (and hence $f_0$) is independent of $\ga$ and of the chosen axis $\vphi$. 
	\end{rem}
	
	\begin{lem}\label{L:independent}
		Let $\vphi$ be an axis of  $\ga\in \hat{\sr U}_c$ and $s\mapsto \ga_s$ $(s\in [0,1])$ be the deformation described in \cref{C:excavator}. Then $\ga_0\in \hat{\sr U}_c$ is the unique curve of minimal length in $\hat{\sr L}_{-\ka_0}^{+\ka_0}(Q)$.\qed
	\end{lem}
	
	\begin{urem}
		Notice that this proves Dubins' Theorem I in \cite{Dubins} in the case where $\hat{\sr L}_{-\ka_0}^{+\ka_0}(Q)$ contains condensed curves. Furthermore, given $Q$ and $\ka_0$, we can use \cref{C:excavator} to describe $\ga_0$ explicitly.
	\end{urem}
	
	\begin{lem}\label{L:A+A-}
	Let $S_+=\set{x\in [0,b]}{f(x)\geq \mu_0}$ and $S_-=\set{x\in [0,b]}{f(x)\leq \mu_0}$. Then $f_s(x)$ is an increasing \tup(resp.~decreasing\tup) function of $s\in [0,1]$ if $x\in S_+$ \tup(resp.~$S_-$\tup). Moreover, for all $s\in [0,1]$, $f_s(x)\geq \mu_0$ if $x\in S_+$ and $f_s(x)\leq \mu_0$ if  $x\in S_-$.
\end{lem}
\begin{proof}
	Suppose that $x\in S_+$. From \eqref{E:ineq} and \eqref{E:sigmas}, we deduce that 
	\begin{equation*}
		g_-(x),\,h_-(x),\,\mu_-(s)\leq f(x)\leq g_+(x),\,h_+(x).
	\end{equation*}
	Hence, $f_s(x)=\min \se{\mu_+(s),f(x)}\geq \mu_0$ and $f_s(x)$ increases with $s$ since $\mu_+(s)$ does. The proof for $x\in S_-$ is analogous.
\end{proof}

\begin{cor}\label{C:maxmin}
	Let $m_-(s)=\inf_{x\in [0,b]}f_s(x)$ and $m_+(s)=\sup_{x\in [0,b]}f_s(x)$. Then $m_+(s)$ is an increasing and $m_-(s)$ a decreasing  function of $s\in [0,1]$.\qed
\end{cor}

\begin{lem}\label{L:Dpath}
		Let $\ga\in \hat{\sr U}_c$ and $s\mapsto \ga_s\in \hat{\sr L}_{-\ka_0}^{+\ka_0}(Q;\theta_1)$ be the homotopy described in \cref{C:excavator}. Let $\om_s$ denote the amplitude of $\ga_s$. Then $\om_s$ is an increasing function of $s$; in particular, $\ga_s$ is condensed \tup(i.e., $\ga_s\in \hat{\sr U}_c$\tup) for all $s\in [0,1]$. 
	\end{lem}
	\begin{proof}
		Let $\vphi$ be the axis of $\ga$ chosen for the construction. Recall that, by definition,
		\begin{equation}\label{E:amp}
			\om_s=\sup_{x\in [0,b]}\theta_s(x)-\inf_{x\in [0,b]}\theta_s(x)\quad (s\in [0,1]),
		\end{equation}
		where $\theta_s$ is the argument of $\ta_{\ga_s}$ such that $\theta_s(0)=0$. By \eqref{E:f_s}, 
		\begin{equation}\label{E:tan}
			f_s(x)=\tan(\theta_s(x)-\vphi).
		\end{equation}
		Because the tangent is an increasing function, \cref{C:maxmin} immediately implies that $\om_s$ is increasing.
	\end{proof}

	\begin{urem}
		Although $\ga_0$ has minimal amplitude in $\hat{\sr U}_c$ by the previous lemma, there may be other curves in $\hat{\sr U}_c$ with the same amplitude. This is the case, for instance, for the curves $\ga_0$ and $\ga$ corresponding to the functions $f$ and $f_0$ of Figure \ref{F:excavator}.
	\end{urem}
	
	\begin{lem}\label{L:length}
		Let $\ga\in \hat{\sr U}_c$ and $s\mapsto \ga_s$ the deformation described in \cref{C:excavator}. Then the length of $\ga_s$ is an increasing function of $s\in [0,1]$.
	\end{lem}
	\begin{proof}
		Let $\la\colon \R\to \R$ be given by $\la(u)=(1+u^2)^{\frac{1}{2}}$. A straightforward computation shows that
		\begin{equation}\label{E:computation}
			\la''(u)=(1+u^2)^{-\frac{3}{2}}>0\ \ \text{ for all $u\in \R$}.
		\end{equation}
		Moreover, by definition \eqref{E:f_s}, the length $L_s$ of $\ga_s$ is given by 
		\begin{equation*}
			L_s=\int_0^b(\la\circ f_s)(x)\,dx.
		\end{equation*}
		Let $s_1\leq s_2\in [0,1]$,  $S_+$, $S_-$ be as in \lref{L:A+A-} and
		\begin{alignat*}{10}
			T_+=&\set{(x,y)\in [0,b]\times \R}{f_{s_1}(x)\leq y\leq f_{s_2}(x)},\\
			T_-=&\set{(x,y)\in [0,b]\times \R}{f_{s_2}(x)\leq y\leq f_{s_1}(x)}.
		\end{alignat*}
		Using \lref{L:A+A-}, we deduce that 
		\begin{alignat*}{10}
			L_{s_2}-L_{s_1}&=\int_0^b(\la\circ f_{s_2})(x)-(\la\circ f_{s_1})(x)\,dx & & \\ 
			&=\bigg(\int_{S_+}+\int_{S_-}\bigg)(\la\circ f_{s_2})(x)-(\la\circ f_{s_1})(x)\,dx & &\\
			&=\bigg(\int_{T_+}-\int_{T_-}\bigg)\la'(y)\,dy\,dx & & \\
			&\geq \bigg(\int_{T_+}-\int_{T_-}\bigg)\la'(\mu_0)\,dy\,dx& &\ \  \text{by \eqref{E:computation}}\\
			&=\la'(\mu_0)\Big(\int_0^b f_{s_2}-\int_0^b f_{s_1}\Big)=0& &\ \ \text{by the definition of $f_s$}.
		\end{alignat*}
		Therefore, $L_s$ is an increasing function of $s\in [0,1]$.
	\end{proof}
	
	We are finally ready to prove \pref{P:excavator} and \tref{T:condensed}.
\begin{proof}[Proof of \pref{P:excavator}]
	For each $\ga\in \hat{\sr U}_c$, let
	\begin{equation}\label{E:aver}
		\bar\vphi_\ga=\frac{1}{2}\Big(\sup_{t\in [0,1]}\theta_\ga(t)+\inf_{t\in [0,1]}\theta_\ga(t)\Big),
	\end{equation}
	where $\theta_\ga\colon [0,1]\to \R$ is the argument of $\ta_{\ga}$ satisfying $\theta_\ga(0)=0$. It is clear that $\bar\vphi_\ga$ depends continuously on $\ga\in \hat{\sr U}_c$. Define $H\colon [0,1]\times \hat{\sr U}_c\to \hat{\sr U}_c$ by $H(s,\ga)=\ga_s$, where $\ga_s$ is the curve \eqref{E:f_s} constructed in \cref{C:excavator} with chosen axis $\bar\vphi_\ga$. Then part (ii) of \pref{P:excavator} follows from \lref{L:Dpath}, and part (iii) from \lref{L:length}. The last assertion of \pref{P:excavator} and part (i) were established in \rref{R:geom}.
\end{proof}

\begin{proof}[Proof of \tref{T:condensed}]
	Assume that $\sr U_c$ is nonempty. It is certainly open in $\sr L_{-1}^{+1}(Q;\theta_1)$. Hence, by \lref{L:Hilbert}, it suffices to prove that $\sr U_c$ is weakly contractible. Let $K$ be a compact manifold and $g\colon K\to \sr U_c$, $a\mapsto \ga^a$, be a continuous map. Using \lref{L:smoothie}, we may assume that the image of $g$ is contained in (the image under set inclusion of) $\sr C_{-\ka_0}^{+\ka_0}(Q;\theta_1)$ for some  $\ka_0\in (0,1)$. By \lref{L:inclusions}, we have continuous injections
	\begin{equation*}
		\sr C_{-\ka_0}^{+\ka_0}(Q;\theta_1)\to \hat{\sr L}_{-\ka_0}^{+\ka_0}(Q;\theta_1)\to \sr L_{-1}^{+1}(Q;\theta_1).
	\end{equation*} 
	Let $G\colon [0,1]\times K\to \sr L_{-1}^{+1}(Q;\theta_1)$ map $(s,a)$ to (the image under set inclusion of) $H(s,\ga^a)$, with $H$ as in \pref{P:excavator}. Then $G$ is a null-homotopy of $g$ in $\sr U_c$. \end{proof}
	
	The next couple of lemmas will only be needed in later sections.

\begin{lem}\label{L:bounded}
	Suppose that there exists $\hat \om\in (0,\pi)$ such that if $\ga\in \hat{\sr U}_c\subs \hat{\sr L}_{-\ka_0}^{+\ka_0}(Q;\theta_1)$ then its amplitude $\om_\ga$ satisfies $\om_\ga\leq \hat \om$. Let $L(\eta)$ denote the length of $\eta$. Then $\sup_{\ga\in \hat{\sr U}_c}L(\ga)$ is finite. In particular,  the images of $\ga\in \hat{\sr U}_c$ are all contained in some bounded subset of $\C$.
\end{lem}
\begin{proof}
	Let $\ga\in \hat{\sr U}_c$ and $\bar\vphi_\ga$ be as in \eqref{E:aver}. By hypothesis, the image of $\theta_\ga\colon [0,1]\to \R$ is contained in $\big[\bar\vphi_\ga-\frac{\hat\om}{2}\,,\,\bar\vphi_\ga+\frac{\hat\om}{2}\big]$. Let $f\colon [0,b]\to \R$ be the function corresponding to $\ga$ and the axis $\bar\vphi_\ga$,  in the notation of \cref{C:excavator}. Note that $b=\gen{e^{i\bar\vphi_\ga},q}\leq \abs{q}$, where $q$ is the $\C$-coordinate of $Q$. By \eqref{E:tan}, 
	\begin{equation*}
		\abs{f(x)}\leq \tan\Big({\frac{\hat \om}{2}}\Big)\text{\ for all $x\in [0,b]$}.
	\end{equation*}
	Therefore, the length $L(\ga)$ of $\ga$ satisfies
	\begin{equation*}
		L(\ga)=\int_0^b\sqrt{1+f(x)^2}\?dx\leq b\sec\Big({\frac{\hat \om}{2}}\Big)\leq \abs{q}\sec\Big({\frac{\hat \om}{2}}\Big).\qedhere
	\end{equation*}
\end{proof}

\begin{lem}\label{L:compress}
	Let $\hat{\sr U}_c\subs \hat{\sr L}_{-\ka_0}^{+\ka_0}(Q;\theta_1)$ and $H\colon [0,1]\times \hat{\sr U}_c\to \hat{\sr U}_c$ be the deformation described in \pref{P:excavator} and \cref{C:excavator}. Suppose that $\theta_1=0$. Then $\om_0<\om_1$ unless $\ga_1=\ga_0$. 
\end{lem}
	\begin{proof}
		It is obvious that $\om_1=\om_0$ if $\ga_1=\ga_0$. The condition $\theta_1=0$ is equivalent to $r_0=r_b$, in the notation of \cref{C:excavator}. Suppose without loss of generality that $\mu_0\geq r_0$, so that $m_+(0)=\mu_0$. 
		
		If $m_+(1)\leq \mu_0$, then $S_-=[0,b]$. Hence, by \lref{L:A+A-}, $f_1(x)\leq f_0(x)$ for all $x\in [0,b]$.  Since $f_1$ and $f_0$ have the same area, we conclude that $f_1=f_0$, that is, $\ga_1=\ga_0$.
		
		By \cref{C:maxmin}, $m_-(1)\leq m_-(0)$. Hence, if $m_+(1)>\mu_0=m_+(0)$, then $\om_0<\om_1$ by \eqref{E:amp} and \eqref{E:tan}. 
	\end{proof}

\subsection*{Existence of condensed curves}The question of whether $\sr U_c\neq \emptyset$ is settled by means of an elementary geometric construction. In all that follows, $C_r(a)$ denotes the circle of radius $r>0$ centered at $a\in \C$. 

\begin{figure}[ht]
	\begin{center}
		\includegraphics[scale=.32]{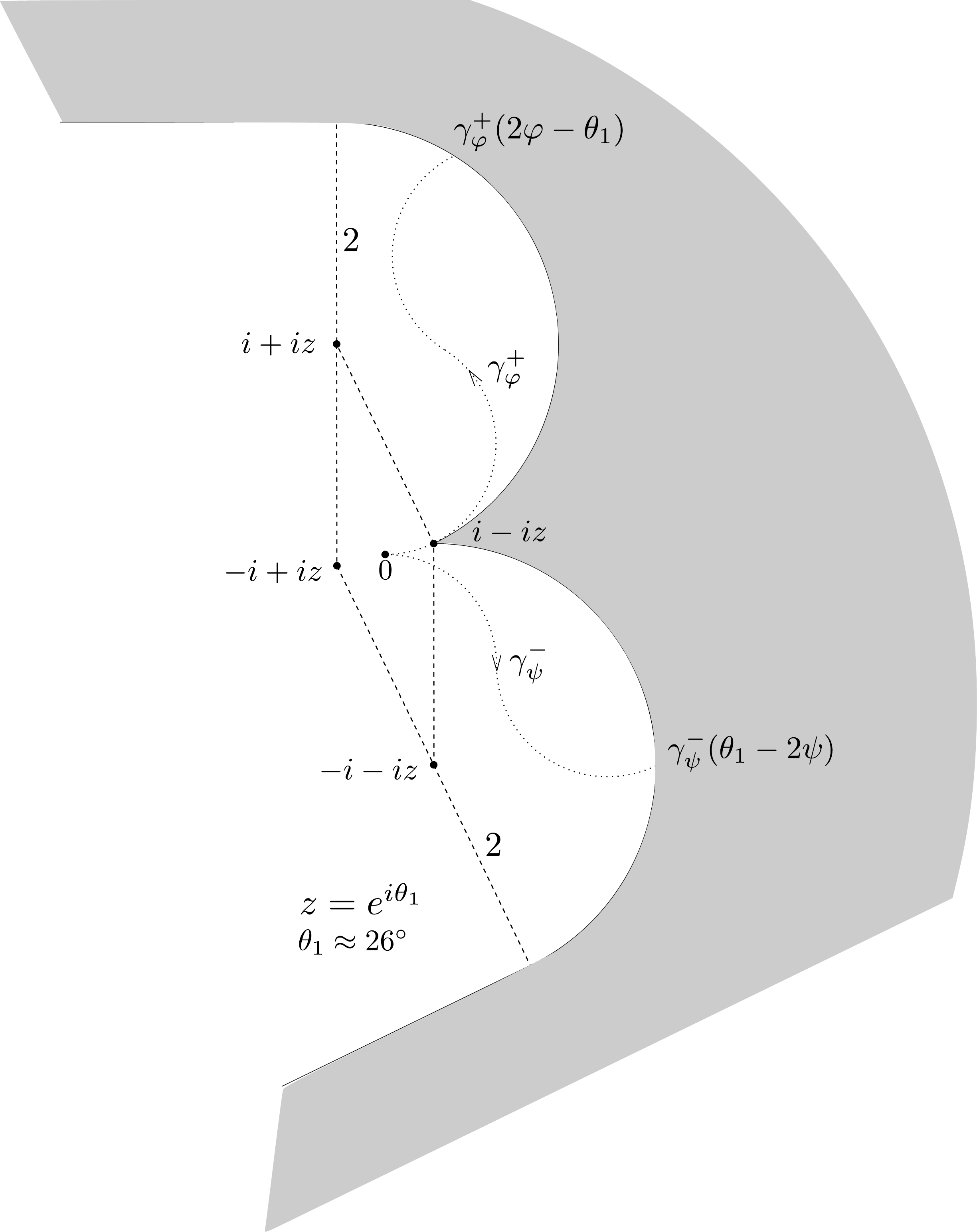}
		\caption{Let $\theta_1\in [0,\pi)$ be fixed and $Q=(q,z)$, where $z=e^{i\theta_1}$. There exist condensed curves in $\sr L_{-1}^{+1}(Q;\theta_1)$ if and only if $q$ belongs to the open gray region.}
		\label{F:condensed}
	\end{center}
\end{figure}

\begin{prop}\label{P:condensedlocation}
	Let $\theta_1\in [0,\pi)$ be fixed, $z=e^{i\theta_1}$ and $Q=(q,z)\in \C\times \Ss^1$. Let $R_{\sr U_c}$ be the open region of the plane which does \emph{not} contain $-i+iz$ and which is bounded by the shortest arcs of the circles $C_2(\pm(i+iz))$ joining $i-iz$ to $i-iz\pm 2(i+iz)$ and their tangent lines at the latter points. Then $\sr L_{-1}^{+1}(Q;\theta_1)$ contains condensed curves if and only if $q\in R_{\sr U_c}$.
\end{prop}
It is clear from the definition of condensed curve that $\sr U_c\subs \sr L_{-1}^{+1}(Q;\theta_1)$ is empty if $\abs{\theta_1}\geq \pi$. In other words, all condensed curves in $\sr L_{-1}^{+1}(Q)$ must be contained in the subspace $\sr L_{-1}^{+1}(Q;\theta_1)$ with $\theta_1$ the unique number in $(-\pi,\pi)$ satisfying $e^{i\theta_1}=z$ (for $z\neq -1$). We have assumed that $\theta_1\in [0,\pi)$ just to simplify the statement. If $\theta_1\in (-\pi,0]$, the only difference is that $i-iz$ should be interchanged with $-i+iz$. The proof is analogous to the one given below. Alternatively, it can be deduced from the proposition by applying a reflection across the $x$-axis. When $\theta_1=0$, the statement becomes ambiguous; in this case the arcs of circles which bound $R_{\sr U_c}$ are centered at $\pm2i$, bounded by 0 and $\pm 4i$, and pass through the points $2\pm 2i$, respectively.

\begin{proof}[Proof of \pref{P:condensedlocation}.]
Let $\eta\colon [0,1]\to \C$ be condensed and let $\theta_\eta\colon [0,1]\to \R$ be the argument of $\ta_\eta$ satisfying $\theta_\eta(0)=0$. Observe that
	\begin{equation*}\label{E:minmax}
		 \inf\set{\theta_\eta(t)}{t\in [0,1]}\in [\theta_1-\pi,0]\text{\quad and\quad}\sup\set{\theta_\eta(t)}{t\in [0,1]}\in[\theta_1,\pi].
	\end{equation*}
	The proof relies on the study of the following curves. For each $\vphi\in [\theta_1,\pi]$, define $\ga_\vphi^+\colon [0,2\vphi-\theta_1]\to \C$ to be the unique curve parametrized by arc-length satisfying:
	\begin{equation*}
		\ga_\vphi^+(0)=0\text{\quad and\quad}\ta_{\ga_\vphi^+}(s)=\begin{cases}
		e^{is} & \text{\quad if $s\in [0,\vphi]$} \\
		e^{i(2\vphi-s)} & \text{\quad if $s\in [\vphi,2\vphi-\theta_1]$}
		\end{cases}
	\end{equation*}
	Then $\ga_{\vphi}^+$ is the concatenation of two arcs of circles of radius 1, 
	\begin{alignat*}{9}
	&\inf_{t\in [0,1]}\theta_{\ga_\vphi^+}(t)=0,\quad \sup_{t\in [0,1]}\theta_{\ga_\vphi^+}(t)=\vphi,\quad \ta_{\ga_\vphi^+}(2\vphi-\theta_1)=z \quad\text{and} \\
	&\ga_\vphi^+(2\vphi-\theta_1)=\int_0^\vphi e^{is}\,ds+\int_\vphi^{2\vphi-\theta_1}e^{i(2\vphi-s)}\,ds=(i+iz)-2ie^{i\vphi}.
	\end{alignat*}
	Thus, as $\vphi$ increases from $\theta_1$ to $\pi$, the endpoints of the $\ga_\vphi^+$ trace out the arc of $C_2(i+iz)$ bounded by $i-iz$ and $3i+iz$. Further, the tangent line to $C_2(i+iz)$ at $\ga_\vphi^+(2\vphi-\theta_1)$ is parallel to $e^{i\vphi}$, for it must be orthogonal to $-2ie^{i\vphi}$. 
	
	Similarly, for each $\psi\in [\theta_1-\pi,0]$, let $\ga_\psi^-\colon [0,\theta_1-2\psi]\to \C$ be the curve, parametrized by arc-length, which satisfies
	\begin{equation*}
		\ga_\psi^-(0)=0\text{\quad and\quad}\ta_{\ga_\psi^-}(s)=\begin{cases}
		e^{-is} & \text{\quad for $s\in [0,-\psi]$} \\
		e^{i(2\psi+s)} & \text{\quad for $s\in [-\psi,\theta_1-2\psi]$}
		\end{cases}
	\end{equation*}
	Then $\ga_{\psi}^-$ is the concatenation of two arcs of circles of radius 1, $\ta_{\ga_\psi^-}(\theta_1-2\psi)=z$ for all $\psi\in [\theta_1-\pi,0]$, and as $\psi$ decreases from $0$ to $\theta_1-\pi$, the endpoints of the $\ga_\psi^-$ traverse the arc of $C_2(-i-iz)$ bounded by $i-iz$ and $-i-3iz$. Moreover, the tangent line to this circle at $\ga_\psi^-(\theta_1-2\psi)$ is parallel to $e^{i\psi}$.
	
	Any $q\in \ol{R}_{\sr U_c}$ is the endpoint of a curve of one of the following three types:
	\begin{enumerate}
		\item [(i)] The concatenation of a $\ga_{\vphi}^+$ or a $\ga_{\psi}^-$ with a line segment of direction $z$.
		\item [(ii)] The concatenation of $\ga_{\pi}^+|_{[0,\pi]}$, a line segment of length $\ell\geq 0$ having direction $-1$, the arc $-\ell +\ga_\pi^+|_{[\pi,2\pi-\theta_1]}$, and a line segment of direction $z$.
		\item [(iii)] The concatenation of $\ga^-_{\theta_1-\pi}|_{[0,\pi-\theta_1]}$, a line segment of length $\ell_1\geq 0$ of direction $-z$, the arc $-\ell_1 z+\ga^-_{\theta_1-\pi}|_{[\pi-\theta_1,\frac{3\pi}{2}-\theta_1]}$, a line segment of length $\ell_2\geq 0$ and direction $-iz$, the arc $-\ell_1 z-\ell_2iz+\ga^-_{\theta_1-\pi}|_{[\frac{3\pi}{2}-\theta_1,2\pi-\theta_1]}$, and a line segment of direction $z$.
	\end{enumerate}
	The curves which we have described have curvature equal to  $\pm 1$ over intervals of positive measure and, additionally, may be critical curves. Nevertheless, for any $q\in R_{\sr U_c}$, we can find a condensed $\ga\in \sr L_{-1}^{+1}(Q;\theta_1)$ by composing one of these curves with a dilatation through a factor $c>1$, with $c$ close to $1$ if $q$ lies close to $\bd R_{\sr U_c}$, and by avoiding the argument $\pi$ (for a curve of type (i)) or $\theta_1-\pi$ (for a curve of type (iii)).
	
	Conversely, suppose that $\sr L_{-1}^{+1}(Q)$ contains condensed curves. Let $\eta\colon [0,L]\to \C$ be such a curve, parametrized by arc-length, and let $\vphi=\sup\theta$, where $\theta\colon [0,L]\to \R$ is an argument of $\ta_\eta$ satisfying $\theta(0)=0$. Define
	\begin{equation*}
g\colon [0,L]\to \R\quad\text{by}\quad g(s)=\big\langle\eta(s)-\ga_\vphi^+(2\vphi-\theta_1)\,,\,ie^{i\vphi}\big\rangle.
	\end{equation*}
	Note that $g(s)>0$ if and only if $\eta(s)$ lies to the left of the line through $\ga_{\vphi}^+(2\vphi-\theta_1)\in C_2(i+iz)$ having direction $e^{i\vphi}$; we have already seen that this line is tangent to this circle at this point. We claim that $g(L)<0$. Since $\eta$ is admissible, $\theta=\arg\circ \ta_\eta$ is an absolutely continuous function, and $\abs{\theta'}=\abs{\ka_\eta}<1$ almost everywhere by eq.~\eqref{E:rate}. Moreover, $\theta(s)\in [\vphi-\pi,\vphi]$ for all $s$ because $\eta$ is condensed. Hence, 
	\begin{equation}\label{E:nonnega}
		g'(s)=\big\langle e^{i\theta(s)},ie^{i\vphi}\big\rangle=\cos\big(\theta(s)-\vphi-\tfrac{\pi}{2}\big)\leq 0\text{\quad  for all $s\in [0,L]$}.
	\end{equation}
	Let $J_i=(a_i,b_i)\subs (0,L)$ ($i=1,2,3$) be disjoint intervals such that:
	\begin{enumerate}
		\item [(I)] $\theta(a_1)=0$ and $\theta(b_1)=\theta_1$; 
		\item [(II)]  $\theta(a_2)=\theta_1$ and $\theta(b_2)=\vphi$; 
		\item [(III)] $\theta(a_3)=\vphi$ and $\theta(b_3)=\theta_1$.
	\end{enumerate}
	Such intervals exist because $\theta$ is a continuous function satisfying $\theta(0)=0$, $\theta_1\leq \vphi=\sup \theta$ and $\theta(L)=\theta_1$. Let $\la$ denote the Lebesgue measure on $\R$. Fix $i$ and let $[\al,\be]$ be any nondegenerate subinterval of $\theta((a_i,b_i))$. Since $\theta$ is strictly 1-Lipschitz, if $S=\set{s\in (a_i,b_i)}{\al\leq \theta(s)\leq \be}$, then $\la(S)>\be-\al$. Combining this with \eqref{E:nonnega}, we deduce that
	\begin{alignat*}{9}
		g(L)-g(0)& \leq \bigg(\int_{a_1}^{b_1}+\int_{a_2}^{b_2}+\int_{a_3}^{b_3}\bigg) g'(s)\,ds\\
		&<\int_0^{\theta_1}\gen{e^{it},e^{i\vphi}}\,dt+2\int_{\theta_1}^{\vphi}\big\langle e^{it},e^{i\vphi}\big\rangle\,dt=\big\langle\ga_\vphi^+(2\vphi-\theta_1),ie^{i\vphi}\big\rangle.
	\end{alignat*}
	Therefore, $g(L)<0$ as claimed. Similarly, if $\psi=\inf \theta$, then $\eta(L)$ lies on the side of the tangent to $C_2(-i-iz)$ at $\ga_{\psi}^-(\theta_1-2\psi)$ which does not contain $-i+iz$. It follows that $q=\eta(L)\in R_{\sr U_c}$.
\end{proof}


\section{Topology of $\sr U_d$}\label{S:diffuse}

Throughout this section, let $K$ denote a compact manifold, possibly with boundary. Also, let $Q=(q,z)\in \C\times \Ss^1$ be fixed (but otherwise arbitrary) and let $\sr U_d\subs  \sr L^{+1}_{-1}(Q;\theta_1)$ denote the subset consisting of all diffuse curves in $\sr L_{-1}^{+1}(Q)$ having total turning $\theta_1$, for some fixed $\theta_1\in \R$ satisfying $e^{i\theta_1}=z$. Finally, let $O=(0,1)\in \C\times \Ss^1$, the identity element of the group $UT\C$.

Our next objective is to prove that $\sr U_d$ is contractible. The idea behind the proof is quite simple. If $\ga$ is diffuse, then we can ``graft'' a straight line segment of length greater than 4 onto $\ga$, as illustrated in Figure \ref{F:grafting}. By the theorem of Dubins stated in the introduction, this segment can be deformed so that in the end an eight curve of large radius traversed a number $n$ of times has been attached to it, as in Figure \ref{F:spreading}\?(e). These eights are then spread along the curve, as in Figure \ref{F:spreading}\?(f). If $n\in \N$ is large enough, the spreading can be carried out within $\sr L_{-1}^{+1}(Q)$. The result is a curve which is so loose that the constraints on the curvature may be safely forgotten, allowing us to use the following fact.

\begin{thm}[Smale]\label{T:Smale}
	Let $Q=(q,z)\in \C\times \Ss^1$. Then $\sr C_{-\infty}^{+\infty}(Q)$ and $\sr L_{-\infty}^{+\infty}(Q)$ have $\aleph_0$ connected components, one for each $\theta_1\in \R$ satisfying $e^{i\theta_1}=z$, all of which are contractible.
\end{thm}
\begin{proof}
	For the space $\sr C_{-\infty}^{+\infty}(Q)$, the proof was discussed in the introduction. We may replace $\sr C_{-\infty}^{+\infty}(Q)$ by $\sr L_{-\infty}^{+\infty}(Q)$ using \lref{L:C^2}.
\end{proof}


\begin{lem}\label{L:WG}
	Let $P\in UT\C$. Then $\sr C_{-1}^{+1}(P,P)$ and $\sr L_{-1}^{+1}(P,P)$ have $\aleph_0$ connected components, one for each turning number $n\in \Z$, all of which are contractible.
\end{lem}
\begin{proof}
	By \lref{L:C^2}, it suffices to prove the result for $\sr C_{-1}^{+1}(P,P)$. Let $\sr C_n\subs \sr C_{-1}^{+1}(P,P)$ denote the subset of all curves which have turning number $n$. Then each $\sr C_n$ is closed and open. Hence, to establish that $\sr C_n$ is a contractible component, it suffices, by \lref{L:Hilbert}\?(b), to prove that it is weakly contractible. 
	
	Recall that $\sr C_{-1}^{+1}(P,P)\home \sr C_{-1}^{+1}(O)$, the homeomorphism coming from composing all curves with a suitable Euclidean motion. We may thus assume that $P=O$. Let $K$ be a compact manifold and $f\colon K\to \sr C_n$ a continuous map. By \tref{T:Smale}, there exists a continuous $F\colon[0,1]\times K\to \sr C_{-\infty}^{+\infty}(O)$ such that $F_0=f$ and $F_1$ is a constant map. Let 
	\begin{equation*}
		M=2\sup\set{\vert\ka_{F(s,a)}(t)\vert}{s,t\in [0,1],~a\in K}.
	\end{equation*}
	Given a curve $\ga$, let $M\ga$ denote the dilated curve $t\mapsto M\ga(t)$. It is easy to see that $\ka_{M\ga}=\frac{\ka_\ga}{M}$. Hence, $MF$ is a homotopy between $Mf$ and a constant map within $\sr C_{-1}^{+1}(O)$. But $f$ and $Mf$ are homotopic within $\sr C_{-1}^{+1}(O)$ through  $u\mapsto uf$ ($u\in [1,M]$). Therefore, $f$ is null-homotopic.
\end{proof}	

\subsection*{Loops and eights}We shall now explain how to attach loops and eights to a curve, and how to spread eights along it (Figure \ref{F:spreading}).

\begin{defn}
	We denote by $\al\colon \R\to \C$ the \tdef{loop} of radius $2$ and by $\be\colon \R\to \C$ the \tdef{eight} curve of the same radius (see figs.~\ref{F:spreading}\?(b) and (d)) given by:
\begin{alignat*}{9}
	\al(t)&=2i\big(1-\exp(2\pi it)\big) \\
	\be(t)&=\begin{cases}
	\al(2t) & \text{\ \ for\ \ }t\in [\frac{m}{2},\frac{m+1}{2}],~m\equiv 0\pmod 2 \\
	-\al(-2t) & \text{\ \ for\ \ }t\in [\frac{m}{2},\frac{m+1}{2}],~m\equiv 1\pmod 2
	\end{cases}\quad (m\in \Z).
\end{alignat*}
We shall also denote by $\al_n,\be_n\colon [0,1]\to \C$ a loop (resp.~eight) traversed $n\geq 1$ times: $\al_n(t)=\al(nt)$ and $\be_{n}(t)=\be(nt)$ ($t\in [0,1]$). 
\end{defn} 

Note that $\al_n,\be_{n}\in \sr L_{-1}^{+1}(O)$. The curvature of $\al_n$ is everywhere equal to $\tfrac{1}{2}$, and that of $\be_n$ equals $\pm \tfrac{1}{2}$ except at the $2n-1$ points where it is undefined. The turning number of $\al_n$ is $n$, and that of $\be_n$ is 0. 

\begin{figure}[ht]
	\begin{center}
		\includegraphics[scale=.08]{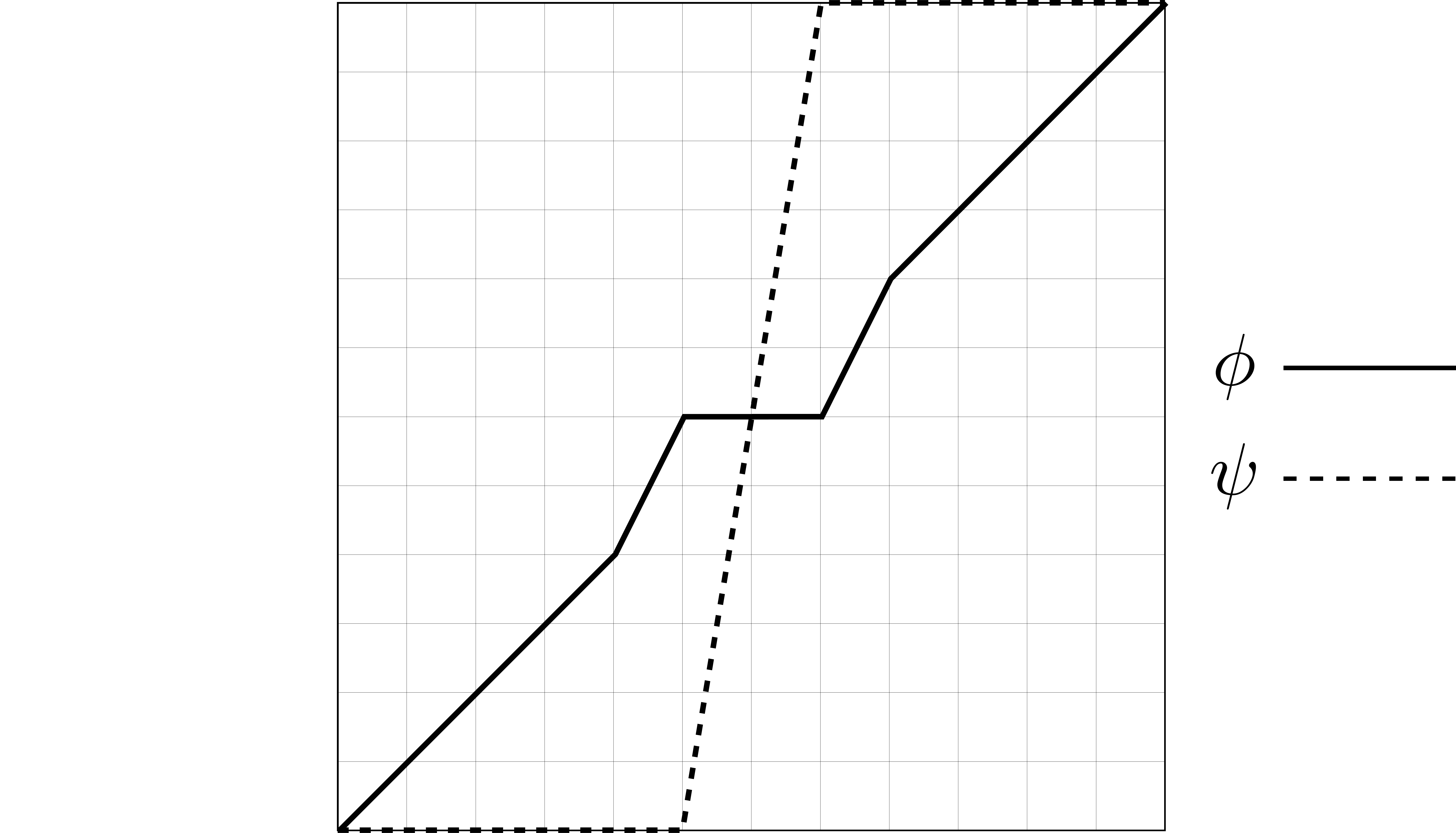}
		\caption{}
		\label{F:phipsi}
	\end{center}
\end{figure}

\begin{defn}\label{D:eighting}
	Let $t_0\in (0,1)$, $0<2\eps<\min\se{1-t_0,t_0}$, $1\leq n\in \N$ and $\ga$ be an admissible plane curve. Define piecewise linear functions $\phi,\psi\colon [0,1]\to [0,1]$ (whose graphs are depicted in Figure \ref{F:phipsi}) by:
	\begin{equation}\label{E:phipsi}
			\phi(t)=\begin{cases}
			t & \text{ if \ \  $t\nin [t_0-2\eps,t_0+2\eps]$}\\
			2t-t_0+2\eps & \text{ if \ \  $t\in [t_0-2\eps,t_0-\eps]$} \\
			t_0 & \text{ if \ \  $t\in [t_0-\eps,t_0+\eps]$} \\
			2t-t_0-2\eps & \text{ if \ \  $t\in [t_0+\eps,t_0+2\eps]$} 
			\end{cases}\ \  \psi(t)=\begin{cases}
			0 & \text{ if \ \  $t\in [0,t_0-\eps]$} \\
			\frac{t-t_0+\eps}{2\eps} & \text{ if \ \  $t\in [t_0-\eps,t_0+\eps]$} \\
		 	1 & \text{ if \ \  $t\in [t_0+\eps,1]$} 
			\end{cases}	
		\end{equation}
Define curves $A_{\ga,n,t_0}$,~$B_{\ga,n,t_0}$ (attaching loops, eights) and $S_{\ga,n}\colon [0,1]\to \C$ (spreading eights) by:
	\begin{alignat*}{9}
		A_{\ga,n,t_0}(t)&=\Phi_\ga(\phi(t))\al_n(\psi(t)) \\
		B_{\ga,n,t_0}(t)&=\Phi_\ga(\phi(t))\be_n(\psi(t))\qquad (t\in [0,1]). \\
		S_{\ga,n}(t)&=\Phi_\ga(t)\be_{n}(t).
	\end{alignat*}
Here  $\Phi_\ga\colon [0,1]\to \C\times \Ss^1$ is the frame of $\ga$ (as in \eqref{E:frame}), but viewed as a curve in the \tsl{group} $UT\C$: Each $\Phi_\ga(t)$ is an Euclidean motion, with $\Phi_\ga(t)a=\ga(t)+\ta_\ga(t)a$ for $a\in \C$. Different values of $\eps$ and $t_0$ yield curves which are homotopic in whichever space one is working with. 
\end{defn}

\begin{figure}[ht]
	\begin{center}
		\includegraphics[scale=.22]{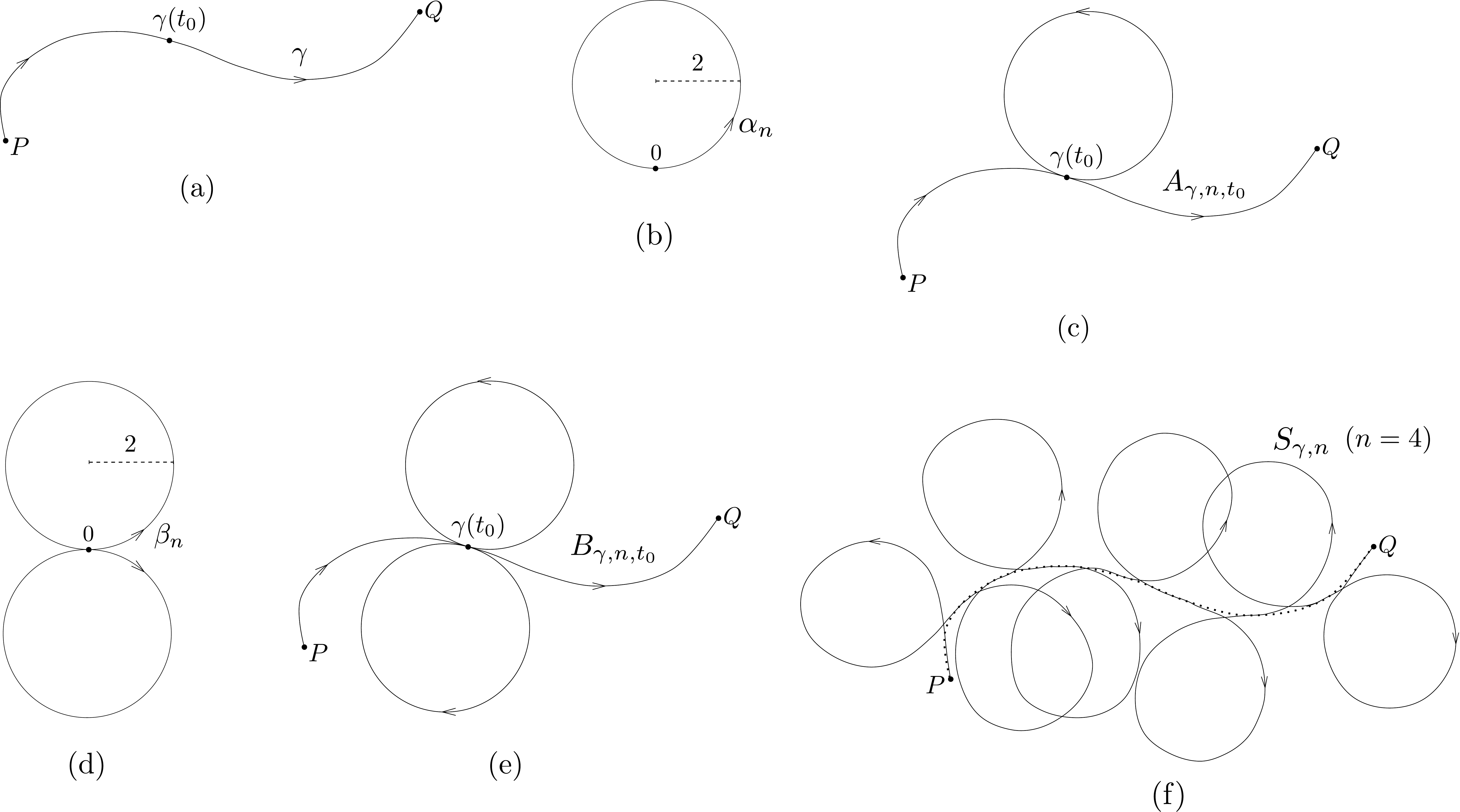}
		\caption{}
		\label{F:spreading}
	\end{center}
\end{figure}

\begin{lem}\label{L:basic}
		Let $t_0\in (0,1)$, $1\leq n\in \N$ and $\ga$ be an admissible plane curve. Then:
		\begin{enumerate}
			\item [(a)] $A_{\ga,n,t_0}$, $B_{\ga,n,t_0}$ and $S_{\ga,n}$ have the same initial and final frames as $\ga$. 
			\item [(b)] $B_{\ga,n,t_0}$ and $S_{\ga,n}$ lie in the same connected component of $\sr L_{-\infty}^{+\infty}(P,Q)$ \tup($P=\Phi_\ga(0)$, $Q=\Phi_\ga(1)$\tup).
			\item [(c)]  If $\ga\in \sr L_{-1}^{+1}(P,Q)$, then $A_{\ga,n,t_0},\,B_{\ga,n,t_0}\in \sr L_{-1}^{+1}(P,Q)$ also.
			\item [(d)] Let $O=(0,1)\in \C\times \Ss^1$. Then $\al_1$ and $B_{\al_1,n,t_0}$ lie in the same connected component of $\sr L_{-1}^{+1}(O)$ for all $n\geq 1$.
			\item [(e)] If $f,g\colon K\to \sr U_d$ are continuous and homotopic within $\sr U_d$, then so are $B_{f,n,t_0}$ and $B_{g ,n,t_0}$.
			\item [(f)] If $\ga$ is a reparametrization of $\al_1$, then $A_{\ga,n,t_0}$ is a reparametrization of $\al_{n+1}$.
		\end{enumerate}
\end{lem}
\begin{proof}
	It is clear that $A_{\ga,n,t_0}$, $B_{\ga,n,t_0}$ have the same initial and final frames as $\ga$, since they agree with $\ga$ in neighborhoods of the endpoints of $[0,1]$. From the definition of $S_{\ga,n}$, we find that
		\begin{equation*}
			\dot S_{\ga,n}=\dot \ga+\dot\ta_\ga\be_n+\ta_\ga\dot\be_n
		\end{equation*}
		Using that $\Phi_{\be_{n}}(0)=\Phi_{\be_n}(1)=(0,1)\in \C\times \Ss^1$, we deduce that:
		\begin{equation*}\label{E:}
			S_{\ga,n}(0)=\ga(0)\text{\quad and \quad} \dot S_{\ga,n}(0)=\big(\vert\dot\ga(0)\vert+\vert\dot \be_{n}(0)\vert\big)\ta_\ga(0).
		\end{equation*}
		Similarly, $S_{\ga,n}(1)=\ga(1)$ and  $\dot S_{\ga,n}(1)$ is a positive multiple of $\ta_\ga(1)$. This establishes (a). 
		
		Let $\phi,\psi \colon [0,1]\to [0,1]$ be as in \eqref{E:phipsi}, and set 
		\begin{equation}\label{E:AS}
			\phi_s(t)=(1-s)\phi(t)+st\quad \text{ and }\quad \psi_s(t)=(1-s)\psi(t)+st\quad (s,\,t\in [0,1]). 
		\end{equation}
		Then 
		\begin{equation*}
			(s,t)\mapsto \Phi_\ga(\phi_s(t))\be_{n}(\psi_s(t))\qquad (s,t\in [0,1])
		\end{equation*}
		defines a homotopy between $B_{\ga,n,t_0}$ and $S_{\ga,n}$ in $\sr L_{-\infty}^{+\infty}(P,Q)$. This proves (b).
		
		Part (c) follows from (a) and the fact that the curvatures of $\al_n,\be_n$ equal $\pm \tfrac{1}{2}$ a.e.. 
		
		Part (d) is a corollary of \lref{L:WG}.
		
		For part (e), let $H\colon [0,1]\times K\to \sr U_d$ be a continuous map with $H_0=f$ and $H_1=g$. Set
		\begin{equation*}
			\hat H(s,a)(t)=\Phi_{H(s,a)}(\phi(t))\be_n(\psi(t))\qquad (s,\,t\in [0,1],~a\in K).
		\end{equation*}
		Then $\hat H$ is a homotopy between $B_{f,n,t_0}=\hat H_0$ and $B_{g,n,t_0}=\hat H_1$ in $\sr U_d$.
		
		Part (f) is obvious.
\end{proof}


\begin{lem}\label{L:loopy}
	Let $f\colon K\to \sr C_{-\infty}^{+\infty}(Q)$ be continuous. Then there exists $n_0\in \N$ such that  $S_{f(a),n}\in \sr L_{-1}^{+1}(Q)$ for all $a\in K$ whenever $n\geq n_0$ \tup($n\in \N$\tup).
\end{lem}
\begin{proof}
	For $a\in K$, let $\ga_a=f(a)$ and $\ta_a=\ta_{\ga_a}$. Let $T=\se{\frac{1}{2n},\frac{2}{2n},\dots,\frac{2n-1}{2n}}$. Then:
	\begin{alignat*}{10}
		S_{\ga_a,n}(t)&=\Phi_{\ga_a}(t)\be_n(t)=\ga_a(t)+\ta_{a}(t)\be(nt)& &\qquad (t\in [0,1],~a\in K) \\
		\dot S_{\ga_a,n}(t)&=\dot \ga_a(t)+\dot\ta_a(t)\be(nt)+n\ta_a(t)\dot\be(nt) & & \qquad (t\in [0,1],~a\in K) \\
		\ddot S_{\ga_a,n}(t)&=\ddot \ga_a(t)+\ddot \ta_a(t)\be(nt)+2n\dot\ta_a(t)\dot\be(nt)+n^2\ta_a(t)\ddot \be(nt) & & \qquad (t\in [0,1]\ssm T,~a\in K)
	\end{alignat*}
Since $f\colon K\to \sr C_{-\infty}^{+\infty}(Q)$ is continuous and $K$ compact, $\vert\ga^{(k)}_a(t)\vert$ and $\vert\ta_a^{(k)}(t)\vert$ ($k=0,1,2$) are all bounded by some constant as $(t,a)$ ranges over $[0,1]\times K$. Using the third expression for the curvature in \eqref{E:curvature} and the multilinearity of the determinant, we conclude that
\begin{equation*}
	\ka_{S_{\ga_a,n}}(t)=\tfrac{1}{2}+O(\tfrac{1}{n})\quad (t\in [0,1]\ssm T,~a\in K),
\end{equation*}
where $O(\frac{1}{n})$ is a function of $(t,a)$ such that $n\abs{O(\tfrac{1}{n})}$ is uniformly bounded over $([0,1]\ssm T)\times K$ as $n$ ranges over $\N$. It follows that $S_{\ga_a,n}\in \sr L_{-1}^{+1}(Q)$ for all sufficiently large $n$.
\end{proof}

\begin{lem}\label{L:LLoopy}
	Let $f\colon K\to \sr C_{-1}^{+1}(Q)$ be continuous, $t_0\in (0,1)$. Then for all sufficiently large $n\in \N$, there exists a continuous $H\colon [0,1]\times K\to \sr L_{-1}^{+1}(Q)$ with $H_0=B_{f,n,t_0}$ and $H_1=S_{f,n}$. 
\end{lem}
\begin{proof}Let $H$ be given by
\begin{equation*}
	H(s,a)(t)=\Phi_{f(a)}(\phi_s(t))\be_{n}(\psi_s(t))\qquad (s,t\in [0,1],~a\in K),
\end{equation*}
where $\phi_s$, $\psi_s$ are as in \eqref{E:AS}. Then $H(0,a)=B_{f(a),n,t_0}$ and $H(1,a)=S_{f(a),n}$. A computation entirely similar to the one in the proof of \lref{L:loopy} establishes that $H(s,a)\in \sr L_{-1}^{+1}(Q)$ for all $s\in [0,1]$, $a\in K$ if $n$ is sufficiently large. The details will be left to the reader, but to make things easier, notice that $\phi_s,\psi_s$ are piecewise linear for all $s\in [0,1]$, so that $\ddot \psi_s=\ddot \phi_s=0$ except at a finite set of points (which depends on $s$).
\end{proof}

The next result provides a sufficient condition, which does not involve $g$, for one to be able to write a compact family of curves $f$ as $f=A_{g,n,t_0}$.

\begin{lem}\label{L:implicit}
	Let $X$ be a compact Hausdorff topological space and $f\colon X\to \sr U_d\subs \sr L_{-1}^{+1}(Q;\theta_1)$, $t_0\colon X\to (0,1)$ be continuous maps. Then it is possible to reparametrize each $f(a)$ \tup{(}continuously with $a$\tup{)} and find a continuous $g\colon X\to \sr L_{-1}^{+1}(Q;\theta_1)$ so that $f(a)=A_{g(a),n,t_0(a)}$ for all $a\in X$ if and only if there exists a continuous function $\eps\colon X\to (0,1)$ such that, for all $a\in X$:
	\begin{enumerate}
		\item [(i)] $0<t_0(a)-\eps(a)<t_0(a)+\eps(a)<1$;
		\item [(ii)]  $f(a)|_{[t_0(a)-\eps(a),t_0(a)+\eps(a)]}$ is some parametrization of $\Phi_{f(a)}(t_0(a)-\eps(a))\al_n$.
	\end{enumerate}
\end{lem}
\begin{proof}
	Suppose that such a function $\eps\colon X\to (0,1)$ exists. Since $X$ is compact, we may reparametrize all $f(a)$ so that $\eps$ becomes a constant function and, for all $a\in X$, satisfies:
	\begin{enumerate}
		\item [(I)] $0<t_0(a)-2\eps<t_0(a)+2\eps<1$;
		\item [(II)] $f(a)|_{[t_0(a)-\eps,t_0(a)+\eps]}$ is a parametrization of  $\Phi_{f(a)}(t_0(a)-\eps)\al_n$ by a multiple of arc-length.
	\end{enumerate}
Define $g\colon X\to \sr L_{-1}^{+1}(Q)$ by:
	\begin{equation*}
		g(a)(t)=\begin{cases}
		f(a)(t) & \text{\quad if $t\nin [t_0(a)-2\eps,t_0(a)+2\eps]$} \\
		f(a)\big(\frac{1}{2}(t+t_0(a)-2\eps)\big) & \text{\quad if $t\in [t_0(a)-2\eps,t_0(a)]$} \\
		f(a)\big(\frac{1}{2}(t+t_0(a)+2\eps)\big) & \text{\quad if $t\in [t_0(a),t_0(a)+2\eps]$} 
		\end{cases}\quad (a\in X,~t\in [0,1]).
	\end{equation*}
	Then $g$ is continuous because $f$ and $t_0$ are continuous, and $f(a)=A_{g(a),n,t_0(a)}$ for all $a\in X$. This proves the ``if'' part of the lemma. The converse is obvious. 
\end{proof}

As a simple application of \lref{L:loopy}, we prove that this article is not a study of the empty set.

\begin{lem}\label{L:nonempty}
	Let $\ka_1<\ka_2$, $P=(p,w),\,Q=(q,z)\in \C\times \Ss^1$ and let $\theta_1\in \R$ satisfy $e^{i\theta_1}=z\bar w$. Then:
	\begin{enumerate}
		\item [(a)] $\sr L_{\kappa_1}^{\kappa_2}(P,Q)\neq \emptyset$. 
		\item [(b)]  $\sr L_{\kappa_1}^{\kappa_2}(P,Q;\theta_1)\neq \emptyset$ if $\ka_1\ka_2<0$.
		\item [(c)] If $\ka_1<\ka_2\leq 0$, then $\sr L_{\kappa_1}^{\kappa_2}(P,Q;\theta_1)=\emptyset$ for all sufficiently large $\theta_1$. If $0\leq \ka_1<\ka_2$, then $\sr L_{\kappa_1}^{\kappa_2}(P,Q;\theta_1)=\emptyset$ for all sufficiently small $\theta_1$.
	\end{enumerate}
\end{lem}
\begin{proof}
	By \cref{C:normalized}, we need only consider spaces of the form $\sr L_{-1}^{+1}(Q)$, $\sr L_{0}^{+\infty}(Q)$ and $\sr L_{1}^{+\infty}(Q)$. It is clear that $\sr C_{-\infty}^{+\infty}(Q)\neq \emptyset$ for all $Q\in UT\C$. Let $\ga\in \sr C_{-\infty}^{+\infty}(Q)$ be arbitrary. 
	
	By \lref{L:loopy}, if $n$ is sufficiently large, then $S_{\ga,n}\in \sr L_{-1}^{+1}(Q)$. Furthermore, attaching loops (possibly with reversed orientation) to $S_{\ga,n}$, we can obtain a curve in $\sr L_{-1}^{+1}(Q;\theta_1)$ for any $\theta_1\in \R$ satisfying $e^{i\theta_1}=z$. This proves (b), and also part (a) when $\ka_1\ka_2<0$.
	
	Similarly, define a curve $\bar S_{\ga,n}$ by $\bar S_{\ga,n}(t)=\Phi_{\ga}(t)\big(\frac{1}{4}\al_n(t)\big)$ ($t\in [0,1]$). In words, $\bar S_{\ga,n}$ is obtained from $\ga$ by spreading $n$ loops of radius $\frac{1}{2}$, instead of $n$ eights of radius 2. Using an argument analogous to the one which established \lref{L:loopy}, one sees that $\bar S_{\ga,n}\in \sr L_{1}^{+\infty}(Q)$ for all sufficiently large $n$. This completes the proof of (a).
	
	To see that $\sr L_{\kappa_1}^{\kappa_2}(P,Q;\theta_1)$ may be empty if $\ka_1\ka_2\geq 0$, we use eq.~\eqref{E:rate}: If $\ka_1,\ka_2$ are both non-negative, for example, then $\sr L_{\kappa_1}^{\kappa_2}(P,Q)$ can only contain curves having positive total turning.
\end{proof}

\begin{urem}
	Invoking \lref{L:smoothie}, we obtain a version of \lref{L:nonempty} with $\sr C$ in place of $\sr L$. 
\end{urem}

\begin{cor}\label{C:top}
	Let $\sr U_d$ denote the subset of $\sr L_{-1}^{+1}(Q;\theta_1)$ consisting of all diffuse curves, where $Q=(q,z)$ and $e^{i\theta_1}=z$. Then $\sr U_d\neq \emptyset$.
\end{cor}	 
\begin{proof}
	Lemma \lref{L:basic}\?(c) implies that $B_{\ga,1,\frac{1}{2}}\in \sr U_d$ for any $\ga\in \sr L_{-1}^{+1}(Q;\theta_1)$. Since the latter is nonempty by \lref{L:nonempty}\?(b), so is $\sr U_d$.
\end{proof}

\begin{thm}[Dubins]\label{P:Dubins}
	Let $x>0$, $Q=(x,1)$ and $\eta \in \sr L_{-1}^{+1}(Q)$ be the line segment $\eta\colon t\mapsto xt$. Then $\eta$ and $B_{\eta,1,\frac{1}{2}}$ lie in the same component of $\sr L_{-1}^{+1}(Q)$ if and only if $x>4$.
\end{thm}
\begin{proof}
	See \cite{Dubins1}, Theorem 5.3.
\end{proof}

The next construction provides a homotopy of the straight line segment $[0,x]$ to the same segment with an eight attached which is continuous with respect to $x$.
\begin{cons}\label{C:eighting}
	For $x>0$, let $\eta_x\colon [0,x]\to \C$ be the line segment $t\mapsto t$. Take $t_0=\frac{1}{2}$ in \eqref{E:phipsi} and let $h\colon [0,1]\times [0,6]\to \C$ be a fixed homotopy between $h_0=\eta_6$ and
	\begin{equation*}
		h_1=\Phi_{\eta_6}\Big(6\phi(\tfrac{t}{6})\Big)\be_1\big(\psi(\tfrac{t}{6})\big)\quad \text{($\eta_6$ with an eight attached)}
	\end{equation*}
	such that $t\mapsto h_s(6t)$ ($t\in [0,1]$) is a curve in $\sr L_{-1}^{+1}(Q)$ for all $s\in [0,1]$. The existence of $h$ is guaranteed by \pref{P:Dubins}. Let $\mu\colon [0,+\infty)\to [0,1]$ be a smooth function such that $\mu(x)=0$ if $x\in [0,6]$ and $\mu(x)=1$ if $x\geq 8$. Define a family of curves $\eta_x^u\colon [0,1]\to \C$ by:
	\begin{equation}\label{E:eighting}
		\eta_x^u(t)=\begin{cases}
		\eta_x(t) & \text{\quad if $t\geq 6$ or $x\leq 6$} \\
		h(u\mu(x),t) & \text{\quad if $t\leq 6$ and $x\geq 6$}
		\end{cases}\qquad (u\in [0,1],~t\in [0,x],~x>0).
	\end{equation}
	Of course, $\eta_x^0=\eta_x$ for all $x>0$. If $x\geq 8$, then $\eta_x^1$ equals $\eta_x$ with an eight attached; in particular, $\eta_x^1|_{[3-6\eps,3]}$ is a loop.
\end{cons}

\subsection*{Grafting}We now explain how to graft straight line segments onto a diffuse curve (see Figure \ref{F:grafting}).
\begin{defn}\label{D:graft}
Let $\ga\in \sr L_{-1}^{+1}(Q)$ be a curve of length $L$ parametrized by arc-length, $\sig_i\geq 0$ and $s_i\in [0,1]$, $i=1,\dots,2n$, where the $s_i$ form a non-decreasing sequence. Suppose that there exists a bijection $p$ of  $\se{1,\dots,2n}$ onto itself such that, for each $i$:
\begin{enumerate}
	\item [($\ast$)] $\sig_{p(i)}=\sig_{i}$ and $\ta_{\ga}(s_{p(i)})=-\ta_\ga(s_i)$.
\end{enumerate}
Then we define the \tdef{graft} $G_\ga=G_{\ga,(s_i),(\sig_i)}\colon \big[0,L+\sum_{i=1}^{2n}\sig_i\big]\to \C$ by:
\begin{equation}\label{E:graft}
	G_\ga(s)=\begin{cases}
	\ga(s) & \text{ if\ \ $s\in [0,s_1]$ } \\
	\ga(s_1)+(s-s_1)\ta_{\ga}(s_1) & \text{ if\ \ $s\in [s_1,s_1+\sig_1]$ }\\
	\ga(s-\sig_1)+\sig_1\ta_{\ga}(s_1) & \text{ if\ \ $s\in [s_1+\sig_1,s_2+\sig_1]$}\\
	\ga(s_2)+\sig_1\ta_\ga(s_1)+(s-s_2-\sig_1)\ta_{\ga}(s_2) & \text{ if\ \ $s\in [s_2+\sig_1,s_2+\sig_1+\sig_2]$ }\\
	\phantom{\ga(s-\sig_1)}\vdots & \phantom{\text{if $s\in a$}}\vdots \\
	\ga(s-\sum_{i=1}^{2n}\sig_i)+\sum_{i=1}^{2n}\sig_i\ta_{\ga}(s_i) & \text{ if\ \ $s\in [s_{2n}+\sum_{i=1}^{2n}\sig_i,L+\sum_{i=1}^{2n}\sig_i]$}
	\end{cases}
\end{equation}
Although it simplifies the previous formula, the assumption that $(s_i)$ is a non-decreasing sequence is not necessary for the construction to work, since we may always relabel the $s_i$.
\end{defn}

\begin{figure}[ht]
	\begin{center}
		\includegraphics[scale=.195]{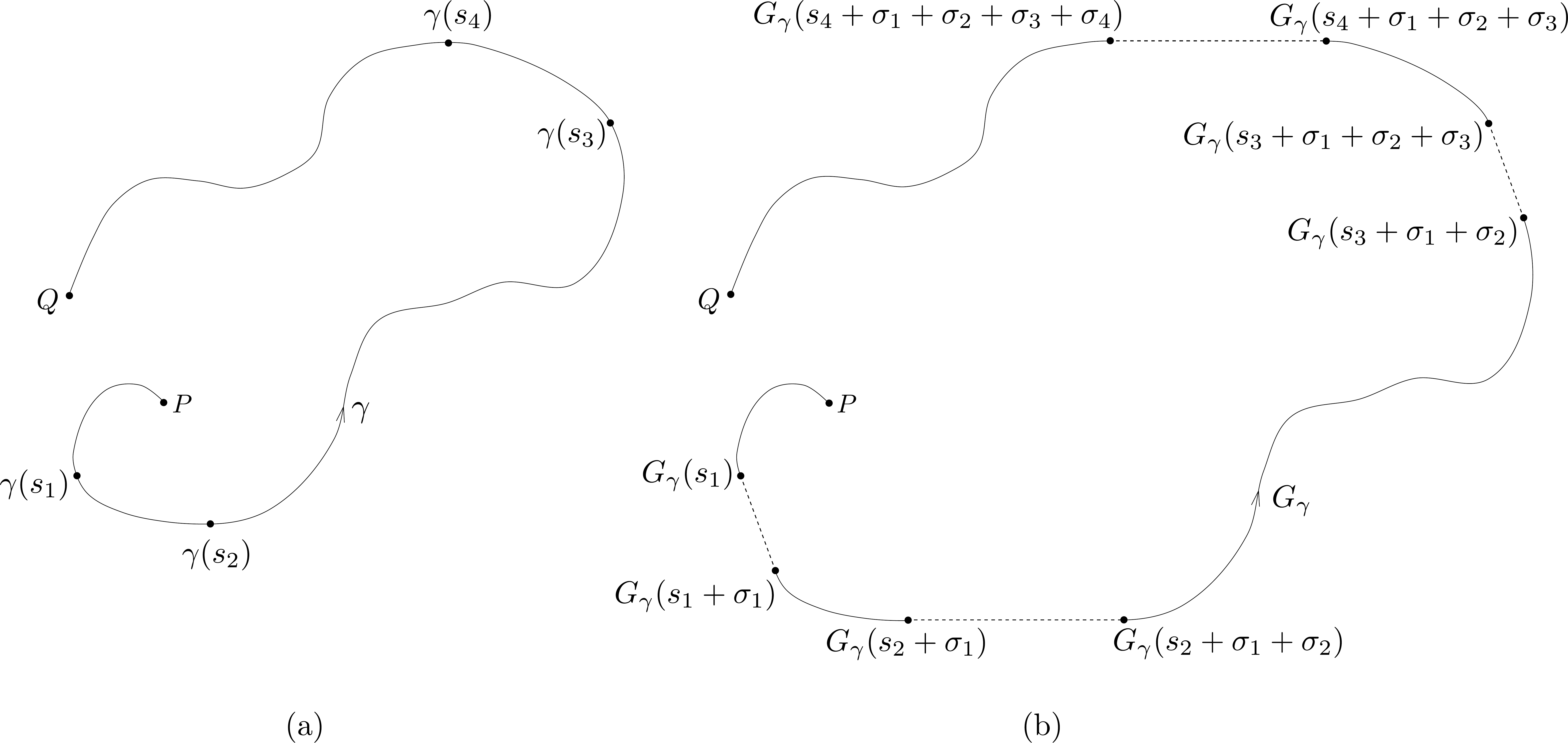}
		\caption{A diffuse curve $\ga$ and its graft $G_\ga=G_{\ga,(s_1,s_2,s_3,s_4),(\sig_1,\sig_2,\sig_3,\sig_4)}$.}
		\label{F:grafting}
	\end{center}
\end{figure}

\begin{lem}
	Let $\ga\in \sr L_{-1}^{+1}(Q)$ be diffuse and $G_\ga$ be as in \dref{D:graft}. Then $G_\ga$ is parametrized by arc-length and it lies in the same connected component of $\sr L_{-1}^{+1}(Q)$ as $\ga$.
\end{lem}
\begin{proof}
	It is obvious from \eqref{E:graft} that $\Phi_{G_\ga}(0)=\Phi_{\ga}(0)$. Looking at the last line of \eqref{E:graft} and using ($\ast$) we deduce that 
\begin{equation*}
	G_\ga(s)=\ga\Big(s-{\textstyle\sum_{i=1}^{2n}}\sig_i\Big)  \text{\ \ for\ \ $\textstyle  s\in \big[s_{2n}+\sum_{i=1}^{2n}\sig_i,L+\sum_{i=1}^{2n}\sig_i\big]$}.
\end{equation*}
Hence, $\Phi_{G_\ga}(L+\sum_{i=1}^{2n}\sig_i)=\Phi_\ga(L)$. Since $G_\ga$ is made up of line segments and arcs of $\ga$ (composed with translations), $G_\ga\in \sr L_{-1}^{+1}(Q)$. It is clear that $G_\ga$ is parametrized by arc-length. Finally, 
\begin{equation*}
\qquad	u\mapsto G_{\ga,(s_i),(u\sig_i)}\qquad (u\in [0,1])
\end{equation*}
defines a path in $\sr L_{-1}^{+1}(Q)$ joining $\ga$ to $G_\ga$.
\end{proof}	

\subsection*{Contractibility of $\sr U_d$} Recall that $K$ denotes a compact manifold, possibly with boundary.
\begin{lem}\label{L:L1}
	Let $f\colon K\to \sr U_d$ be continuous. Then there exist an open cover $(V_j)_{j=1}^m$ of $K$ and continuous maps $\tau_{j}^{\pm}\colon K\to (0,1)$, $f_1\colon K\to \sr U_d$ such that:
	\begin{enumerate}
		\item [(i)] $f\iso f_1$ within $\sr U_d$ and $f_1$ satisfies conditions (ii) and (iii) of \lref{L:smoothie}.
		\item [(ii)] $\ta_{f_1(a)}(\tau_j^+(a))=-\ta_{f_1(a)}(\tau_j^-(a))$ whenever $a\in V_j$.
	\end{enumerate}
\end{lem}

\begin{proof}
	Apply \lref{L:smoothie} to $f$ and $\sr U_d$ to obtain $f_1$. The idea is to use the implicit function theorem to find $\tau_{j}^{\pm}$. However, some care must be taken since $f_1$ need not be differentiable with respect to $a$. 
	
	For each $a\in K$, let $\theta_a\colon [0,1]\to \R$ be the argument of $\ta_{f_1(a)}$ satisfying $\theta_a(0)=0$, and set
	\begin{equation*}
		\vphi_a=\frac{1}{2}\Big(\sup_{t\in [0,1]}\theta_a(t)+\inf_{t\in [0,1]}\theta_a(t)\Big).
	\end{equation*}
	Because each $\ga_a$ is diffuse and $K$ is compact, we can find $\de>0$ such that 
	\begin{equation*}
		\theta_a([0,1])\sups \Big(\vphi_a-\frac{\pi}{2}-\de\,,\,\vphi_a+\frac{\pi}{2}+\de\Big)\text{\quad for all $a\in K$}.	
	\end{equation*}
	Fix $a_0\in K$. By Sard's theorem, we can find $\psi\in (\vphi_{a_0}+\frac{\pi}{2},\vphi_{a_0}+\frac{\pi}{2}+\de)$ such that both $\psi$ and $\psi-\pi$ are regular values of $\theta_{a_0}$. Let $\tau^{\pm}(a_0)\in (0,1)$ satisfy $\theta_{a_0}(\tau^{+}(a_0))=\psi$ and $\theta_{a_0}(\tau^{-}(a_0))=\psi-\pi$. No generality is lost in assuming that $\dot\theta_{a_0}(\tau^+(a_0))>0$. From eq.~\eqref{E:rate}, $\dot\theta_a=\abs{\dot\ga_{f_1(a)}}\ka_{f_1(a)}$. Thus, $\dot \theta_a$ depends continuously on $a$, so we can find $\mu,\eps>0$ and a compact neighborhood $V\subs K$ of $a_0$ such that 
	\begin{equation*}
		\psi\in \theta_a((\tau^+(a_0)-\eps,\tau^+(a_0)+\eps))\ \ \text{and}\ \ \dot\theta_{a}(t)>\mu \ \ \text{whenever $a\in V$,~$\abs{t-\tau^+(a_0)}<\eps$.}
	\end{equation*}
	Hence, for each $a\in V$, there exists a \tsl{unique} $\tau^+(a)\in (\tau^+(a_0)-\eps,\tau^+(a_0)+\eps)$ with $\theta_a(\tau^+(a))=\psi$. We claim that the function $\tau^+\colon V\to (0,1)$ so defined is continuous. Consider the equation
	\begin{equation*}
		\theta_a(\tau^+(b))-\theta_a(\tau^+(a))=\big[\theta_b(\tau^+(b))-\theta_a(\tau^+(a))\big]+\big[\theta_a(\tau^+(b))-\theta_b(\tau^+(b))\big]\ \ (a,b\in V).
	\end{equation*}
	The first term on the right side equals 0 by the definition of $\tau^+$, and the second converges to $0$ as $b\to a$ since $\theta_b(t)$ is a uniformly continuous function of $(b,t)\in K\times [0,1]$. Hence, by the mean-value theorem, 
	\begin{equation*}
		\abs{\tau^+(b)-\tau^+(a)}< \frac{1}{\mu}\abs{\theta_a(\tau^+(b))-\theta_a(\tau^+(a))} \to 0\text{\ \ as \ \ $b\to a\quad (a,b\in V)$}.
	\end{equation*}
	It follows that $\tau^+$ is continuous. Similarly, reducing $V$ if necessary, we can find a continuous function $\tau^-\colon V\to (0,1)$ with $\theta_a(\tau^-(a))=\psi-\pi$ for all $a\in V$. To finish the proof, cover $K$ by finitely many such compact neighborhoods $V_j$, let $\tau^{\pm}_j\colon V_j\to (0,1)$ be the corresponding functions and extend each $\tau^{\pm}_j$ to $K$ using the Tietze extension theorem.
\end{proof}

\begin{lem}\label{L:L2}
	Let $f\colon K\to \sr U_d$ be continuous. Then there exist an open cover $(W_j)_{j=1}^m$ of $K$ and continuous maps $t_{j}\colon K\to (0,1)$, $g_j\colon W_j\to \sr L_{-1}^{+1}(Q)$ and $f_2\colon K\to \sr U_d$ such that:
	\begin{enumerate}
		\item [(i)] $f\iso f_2$ within $\sr U_d$;
		\item [(ii)] $f_2(a)=A_{g_j(a),1,t_j(a)}$ for all $a\in W_j$. 
	\end{enumerate}
\end{lem}

\begin{proof}
	Take $f_1$ as in \lref{L:L1}. By \cref{C:reparametrize}, we may assume that each $\ga_a=f_1(a)\colon [0,L_a]\to \C$ is parametrized by arc-length, so that now $\tau_j^{\pm}(a)\in (0,L_a)$ for each $a$. Let $( \la_j)_{j=1}^m$ be a partition of unity subordinate to $(V_j)_{j=1}^m$, with $V_j$ as in \lref{L:L1}. Set $\sig_j=10m\la_j$ and $W_j=\set{a\in K}{\sig_j(a)>8}$. Then $\overline{W}_j\subs V_j$ and the $W_j$ form an open cover of $K$. Define 
	\begin{equation*}
		\ga_a^u=G_{\ga_a,(\tau_1^-(a),\dots,\tau_m^-(a),\tau_1^+(a),\dots,\tau_m^+(a))\?,\?(u\sig_1(a),\dots,u\sig_m(a),u\sig_1(a),\dots,u\sig_m(a))}\quad (u\in [0,1],~a\in K),
	\end{equation*}
	as in \dref{D:graft}. Let us suppose that $\tau_1^-\leq \dots\leq \tau_m^-(a)\leq \tau_1^+(a)\leq \dots\leq \tau_m^+(a)$ for each $a$ to abbreviate the notation, and set
	\begin{equation*}
		\xi_j^-(a)=\sum_{i<j} \sig_i(a)\text{\ \ and \ \ }\xi_j^+(a)=10m+\sum_{i<j}\sig_i(a)\quad (a\in K,~j=1,\dots,m).
	\end{equation*}
	Then 
	\begin{equation*}
		\ga_a^1\big([\tau_j^-(a)+\xi_j^-(a),\tau_j^-(a)+\xi_j^-(a)+\sig_j(a)]\big)
	\end{equation*} is a line segment, corresponding to the graft at $\ga_a(\tau_j^-(a))$. Its length $\sig_j(a)$ is at least 8 if $a\in \ol{W}_j$. Of course, the same statements hold with $^+$ instead of $^-$.  We obtain $f_2$ by deforming all of these segments to eights. More precisely, for $u\in [1,2]$ and $a\in K$, let
	\begin{equation*}
		\ga_a^u(s)=\begin{cases}
		\Phi_{\ga_a^1}(\tau_j^{\pm}(a)+\xi^{\pm}_j(a))\eta_{\sig_j(a)}^{u-1}(s-\tau^{\pm}_j(a)-\xi^{\pm}_j(a)) & \\
		\ga_a^1(s) &
		\end{cases}\ (s\in [0,L_a+20m])
	\end{equation*}
	according as  $s\in [\tau^{\pm}_j(a)+\xi^{\pm}_j(a)\?,\?\tau^{\pm}_j(a)+\xi^{\pm}_j(a)+\sig_j(a)]$ for some $j$ or not, respectively. Here $\eta_x^u$ is as in \cref{C:eighting}. Let $f_2\colon K\to \sr U_d$ be given by $f_2(a)=\ga_a^2$. Note that
	\begin{equation*}
		\ga_a^2\big([\tau^{\pm}_j(a)+\xi^{\pm}_j(a)+3-6\eps\?,\?\tau^{\pm}_j(a)+\xi^{\pm}_j(a)+3]\big)\qquad (j=1,\dots,m)
	\end{equation*}
	is a loop whenever $a\in \ol{W}_j$. Thus (after reparametrizing the $\ga_a^2$ so that their domains become $[0,1]$) we may apply \lref{L:implicit} to each family $f_2|_{\ol{W}_j}$ to find $g_j\colon \ol{W}_j\to \sr L_{-1}^{+1}(Q)$ and $t_j\colon \ol{W}_j \to (0,1)$ such that 
	\begin{equation*}
		f_2(a)=A_{g_j(a),1,t_j(a)}\text{\ \ for all $a\in W_j$}. 
	\end{equation*}
	The functions $t_j$ may be extended to all of $K$ by the Tietze extension theorem. 
\end{proof}

\begin{lem}\label{L:L3}
	Let $f\colon K\to \sr U_d$ be continuous. Suppose that there exists a covering of $K$ by open sets $W_j$ and continuous maps $t_j\colon K\to (0,1)$, $g_j\colon W_j\to \sr L_{-1}^{+1}(Q)$ with $f(a)=A_{g_j(a),1,t_j(a)}$ whenever $a\in W_j$, $j=1,\dots,m$. Then there exist continuous $g\colon K\to \sr L_{-1}^{+1}(Q)$ and $H\colon [0,1]\times K\to \sr U_d$ with $H_0=f$ and $H_1=A_{g,1,\frac{1}{2}}$.
\end{lem}

\begin{proof}
	The proof will be by induction on $m$. If $m=1$ then $W_1=K$, and $H$ just slides the loop from $t_1$ to $\frac{1}{2}$:
	\begin{equation*}
		H(s,a)=A_{g_1(a),1,\?(1-s)t_1(a)+\frac{s}{2}}\quad (s\in [0,1],~a\in K).
	\end{equation*}
	Suppose now that $m>1$. Let $W$ be an open set such that $\ol{W}\subs W_m$ and $W_1\cup \dots W_{m-1}\cup W=K$. Let $\la\colon K\to [0,1]$ be a continuous function such that $\la(a)=1$ if $a\in W$ and $\la(a)=0$ if $a\nin W_m$. Define $\hat H\colon [0,1]\times K\to \sr U_d$ by: 
	\begin{equation*}
		\hat H(s,a)=\begin{cases}
		A_{g_m(a),1,\?(1-\la(a)s)t_m(a)+\la(a)st_{m-1}(a)} & \text{\ \ if $a\in W_m$} \\
		f(a) & \text{\ \ if $a\nin W_m$ }
		\end{cases} \quad (s\in [0,1],~a\in K).
	\end{equation*}
	Then the induction hypothesis applies to $\hat H_1\colon K\to \sr U_d$, the open sets $\hat W_i=W_i$ ($i=1,\dots,m-2$) and $\hat W_{m-1}=W_{m-1}\cup W$, and the same functions $t_j$ as before, $j=1,\dots,m-1$. The existence of $\hat g_{m-1}$ as in the statement is guaranteed by \lref{L:implicit}: Using \lref{L:basic}\?(f), we deduce that there is at least one loop at $t_{m-1}(a)$ for $a\in \hat W_{m-1}$.
\end{proof}

\begin{prop}\label{P:diffuse}
	Let $f\colon K\to \sr U_d$ be continuous. Then $f\iso B_{f,n,\frac{1}{2}}$ within $\sr U_d$ for all $n\geq 1$. 
\end{prop}
\begin{proof}
	 Applying \lref{L:L2} and \lref{L:L3} to $f$, we obtain continuous maps $g\colon K\to \sr L_{-1}^{+1}(Q)$ and $h\colon K\to \sr U_d$ such that $f\iso h$ in $\sr U_d$ and 
\begin{equation*}
	h(a)=A_{g(a),1,\frac{1}{2}} \text{\quad for all $a\in K$}.
\end{equation*} 
Using \lref{L:basic}\?(d), we may deform the loop at $t=\frac{1}{2}$ to attach $n$ eights to $h$ at $t=\frac{1}{2}$ (for arbitrary $n\geq 1$). Thus $h\iso B_{h,n,\frac{1}{2}}$. Together with \lref{L:basic}\?(e), this implies that $f\iso B_{f,n,\frac{1}{2}}$ within $\sr U_d$.
\end{proof}

\begin{thm}\label{T:diffuse}
	Let $Q=(q,z)\in \C\times \Ss^1$ and $\theta_1\in \R$ satisfy $e^{i\theta_1}=z$. Then the subspace $\sr U_d\subs \sr L_{-1}^{+1}(Q;\theta_1)$ consisting of all diffuse curves is homeomorphic to $\E$, hence contractible.
\end{thm}
\begin{proof}
	Because $\sr U_d$ is open, it suffices to prove that it is weakly contractible, by \lref{L:Hilbert}\?(b). Let $k\in \N$, $f\colon \Ss^k\to \sr U_d$ be continuous and $g\colon \Ss^k\to \sr U_d$ be a map satisfying (i)--(iii) of \lref{L:smoothie} (with $\sr U=\sr U_d$). By \tref{T:Smale}, there exists $G\colon [0,1]\times \Ss^k\to \sr C_{-\infty}^{+\infty}(Q)$ such that $G_0=g$ and $G_1$ is a constant map. By \lref{L:loopy}, there exists $n_0\in \N$ such that if $n\geq n_0$, then $S_{G(s,a),n}\in \sr U_d$ for all $s\in [0,1]$,~$a\in \Ss^k$. Applying \pref{P:diffuse} and \lref{L:LLoopy} to $g$, we obtain $n_1\geq n_0$ and a continuous $F\colon [-1,0]\times \Ss^k\to \sr U_d$ with $F_{-1}=g$ and $F_0=S_{g,n_1}$. Concatenating $F$ and $S_{G,n_1}$ we obtain a null-homotopy of $g$ in $\sr U_d$.
\end{proof}



\section{Critical curves}\label{S:critical}

Fix $Q=(q,z)\in \C\times \Ss^1$ and $\theta_1\in \R$ satisfying $e^{i\theta_1}=z$. Let $\ga\in \sr L_{-1}^{+1}(Q;\theta_1)$ and $\theta\colon [0,1]\to \R$ be the argument of $\ta_\ga$ satisfying $\theta(0)=0$. Finally, let 
\begin{equation}\label{E:amp2}
	\theta^+=\sup_{t\in [0,1]}\theta(t)\quad \text{and}\quad \theta^-=\inf_{t\in [0,1]}\theta(t).
\end{equation}
Recall that $\ga$ is called \tdef{critical} if $\theta^+-\theta^-=\pi$.  A curve $\eta\in \sr L_{-1}^{+1}(Q;\theta_1)$ must be either condensed, diffuse or critical. It has already been shown that the subspaces $\sr U_c$ and $\sr U_d$ of $\sr L_{-1}^{+1}(Q;\theta_1)$ consisting of all condensed (resp.~diffuse) curves are both contractible. Let $\sr T\subs \sr L_{-1}^{+1}(Q;\theta_1)$ denote the subspace of all critical curves. Clearly, $\sr T$ is closed as the complement of $\sr U_c\cup \sr U_d$. Since the difference $\theta^+-\theta^-$ depends continuously on $\ga$, we deduce that $\bd \sr U_c\subs \sr T$ and $\bd \sr U_d \subs \sr T$, where $\bd \sr U_c$ denotes the topological boundary of $\sr U_c$ considered as a subspace of $\sr L_{-1}^{+1}(Q;\theta_1)$ and similarly for $\sr U_d$.

\begin{prop}\label{P:boundary}
Let $\abs{\theta_1}<\pi$ and $\sr U_c,\,\sr U_d,\,\sr T\subs \sr L_{-1}^{+1}(Q;\theta_1)$ be as above. Then $\bd{\sr U}_c=\bd{\sr U}_d=\crit$. Therefore, $\bar{\sr U}_c\cup \bar{\sr U}_d=\sr L_{-1}^{+1}(Q;\theta_1)$ and  $\bar{\sr U}_c\cap \bar{\sr U}_d=\sr T$.
\end{prop}
Observe that $\sr T=\emptyset$ if $\abs{\theta_1}>\pi$ and $\sr U_c=\emptyset$  if $\abs{\theta_1}\geq \pi$. However, in any case $\bd \sr U_d=\sr T$. 
\begin{proof}
	Let $\ga\in \sr L_{-1}^{+1}(Q;\theta_1)$ be a critical curve and $\sr V\subs \sr L_{-1}^{+1}(Q;\theta_1)$ be an open set containing $\ga$. Let $\theta$ be the argument of $\ta_\ga$ satisfying $\theta(0)=0$, and let $\theta^+$, $\theta^-$ be as in \eqref{E:amp2}.
	
	We first prove that $\sr V\cap \sr U_c\neq \emptyset$. Our immediate objective is to replace $\ga$ with another curve in $\sr V\cap \sr T$ having smaller curvature. Choose $t_1\in (0,1)$ and $\de>0$ such that $\theta(t)\in (\theta^-,\theta^+)$ for all $t\in [t_1-\de,t_1]$.  Let $Q_0=\Phi_\ga(t_1-\de)$, $Q_1=\Phi_\ga(t_1)$ and consider the map 
	\begin{equation*}
		F\colon \sr L_{-1}^{+1}(Q_0,\cdot)\to UT\C,\quad F(\eta)=\Phi_\eta(1).
	\end{equation*}
	(Recall that $\sr L_{-1}^{+1}(Q_0,\cdot)$ consists of all $(-1,1)$-admissible curves having initial frame equal to $Q_0$ and arbitrary final frame.) By \lref{L:submersion}, $F$ is an open map. It follows that for any $\te{Q}_1$ close enough to $Q_1$, we can find $\eta\in \sr L_{-1}^{+1}(Q_0,\te{Q}_1)$ such that
	\begin{equation}\label{E:thetaeta}
		\theta_\eta([0,1])\subs (\theta^-,\theta^+).
	\end{equation}
	Let $Q_1=(q_1,z_1)$ and $Q=(q,z)$. Since $\ga$ is critical, the image of $\ta_\ga$ is contained in a semicircle. Consequently, $q\neq 0$. Choose $\ka_0\in (0,1)$ close to 1.  Replace the arc $\ga|_{[t_1-\de,t_1]}$ by a curve $\eta$ as above with $\te{Q}_1=(q_1+(\ka_0-1)q,z_1)$, and the arc $\ga|_{[t_1,1]}$ by its translate $\ga|_{[t_1,1]}+(\ka_0-1)q$. Let $\ga_1$ be the resulting curve; observe that $\ga_1$ is critical, $\Phi_{\ga_1}(0)=(0,1)$ and $\Phi_{\ga_1}(1)=(\ka_0q,z)$. Set $\ga_2=\frac{1}{\ka_0}\ga_1$ (that is, $\ga_2(t)$ is obtained from $\ga_1(t)$ by a dilatation through a factor of $\frac{1}{\ka_0}$ for all $t\in [0,1]$). Then 
	\begin{equation*}
		\Phi_{\ga_2}(0)=(0,1),\ \ \Phi_{\ga_2}(1)=(q,z),\ \ \ta_{\ga_2}(t)=\ta_{\ga_1}(t)\ \ \text{and\ \ }\ka_{\ga_2}(t)=\ka_0\ka_{\ga_1}(t)\ \ \text{for all $t\in [0,1]$}.
	\end{equation*}
	Thus, $\ga_2$ is a critical curve in $\sr L_{-1}^{+1}(Q;\theta_1)$ whose curvature is constrained to $(-\ka_0,\ka_0)$. Moreover, if $\ka_0$ is close enough to $1$ and $\eta$ is chosen appropriately, we can guarantee that $\ga_2\in \sr V$. 
	
	Having established the existence of $\ga_2$ with these properties, let us return to the beginning, setting $\ga=\ga_2\in \sr V$. Since $\abs{\theta_1}<\pi$, either 
	\begin{equation*}
		\theta^{-1}(\se{\theta^{-}})\cap \se{0,1}=\emptyset \quad \text{or}\quad \theta^{-1}(\se{\theta^{+}})\cap \se{0,1}=\emptyset,
	\end{equation*}
	and we lose no generality in assuming the latter. Choose $\eps>0$ small enough to guarantee that
	\begin{equation*}
		W=\theta^{-1}\big((\theta^+-\eps,\theta^+]\big)\subs (0,1).
	\end{equation*}
	Cover $\theta^{-1}(\se{\theta^+})$ by the finite union of disjoint intervals $(a_i,b_i)\subs W$ with $\theta(a_i)=\theta(b_i)=\theta^+-\eps$, $i=1,\dots,m$. Let $P_i=\Phi_{\ga}(a_i)$, $Q_i=\Phi_\ga(b_i)$. We can obtain a curve in $\sr U_c\cap \sr V$ by modifying $\ga$ in each of these intervals to avoid the argument $\theta^+$ using \pref{P:excavator}: Note that $P_i^{-1}\ga|_{[a_i,b_i]}$ satisfies the hypotheses of \lref{L:compress} because it has curvature in the open interval $(-\ka_0,+\ka_0)$ and is not a line segment. Moreover, the inclusion $\hat{\sr L}_{-\ka_0}^{+\ka_0}(P_i^{-1}Q_i)\to \sr L_{-1}^{+1}(P_i^{-1}Q_i)$ is continuous by \lref{L:inclusions}.
	
	The proof that $\sr V\cap \sr U_d\neq \emptyset$ is easier. Let the critical curve $\ga\colon [0,L]\to \C$ be parametrized by arc-length. Then we can find $s_0,\,s_1\in [0,1]$ with $\ta_\ga(s_0)=-\ta_\ga(s_1)$. Choose $\eps>0$ and let 
	\begin{equation*}
		G_\ga=G_{\ga,(s_0,s_1),(\eps,\eps)}.
	\end{equation*}
	(See definition \dref{D:graft} and Figure \ref{F:grafting}.) Choose $\ka_0\in (0,1)$ and construct a curve $\zeta\in \sr L_{-1}^{+1}(Q;\theta_1)$ by replacing the line segment $G_\ga|_{[s_0,s_0+\eps]}$ by three small arcs of circles of radius $\frac{1}{\ka_0}$ as indicated in Figure \ref{F:bump}. If the bump is chosen to lie on the correct side, the curve $\zeta$ will be diffuse, and if $\eps>0$ is small enough, then $\zeta\in \sr V$. (Notice that this part of the proof works even if $\abs{\theta_1}=\pi$.)
	
	\begin{figure}[ht]
		\begin{center}
			\includegraphics[scale=.46]{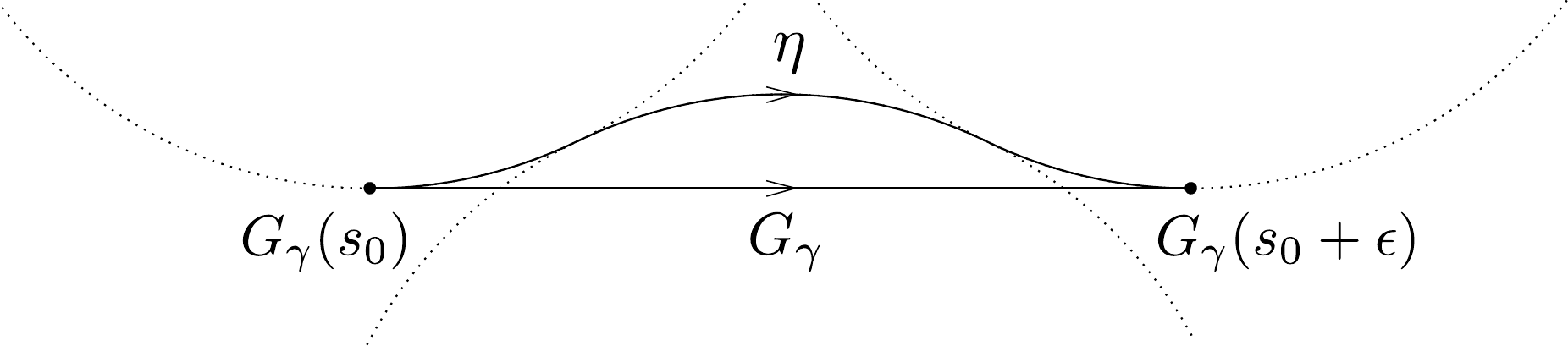}
			\caption{}
			\label{F:bump}
		\end{center}
	\end{figure}
	
	We have established that $\sr T\subs \bd \sr U_d\cap \bd \sr U_c$. As explained at the beginning of the section,
$\bd \sr U_c\subs \sr T$ and $\bd \sr U_d \subs \sr T$. Thus, $\bd \sr U_c=\bd \sr U_d=\sr T$.
	\end{proof}

\subsection*{Existence of critical curves}It is immediate from the definition of ``critical curve'' that if $\abs{\theta_1}>\pi$, then the subspace $\sr T\subs \sr L_{-1}^{+1}(Q;\theta_1)$ must be empty. In this subsection we shall determine exactly when $\sr T=\emptyset$ for $\abs{\theta_1}\leq \pi$.

\begin{defn}
	A \tdef{sign string} $\sig$ is an alternating finite sequence of signs, such as \ty{+-+} or \ty{-+-+}. As part of the definition we require that its \tdef{length} $\abs{\sig}$, the number of terms in the string, satisfy $\abs{\sig}\geq 2$. Let $\sig(k)$ denote its $k$-th term ($1\leq k\leq \abs{\sig}$). The \tdef{opposite} $-\sig$ of $\sig$ is the unique sign string satisfying $\abs{-\sig}=\abs{\sig}$ and $(-\sig)(k)=-\sig(k)$. 
	
	A critical curve $\ga\colon [0,1]\to \C$ is \tdef{of type $\sig$} if there exist $0\leq t_1<t_2<\dots<t_{\abs{\sig}}\leq 1$ with $\theta(t_k)=\theta^{\sig(k)}$ (recall that $\theta^+=\sup_{t\in [0,1]}\theta(t)$ and $\theta^-=\inf_{t\in [0,1]}\theta(t)$, where $e^{i\theta}=\ta_\ga$), but it is impossible to find $0\leq s_1<\dots<s_{\abs{\sig}+1}\leq 1$ such that $\ta_\ga(s_k)=-\ta_\ga(s_{k+1})$ for each $k=1,\dots,\abs{\sig}$.
\end{defn}
 
Given a sign string $\sig$, one can determine whether there exist critical curves of type $\sig$ in $\sr L_{-1}^{+1}(Q;\theta_1)$ using an elementary geometric construction, see Figure \ref{F:critical}.

\begin{figure}[ht]
	\begin{center}
		\includegraphics[scale=.32]{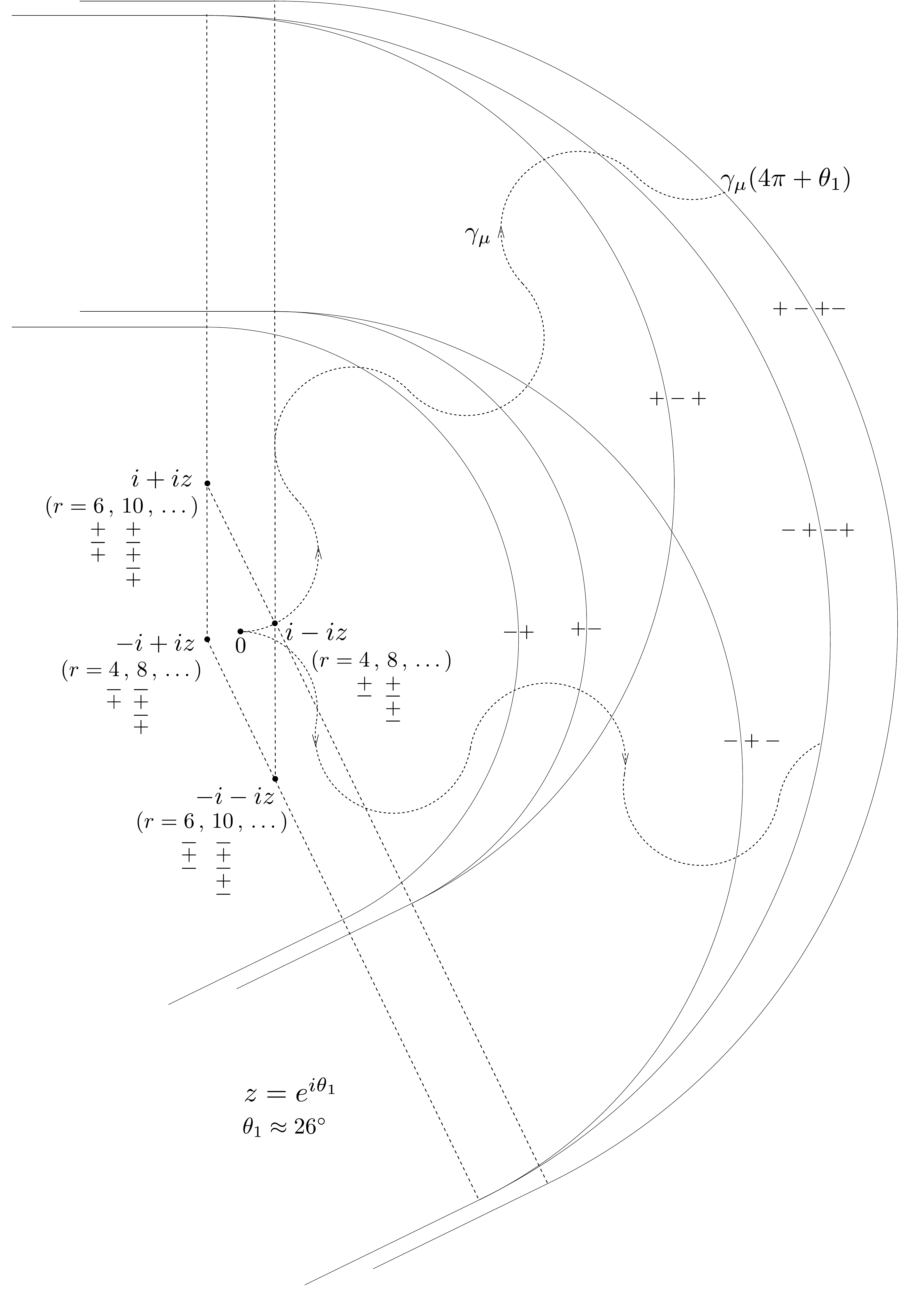}
		\caption{The regions $R_\sig$ of \pref{P:criticallocation}.}
		\label{F:critical}
	\end{center}
\end{figure}

\begin{prop}\label{P:criticallocation}
	Let $\theta_1\in [0,\pi]$, $z=e^{i\theta_1}$ and $Q=(q,z)\in \C\times \Ss^1$. Let $\sig$ be a sign string,
	\begin{equation*}
		a=i\sig(1)\big(1+(-1)^{\abs{\sig}+1}z\big)\in \C\text{\quad and\quad} r=2\abs{\sig}\in \N.
	\end{equation*}
	Let $R_\sig$ be the open region of the plane which does \emph{not} contain $-i+iz$ and which is bounded by the shortest arc of $C_r(a)$ joining $a+ri$ to $a-riz$ and the tangent lines to $C_r(a)$ at these points. Then $\sr L_{-1}^{+1}(Q;\theta_1)$ contains critical curves of type $\sig$ if and only if $q\in R_\sig$.
\end{prop}
	We have assumed that $\theta_1\in [0,\pi]$ just to simplify the statement. If $\theta_1\in [-\pi,0]$, then the only differences are that the points bounding the arc of $C_r(a)$ are now $a-ri$ and $a+riz$ and the region $R_\sig$ is the one not containing $i-iz$. Indeed, reflection across the $x$-axis yields a homeomorphism between $\sr L_{-1}^{+1}(Q;\theta_1)$ and $\sr L_{-1}^{+1}(\bar Q;-\theta_1)$, where $\bar Q=(\bar q,\bar z)$, which maps critical curves of type $\sig$ to critical curves of type $-\sig$.

When $\theta_1=0$, the points $a+ri$ and $a-ri$ determine two shortest arcs of $C_r(a)$, not just one; the region $R_\sig$ is bounded by the one which goes through $a+r$. When $\theta_1=\pm \pi$, the arc of circle degenerates to a single point. In this case, $R_\sig$ is the component of the complement of the horizontal line through $a+\sign(\theta_1) ri$ which does not contain the real axis.

\begin{proof}[Proof of \pref{P:criticallocation}]
	 There are four essentially distinct types of sign strings to consider: 
	 \begin{equation*}
	 \underbrace{\ty{+-\dots+-}}_{2n}\ \ ,\ \ \underbrace{\ty{-+\dots-+}}_{2n}\ \ ,\ \ \ty{+}\underbrace{\ty{-+\dots-+}}_{2n}\ \ \text{and}\ \ \ty{-}\underbrace{\ty{+-\dots+-}}_{2n}\quad (n\in \N,~n\geq 1).
	 \end{equation*}
	 (Note that these are distinguished by the values of $\sig(1)$ and $\abs{\sig}$ appearing in the expression for $a$.) 
	 We shall prove the theorem for a string of the first type; the proof in the remaining three cases is analogous. The argument given here is the same as the one which was used to prove \pref{P:condensedlocation}, so some details will be omitted.

	For each $\mu\in [\theta_1-\pi,0]$, let $\ga_\mu\colon [0,2n\pi+\theta_1]\to \C$ be the unique curve parametrized by arc-length satisfying:
	\begin{equation*}
		\ga_\mu(0)=0\text{\ and\ }\ta_{\ga_\mu}(s)=\begin{cases}
		e^{is} & \text{\  if $s\in [0,\mu+\pi]\cup [\mu+2n\pi,\theta_1+2n\pi]\bigcup_k [\mu+k\pi,\mu+(k+1)\pi]$} \\
		e^{i(2\mu-s)} & \text{\ if $s\in \bigcup_k [\mu+k\pi,\mu+(k+1)\pi]$}
		\end{cases}
	\end{equation*}
	where the first (resp.~second) union is over all $k\equiv 0$ (resp.~$k\equiv 1$)$\pmod 2$, $1\leq k\leq 2n-1$. Notice that $\ga_\mu$ is the concatenation of arcs of circles of radius 1, see~Figure \ref{F:critical}. (Vaguely speaking, $\ga_\mu$ is the ``most efficient'' critical curve $\ga$ of type $\sig$ with $\inf \theta_\ga=\mu$ and $\abs{\ka_\ga}\leq 1$.)  We have:
	\begin{alignat*}{9}
	\Phi_{\ga_\mu}(0)=(0,1),\ \ \ta_{\ga_\mu}(2n\pi+\theta_1)=z,\ \ \inf\theta_{\ga_\mu}=\mu,\ \ \sup\theta_{\ga_\mu}=\mu+\pi\ \ \text{and} \\
	\ga_\mu(2n\pi+\theta_1)=\bigg(\int_0^{\mu+\pi}+(2n-1)\int_{\mu}^{\mu+\pi}+\int_\mu^{\theta_1}\bigg) e^{is}\,ds=(i-iz)+4nie^{i\mu}.
	\end{alignat*}
	From the previous equation it follows that as $\mu$ increases from $\theta_1-\pi$ to $0$, the endpoint of $\ga_\mu$ traces out the arc of $C_r(a)$ joining $a-riz$ to $a+ri$, where $a=i-iz$ and $r=4n=2\abs{\sig}$. Further, the tangent line to $C_r(a)$ at $\ga_\mu(2n\pi+\theta_1)$ is parallel to $e^{i\mu}$, for it must be orthogonal to $4nie^{i\mu}$. 
	
	It is easy to see that any $q\in \overline{R}_\sig$ is the endpoint of a curve of one of the following three types:
	\begin{enumerate}
		\item [(i)] The concatenation of $\ga_\mu$ with a line segment of direction $z$, for some $\mu\in [\theta_1-\pi,0]$.
		\item [(ii)]  The concatenation of $\ga_0|_{[0,\pi]}$, a line segment of length $\ell\geq 0$ having direction $-1$, the arc $-\ell+\ga_0|_{[\pi,2n\pi+\theta_1]}$, and a line segment of direction $z$.
		\item [(iii)] The concatenation of $\ga_{\theta_1-\pi}|_{[0,\theta_1+\pi]}$, a line segment of length $\ell_1\geq 0$ of direction $-z$, the arc $-\ell_1 z+\ga_0|_{[\theta_1+\pi,\theta_1+\frac{3\pi}{2}]}$, a line segment of length $\ell_2\geq 0$ and direction $-iz$, the arc $-\ell_1 z-\ell_2iz+\ga_0|_{[\theta_1+\frac{3\pi}{2},2n\pi+\theta_1]}$, and a line segment of direction $z$.
	\end{enumerate}
	If $q\in R_\sig$, then we can find a critical curve $\ga$ of type $\sig$ in $\sr L_{-1}^{+1}(Q;\theta_1)$ by a slight modification of one of these curves.
	
	Conversely, suppose that $\sr L_{-1}^{+1}(Q;\theta_1)$ contains critical curves of type $\sig$.  Let $\eta\colon [0,L]\to \C$ be such a curve, parametrized by arc-length, and let $\mu=\inf\theta$, where $\theta\colon [0,L]\to \R$ is the argument of $\ta_\eta$ satisfying $\theta(0)=0$. Define
	\begin{equation*}
g\colon [0,L]\to \R\quad\text{by}\quad g(s)=\big\langle\eta(s)-\ga_\mu(2n\pi+\theta_1)\,,\,ie^{i\mu}\big\rangle.
	\end{equation*}
	Notice that $g(s)>0$ if and only if $\eta(s)$ lies to the left of the line through $\ga_{\mu}(2n\pi+\theta_1)\in C_r(a)$ having direction $e^{i\mu}$; we have already seen that this line is tangent to $C_r(a)$ at this point. We claim that $g(L)>0$. Since $\eta$ is critical, $\theta(s)\in [\mu,\mu+\pi]$ for all $s$. Hence, 
	\begin{equation}\label{E:non-negative}
		g'(s)=\big\langle e^{i\theta(s)},ie^{i\mu}\big\rangle=\cos\big(\theta(s)-(\mu+\tfrac{\pi}{2})\big)\geq 0\text{\quad  for all $s\in [0,L]$}.
	\end{equation}
	Let $J_i=(a_i,b_i)\subs (0,L)$, $i=0,\dots,2n=\abs{\sig}$, be disjoint intervals such that:
	\begin{enumerate}
		\item [(I)] $\theta(a_0)=0$ and $\theta(b_1)=\mu+\pi$; 
		\item [(II)]  $\theta(a_i)=\mu+\pi$ and $\theta(b_i)=\mu$\ \ for\ \ $i=1,3,\dots,2n-1$;
		\item [(III)] $\theta(a_i)=\mu$ and $\theta(b_i)=\mu+\pi$\ \ for\ \ $i=2,4,\dots,2n-2$;
		\item[(IV)] $\theta(a_{2n})=\mu+\pi$ and $\theta(b_{2n})=\theta_1$.
	\end{enumerate}
	Such intervals exist because $\theta([0,L])\subs [\mu,\mu+\pi]$, $\theta(0)=0$, $\theta(L)=\theta_1$ and $\eta$ is critical of type $\sig$. It follows from \eqref{E:non-negative} and the fact $\theta$ is strictly 1-Lipschitz (by eq.~\eqref{E:rate} and the fact that  $\theta=\arg\circ \ta_\eta$ is absolutely continuous) that
	\begin{alignat*}{9}
		g(L)-g(0)& \geq \bigg(\sum_{i=0}^{2n}\int_{a_i}^{b_i}\bigg) g'(s)\,ds\\
		&>\bigg(\int_0^{\mu+\pi}+\,(2n-1)\int_{\mu}^{\mu+\pi}+\int_\mu^{\theta_1}\bigg)\gen{e^{it},ie^{i\mu}}\,dt=\big\langle\ga_\mu(2n\pi+\theta_1),ie^{i\mu}\big\rangle.
	\end{alignat*}
	Therefore, $g(L)>0$ as claimed. We conclude that $q=\eta(L)$ lies on the side of a tangent to $C_r(a)$ which only contains points of $R_\sig$.
\end{proof}

\begin{cor}\label{C:criticallocation}
	Let $Q=(q,z)\in \C\times \Ss^1$,
	\begin{equation*}
		a=\sign(\Im(z))\? i(z-1)\in \C
	\end{equation*}
	and let $R_{\sr T}$ be the open region of the plane which does \emph{not} contain $a$ and which is bounded by the shortest arc of $C_4(a)$ joining $a+\sign(\Im(z))\?4i$ to $a-\sign(\Im(z))\?4iz$ and the tangent lines to $C_4(a)$ at these points. Then $\sr L_{-1}^{+1}(Q)$ contains critical curves if and only if $q\in R_{\sr T}$.
\end{cor}
\begin{proof}
	Let $z=e^{i\theta_1}$, $\abs{\theta_1}\leq \pi$.  For $\theta_1\in [0,\pi]$ (resp.~$\theta_1\in [-\pi,0]$), $R_{\sr T}$ is the same as the region $R_{\ty{-+}}$ (resp.~$R_{\ty{+-}}$) appearing in proposition \pref{P:criticallocation}.
\end{proof}
If $z=\pm 1$, then $\sign(\Im z)$ is not defined. When $z=1$, $R_{\sr T}$ is bounded by the semicircle centered at 0 through $4$ and $\pm 4i$ and the tangents to $C_4(0)$ at the latter two points. When $z=-1$, $R_{\sr T}$ is bounded by the horizontal lines through $\pm 2i$. 

\begin{cor}\label{C:bat}
	Let $Q=(q,z)\in \C\times \Ss^1$ and $e^{i\theta_1}=z$. Then there exist condensed curves but not critical curves in $\sr L_{-1}^{+1}(Q;\theta_1)$ if and only if $\abs{\theta_1}<\pi$ and $q$ lies in the region illustrated in Figure \ref{F:S0}.
\end{cor}
\begin{proof}
	This is an immediate consequence of \pref{P:condensedlocation} and \cref{C:criticallocation}.
\end{proof}

\begin{lem}\label{L:omega}
	Let $Q=(q,z)\in \C\times \Ss^1$ and $e^{i\theta_1}=z$, $\abs{\theta_1}\leq \pi$. Let $\om\in [\abs{\theta_1},\pi]$ and $r(\om)=4\sin\big(\frac{\om}{2}\big)$. Suppose that $q$ lies inside of $C_{r(\om)}\big(\sign(\theta_1)i(z-1)\big)$. Then there does not exist a curve in $\hat{\sr L}_{-1}^{+1}(Q;\theta_1)$ or $\sr L_{-1}^{+1}(Q;\theta_1)$ having amplitude $\om$.
\end{lem}
\begin{proof}
	Assume that $\theta_1\in [0,\pi]$; the proof for $\theta_1\in [-\pi,0]$ is analogous. Let $\om\in [\theta_1,\pi]$, $\mu\in [\theta_1-\om,0]$ and let $\ga_\mu\colon [0,2\om-\theta_1]\to \C$ be the unique curve parametrized by arc-length satisfying 
	\begin{equation*}
		\ga_\mu(0)=0\text{ and }\ta_{\ga_{\mu}}(s)=\begin{cases}
		e^{-is} & \text{ if $s\in [0,-\mu]$;} \\
		e^{i(s+2\mu)} & \text{ if $s\in [-\mu,-\mu+\om]$; }\\
		e^{-i(s-2\om)} & \text{ if $s\in [-\mu+\om,2\om-\theta_1]$.}
		\end{cases}
	\end{equation*}
	Notice that $\ta_{\ga_\mu}(0)=1$, $\ta_{\ga_\mu}(2\om-\theta_1)=z$ and $\ga_\mu$ is a concatenation of three arcs of circles of radius 1. Moreover, $\inf \theta_{\ga_\mu}=\mu$ and $\sup \theta_{\ga_\mu}=\mu+\om$, where $\theta_{\ga_\mu}$ is the argument of $\ta_{\ga_\mu}$ satisfying $\theta_{\ga_\mu}(0)=0$. Consequently, $\ga_\mu$ has amplitude $\om$. Further,
	\begin{equation*}
		\ga_{\mu}(2\om-\theta_1)=\bigg(\int_\mu^{0}+\int_{\mu}^{\mu+\om}+\int_{\theta_1}^{\mu+\om}\bigg) e^{is}\,ds=(-i+iz)+4\sin\Big(\frac{\om}{2}\Big)e^{i\big(\mu+\tfrac{\om}{2}\big)}.
	\end{equation*}
	Thus, as $\mu$ increases from $\theta_1-\om$ to $0$, the endpoint of $\ga_\mu$ traverses an arc of $C_{r(\om)}(-i+iz)$. Suppose that there exists $\eta\in \hat{\sr L}_{-1}^{+1}(Q;\theta_1)$ of amplitude $\om$, and let $\eta\colon [0,L]\to \C$ be parametrized by arc-length. Let $\theta_\eta$ be the argument of $\eta$ satisfying $\theta_\eta(0)=0$, take $\mu=\inf \theta_\eta$ and define 
	\begin{equation*}
		g\colon [0,L]\to \R\quad\text{by}\quad g(s)=\Big\langle\eta(s)-\ga_\mu(2\om-\theta_1)\,,\,e^{i\big(\mu+\tfrac{\om}{2}\big)}\Big\rangle.
	\end{equation*}
	Then the same reasoning used to establish \pref{P:condensedlocation} and \pref{P:criticallocation} shows that $g(L)\geq 0$. This implies that $\eta(L)=q$ lies on or to the left of the line through $\ga_\mu(2\om-\theta_1)$ having direction $\exp\big(i\big(\mu+\tfrac{\om-\pi}{2}\big)\big)$. This line is tangent to $C_{r(\om)}(-i+iz)$ at this point, therefore $q$ cannot lie inside of this circle. This proves the assertion about $\hat{\sr L}_{-1}^{+1}(Q;\theta_1)$. Since the latter contains $\sr L_{-1}^{+1}(Q;\theta_1)$ as a subset, the proof is complete.
\end{proof}

The next result will only be needed in \cite{SalZueh2}. Recall the definition of $\bar\vphi_\ga$ in \eqref{E:aver}.
\begin{cor}\label{C:criticalinterval}
	Let $\abs{\theta_1}\leq \pi$, $e^{i\theta_1}=z$, $Q=(q,z)\in \C\times \Ss^1$ and $\sig$ be a sign string. Then the set of all $\vphi\in \R$ such that there exists a critical curve $\ga\in \sr L_{-1}^{+1}(Q;\theta_1)$ of type $\sig$ for which $\bar\vphi_\ga=\vphi$ is an open interval.
\end{cor}
\begin{proof}
	No generality is lost in assuming that $\theta_1\in [0,\pi)$. Let $\vphi\in \R$. It was established in the proof of \pref{P:criticallocation} that a curve $\ga$ as in the statement exists if and only if $\vphi\in [\theta_1-\frac{\pi}{2},\frac{\pi}{2}]$ and $q$ lies in the open external region $E_\vphi$ determined by the tangent orthogonal to $e^{i\vphi}$  to a certain circle $C$ (which depends only on $\theta_1$ and $\sig$). It is straightforward to check that $E_{\vphi_1}\cap E_{\vphi_2}\subs E_\vphi$ whenever $\vphi\in [\vphi_1,\vphi_2]$ with $\vphi_2-\vphi_1\leq \pi$. Moreover, if $q\in E_\vphi$, then $q\in E_{\te{\vphi}}$ for all $\te{\vphi}$ sufficiently close to $\vphi$.
\end{proof}


\section{Components of $\sr L_{\ka_1}^{\ka_2}(P,Q)$ for $\ka_1\ka_2<0$}\label{S:components}



Recall that $\E$ denotes the separable Hilbert space and $B_{\ga,1,\frac{1}{2}}$ is obtained from $\ga$ by attaching a figure eight curve (at $t=1/2$); see \dref{D:eighting} and Figure \ref{F:spreading}\?(d).

\begin{thm}\label{T:criterion}
	Let $Q=(q,z)\in \C\times \Ss^1$ and $\theta_1\in \R$ satisfy $e^{i\theta_1}=z$. Then the following assertions are equivalent:
	\begin{enumerate}
		\item [(i)] $\sr L_{-1}^{+1}(Q;\theta_1)$ is disconnected.
		\item [(ii)] $\abs{\theta_1}<\pi$ and $q$ lies in the region depicted in Figure \ref{F:S0}.
		\item [(iii)] $\abs{\theta_1}<\pi$ and there exist condensed curves, but not critical curves, in $\sr L_{-1}^{+1}(Q)$.
		\item [(iv)] $\abs{\theta_1}<\pi$ and there exist condensed curves in $\sr L_{-1}^{+1}(Q)$, but no condensed curve is homotopic to a diffuse curve within $\sr L_{-1}^{+1}(Q)$.
		\item [(v)] $\abs{\theta_1}<\pi$ and there exists an embedding $\ga\in \sr L_{-1}^{+1}(Q)$ which cannot be homotoped within this space to create self-intersections.
		\item [(vi)] $\abs{\theta_1}<\pi$ and there exists $\ga\in \sr L_{-1}^{+1}(Q)$ which does not lie in the same component as $B_{\ga,1,\frac{1}{2}}$.
	\end{enumerate}
	Furthermore, if $\sr L_{-1}^{+1}(Q;\theta_1)$ is disconnected, then it has exactly two components; one of them is $\sr U_c$ and the other is $\sr U_d\subs \sr L_{-1}^{+1}(Q;\theta_1)$, and both are homeomorphic to $\E$, hence contractible.
\end{thm}
\begin{proof}
	
	 We know from \tref{T:condensed} and \tref{T:diffuse} that each of $\sr U_c,\,\sr U_d\subs \sr L_{-1}^{+1}(Q;\theta_1)$ is homeomorphic to $\E$, hence connected. By \cref{C:top}, $\sr U_d\neq \emptyset$. By \pref{P:boundary}, $\bar{\sr U}_c\cup \bar{\sr U}_d=\sr L_{-1}^{+1}(Q;\theta_1)$, and $\bar{\sr U}_c\cap \bar{\sr U}_d$ consists of all the critical curves in $\sr L_{-1}^{+1}(Q;\theta_1)$. Thus, the latter has at most two connected components. It has exactly two if and only if $\bar{\sr U}_c\neq \emptyset$ but $\bar{\sr U}_c\cap \bar{\sr U}_d=\emptyset$, that is, if and only if there exist condensed curves, but not critical curves. This proves the last assertion of the theorem and also the equivalence (i)$\Dar$(iii). The equivalence (ii)$\Dar$(iii) was proved in \cref{C:bat}.
	
	Suppose that $s\mapsto \ga_s\in \sr L_{-1}^{+1}(Q;\theta_1)$ is a path joining a condensed curve to a diffuse curve. Let $\theta_{s}$ be the argument of $\ta_{\ga_s}$ satisfying $\theta_s(0)=0$. By continuity, there must exist $s_0\in (0,1)$ such that 
	 \begin{equation*}
		\sup_{t\in [0,1]}\theta_{{s_0}}(t)-\inf_{t\in [0,1]}\theta_{{s_0}}(t)=\pi,
	\end{equation*}
	that is, there must exist $s_0$ such that $\ga_{s_0}$ is critical. Hence, (iii)$\Rar$(iv).
	 
	Suppose that (iv) holds, and let $\ga\in \sr L_{-1}^{+1}(Q)$ be smooth and condensed. Then $\ga$ is an embedding, but it cannot be deformed to have a self-intersection since any curve with double points must be diffuse. Thus, (iv)$\Rar$(v). 

	
	Finally, it is obvious that (v)$\Rar$(vi) and (vi)$\Rar$(i).
\end{proof}

\begin{cor}\label{C:topo}
	Let $Q=(q,z)\in \C\times \Ss^1$ and $\theta_1\in \R$ satisfy $e^{i\theta_1}=z$. If $\abs{\theta_1}\geq \pi$, then $\sr L_{-1}^{+1}(Q;\theta_1)$ is connected. If $\abs{\theta_1}>\pi$, then $\sr L_{-1}^{+1}(Q;\theta_1)$ is  homeomorphic to $\E$, hence contractible.
\end{cor}
\begin{proof}
	The first assertion is an immediate consequence of \tref{T:criterion}. If $\abs{\theta_1}>\pi$ then $\sr L_{-1}^{+1}(Q;\theta_1)$ can only contain diffuse curves, and we know from \tref{T:diffuse} that $\sr U_d$ is homeomorphic to $\E$.
\end{proof}
\begin{urem}
	The results of \S\ref{S:diffuse} go through to show that $\sr L_{-1}^{+1}(Q;\theta_1)=\sr T\cup \sr U_d$ is also contractible when $\theta_1=\pm \pi$. Of course, if $\abs{\theta_1}<\pi$ then $\sr L_{-1}^{+1}(Q;\theta_1)$ need not even be connected. We shall prove in the sequel \cite{SalZueh2} that it may also be contractible, or connected but not contractible, depending on $Q$.
\end{urem}

\begin{cor}
	Let $Q=(q,z)\in \C\times \Ss^1$ and $\theta_1\in \R$ satisfy $e^{i\theta_1}=z$. Then the subset $\sr L_{-1}^{+1}(Q;\theta_1)$ is either a connected component or the union of two contractible  components of $\sr L_{-1}^{+1}(Q)$. The latter can occur only if $\abs{\theta_1}<\pi$, that is, for at most one value of $\theta_1$. \qed
\end{cor}

\begin{thm}\label{T:classification}
	Let $P=(p,w),\,Q=(q,z)\in \C\times \Ss^1$, $\ka_1<0<\ka_2$ and let $\theta_1$ satisfy $e^{i\theta_1}=z\bar w$. 
	\begin{enumerate}
		\item [(a)] If $\abs{\theta_1}\geq \pi$, then the subspace $\sr L_{\kappa_1}^{\kappa_2}(P,Q;\theta_1)$ consisting of all curves having total turning $\theta_1$ is a contractible  connected component of $\sr L_{\kappa_1}^{\kappa_2}(P,Q)$, homeomorphic to $\E$.
		\item [(b)]  If $\abs{\theta_1}<\pi$, then $\sr L_{\kappa_1}^{\kappa_2}(P,Q;\theta_1)$ has at most two components. It is disconnected if and only if any of the conditions in \tref{T:criterion} is satisfied for $\hat{Q}=(\hat{q},z\bar w)$, where
	\begin{equation*}
		\hat q=\tfrac{2}{\rho_2-\rho_1}\bar w\big[(q-p)+\tfrac{i}{2}(\rho_1+\rho_2)(z-w)\big] \qquad (\rho_i=\tfrac{1}{\ka_i}, i=1,2).
	\end{equation*}
	 In this case, one component consists of all condensed and the other of all diffuse curves in $\sr L_{\kappa_1}^{\kappa_2}(P,Q;\theta_1)$, and both are homeomorphic to $\E$.
	\end{enumerate}
%
\end{thm}

\begin{proof}
	This is just a corollary of \tref{T:normalized}\?(a)), \tref{T:criterion} and \cref{C:topo}
\end{proof}

We emphasize that the subspace of $\sr L_{\kappa_1}^{\kappa_2}(P,Q)$ which contains curves having least total turning, described in (b), does not have to be contractible even if it is connected. 
Observe also that we may replace $\sr L$ by $\sr C$ invoking \lref{L:C^2}.


\section{Homeomorphism class of $\sr L_{\ka_1}^{\ka_2}(P,Q)$ for $\ka_1\ka_2\geq 0$}\label{S:locallyconvex}
	An admissible plane curve $\ga$ is called locally convex if either $\ka_\ga>0$ a.e.~or $\ka_\ga<0$ a.e..  Notice that $\sr L_{\kappa_1}^{\kappa_2}(P,Q)$ consists of locally convex curves if and only if $\ka_1\ka_2\geq 0$. This corresponds to parts (b)--(e) of \tref{T:normalized}. The topology of these spaces is very simple. 
	
	Suppose that $\ga\colon [0,1]\to \C$ is an admissible curve such that $\ka_\ga>0$ a.e.~and $\Phi_\ga(0)=(0,1)$. By \eqref{E:rate}, any argument $\theta\colon [0,1]\to \R$ of $\ta_\ga$ must be strictly increasing; in particular, the total turning $\theta_1$ of $\ga$ is positive. Thus, $\ga$ may be \tdef{parametrized by} its \tdef{argument} $\theta\in [0,\theta_1]$. By the chain rule, 
	\begin{equation}\label{E:byargument}
		\dot\ga(\theta)=\rho(\theta)e^{i\theta}\quad (\theta\in [0,\theta_1]),
	\end{equation}
	where $\rho\colon [0,\theta_1]\to (0,+\infty)$ is the radius of curvature of $\ga$.\footnote{The idea of parametrizing a locally convex curve by the argument of its unit tangent vector is not new. It appears in \cite{Little}, where it is attributed to W.~Pohl. We do not know whether it is older than that.} 
\begin{thm}\label{T:convex}
	Let $P,\,Q\in UT\C$ and suppose that either $\ka_1\geq 0$ or $\ka_2\leq 0$. Then $\sr L_{\kappa_1}^{\kappa_2}(P,Q)$ has infinitely many connected components, one for each realizable total turning. All of these components are homeomorphic to $\E$, hence contractible.
\end{thm}
\begin{proof}
	 Using an Euclidean motion if necessary, we may assume that $P=(0,1)$. Further, by reversing the orientation of all curves, we pass from the case where $\ka_2\leq 0$ to the case where $\ka_1\geq 0$. 
	
	Let $Q=(q,z)$ and $e^{i\theta_1}=z$. The subspace $\sr L_{\ka_1}^{\ka_2}(Q;\theta_1)$ is both open and closed in $\sr L_{\ka_1}^{\ka_2}(Q)$ (but it may be empty, see \lref{L:nonempty}. In particular, two curves which have different total turnings cannot lie in the same component of $\sr L_{\kappa_1}^{\kappa_2}(Q)$. For any $k\in \N$, we may concatenate a curve in $\sr L_{\kappa_1}^{\kappa_2}(Q)$ with a circle of curvature in $(\ka_1,\ka_2)$ traversed $k$ times. This shows that the number of components is infinite.
	
	Suppose that $\sr L_{\ka_1}^{\ka_2}(Q;\theta_1)\neq \emptyset$. 
	Since $\ka_1\geq 0$ by hypothesis, we may reparametrize all curves in $\sr L_{\ka_1}^{\ka_2}(Q;\theta_1)$ by the argument $\theta\in [0,\theta_1]$ of their unit tangent vectors using \cref{C:reparametrize}. Choose any $\ga_0\in \sr L_{\ka_1}^{\ka_2}(Q;\theta_1)$ and define a map $H$ on $[0,1]\times \sr L_{\ka_1}^{\ka_2}(Q;\theta_1)$ by:
\begin{equation*}
	H(s,\ga)=\ga_s,\quad \ga_s(\theta)=(1-s)\ga_0(\theta)+s\ga(\theta)\quad (s\in [0,1],~\theta\in [0,\theta_1]).
\end{equation*}
Then $\ga_s(0)=0$, $\ga_s(\theta_1)=q$ and the unit tangent vector $\ta_{\ga_s}$ to $\ga_s$ satisfies
\begin{equation*}
	\ta_{\ga_s}(\theta)=e^{i\theta}\quad\text{ for all $\ga\in \sr L_{\ka_1}^{\ka_2}(Q;\theta_1)$, $s\in [0,1]$ and $\theta\in [0,\theta_1]$.}
\end{equation*}
Consequently, each $\ga_s$ has total turning $\theta_1$, $\Phi_{\ga_s}(0)=(0,1)$ and $\Phi_{\ga_s}(\theta_1)=Q$. Let $\rho_0,\,\rho\colon [0,1]\to (0,+\infty)$ denote the radii of curvature of $\ga_0,\,\ga$, respectively. It follows from \eqref{E:byargument} that the radius of curvature $\rho_s$ of $\ga_s$ is given by:
\begin{equation*}
	\rho_s=(1-s)\rho_0+s\rho.
\end{equation*}
Therefore, the curvature $\ka_s=\frac{1}{\rho_s}$ of $\ga_s$ takes values in $(\ka_1,\ka_2)$ and $H$ is a contraction of $\sr L_{\ka_1}^{\ka_2}(Q;\theta_1)$. We conclude that the latter is a connected component of $\sr L_{\kappa_1}^{\kappa_2}(Q)$ and, using \lref{L:Hilbert}\?(b), that it is homeomorphic to $\E$.
\end{proof}

\subsection*{Possible total turnings of a curve in $\sr L_{\kappa_1}^{\kappa_2}(P,Q)$ when $\ka_1\ka_2>0$}  Let $T$ denote the set of all total turnings which are realized by some curve in $\sr L_{\kappa_1}^{\kappa_2}(P,Q)$. If $P=(p,w),\,Q=(q,z)$, then obviously 
\begin{equation*}
	T\subs \set{\theta_1+2k\pi}{k\in \Z},\ \ \text{where $e^{i\theta_1}=z\bar w$}.
\end{equation*}
If $\ka_1\ka_2<0$, this inclusion is an equality by \lref{L:nonempty}\?(b). However, it must be proper when $\ka_1\ka_2\geq 0$. If $\ka_1,\ka_2$ are both positive, for instance, then, by \eqref{E:rate} and the second paragraph of the above proof, $T$ must have the form $\set{\mu+2k\pi}{k\in \N}$, where $\mu\in \R$ ($e^{i\mu}=z\bar w$) is the minimal attainable total turning in this space. It is possible to find the value of $\mu$ in terms of all parameters involved. Because this determination is of lesser interest and relatively technical, we shall not go into it here. However, interested readers can find the details, including the analogue for spaces of the form $\hat{\sr{L}}$, in \cite{SalZuehold}. We mention only that \tref{T:normalized} allows one to restrict attention to the two classes $\sr L_{1}^{+\infty}(Q)$ and $\sr L_{0}^{+\infty}(Q)$.

\section{Components of spaces of curves on complete flat surfaces}\label{S:flat}

By a \tdef{flat surface} we mean a connected Riemannian 2-manifold whose Gaussian curvature is identically zero; it will not be necessary to assume that $S$ is a submanifold of some Euclidean space. The \tdef{unit tangent} $\ta=\ta_\ga\colon [0,1]\to UTS$ to a regular curve $\ga\colon [0,1]\to S$ is defined as before, $\ta=\frac{\dot\ga}{\abs{\dot\ga}}$. If $S$ is orientable, the \tdef{unit normal} $\no=\no_\ga\colon [0,1]\to UTS$ to $\ga$ is defined by the condition that $(\ta(t),\no(t))$ should be a positively oriented orthonormal basis of $TS_{\ga(t)}$ for each $t\in [0,1]$. For $\ga$ of class $C^2$, we can then define its \tdef{curvature} $\ka_\ga\colon [0,1]\to \R$  by
\begin{equation*}
	\ka_\ga=\frac{1}{\abs{\dot\ga}}\Big\langle\frac{D\ta}{dt},\no\Big\rangle,
\end{equation*} 
where $D$ denotes covariant differentiation (along $\ga$).

If $S$ is nonorientable, we can still speak of the \tdef{unsigned curvature} $\ka_\ga\colon [0,1]\to [0,+\infty)$ of a curve $\ga\colon [0,1]\to S$, given by
\begin{equation*}
	\ka_\ga=\frac{1}{\abs{\dot\ga}}\Big\vert\Big\langle\frac{D\ta}{dt},\no\Big\rangle\Big\vert
\end{equation*}
where now $\no(t)$ denotes any of the two unit vectors in $TS_{\ga(t)}$ orthogonal to $\ta(t)$. 

\begin{defn}\label{D:curvespace}
	Let $S$ be an orientable flat surface, $u,v\in UTS$, $-\infty\leq \ka_1<\ka_2\leq +\infty$ and $2\leq r\in \N$. Define $\sr CS_{\ka_1}^{\ka_2}(u,v)$ to be the set of all $C^r$ regular curves $\ga\colon [0,1]\to S$ satisfying:
\begin{enumerate}
	\item [(i)] $\ta_\ga(0)=u$ and $\ta_\ga(1)=v$;
	\item [(ii)] $\ka_1<\ka_\ga(t)<\ka_2$ for each $t\in [0,1]$.
\end{enumerate}
In case $S$ is nonorientable, define $\sr CS_{-\ka_0}^{+\ka_0}(u,v)$ ($\ka_0>0$) as above, but replacing condition (ii) by:
\begin{enumerate}
	\item [(ii$'$)] $\ka_\ga(t)<\ka_0$ for each $t\in [0,1]$.
\end{enumerate}
In both cases, let $\sr C_{\kappa_1}^{\kappa_2}(u,v)$ be furnished with the $C^r$ topology.
\end{defn}


\begin{urem}
	A complete flat surface must be homeomorphic to one of the following five: $\C$ itself, a cylinder $ \Ss^1\times \R$, an open M\"obius band, a torus or a Klein bottle. This is essentially a corollary of the following  result, cf.~\cite{Hopf}, p.~319.
\end{urem}

\begin{thm}[Killing-Hopf]\label{T:KH}
	Any complete flat surface is isometric to the quotient of the Euclidean plane $\C$ by a group of isometries acting freely and properly discontinuously on $\C$. \qed
\end{thm}

Hence, if $S$ is a complete flat surface, there exists a covering map $\C\to S$ which is a local isometry. Any curve on $S$ may thus be lifted to a plane curve whose curvature is the same as that of the original curve, with the proviso that we ignore its sign if $S$ is nonorientable. 
Let $\pr\colon UT\C\to UTS$ denote the natural projection induced by the covering map. In what follows, when referring to $\sr CS_{\kappa_1}^{\kappa_2}(u,v)$, we adopt the convention that $\ka_1=-\ka_2<0$ if $S$ is nonorientable.

\begin{prop}\label{P:flat}
	Let $S$ be a complete flat surface, $u,v\in UTS$, $-\infty\leq \ka_1<\ka_2\leq +\infty$ and $P\in UT\C$ be a fixed element of $\pr^{-1}(u)$. Then $\sr CS_{\kappa_1}^{\kappa_2}(u,v)$ is homeomorphic to $\coprod_{Q\in \pr^{-1}(v)} \sr C_{\kappa_1}^{\kappa_2}(P,Q)$, where the homeomorphism maps a curve in the latter to its image under the quotient map $\C\to S$.\qed
\end{prop}

Here $\coprod$ denotes topological sum. Clearly, this decomposition is sufficient to determine the connected components of $\sr CS_{\kappa_1}^{\kappa_2}(u,v)$ explicitly, using \lref{L:C^2} and \tref{T:classification} if $\ka_1\ka_2<0$ or  \tref{T:convex} if $\ka_1\ka_2\geq 0$. 


%

\begin{cor}
	Let $S$ be a complete flat surface, $\ka_1<\ka_2$ and $u,v\in UTS$. Then $\sr CS_{\kappa_1}^{\kappa_2}(u,v)$ is nonempty and has an infinite number of connected components. 
\end{cor}
\begin{proof}
	By \lref{L:nonempty} and the remark which follows it, $\sr C_{\kappa_1}^{\kappa_2}(P,Q)$ is always nonempty. The assertion is thus an immediate consequence of \pref{P:flat}.
\end{proof}

	Notice that it is irrelevant here whether $S$ is compact. This should be compared to the case of $S=\Ss^2$, where, at least when $u=v$, the number of components of $\sr CS_{\kappa_1}^{\kappa_2}(u,v)$ is finite for any choice of $\ka_1<\ka_2$ (see \S7 of \cite{SalZueh}). This is actually not surprising, since the fundamental group of $UT\C$ is isomorphic to $\Z$, but that of $UT\Ss^2\home \SO_3$ is isomorphic to $\Z_2$. We remark without proof that $\sr CS_{\kappa_1}^{\kappa_2}(u,v)$ may be empty for more general surfaces (for instance, $\sr CS_{-1}^{+1}(u,u)=\emptyset$ when $S$ is the hyperbolic plane $\Hh^2$, for any $u\in UT\Hh^2$).

\begin{cor}
	Let $S$ be a complete flat surface and $u,v\in UTS$. Let $\eta \in \sr CS_{-\infty}^{+\infty}(u,v)$ and suppose that $\ka_1\ka_2<0$. Then there exists $\ga\in \sr CS_{\kappa_1}^{\kappa_2}(u,v)$ lying in the same component of $\sr CS_{-\infty}^{+\infty}(u,v)$ as $\eta$.
\end{cor}
In other words, given a regular curve $\eta$ on $S$ with $\ta_\eta(0)=u$, $\ta_\eta(1)=v$, we may deform $\eta$ through regular curves, keeping $\ta(0)$, $\ta(1)$ fixed, to obtain a curve having curvature in $(\ka_1,\ka_2)$ everywhere.
\begin{proof}
	Take $P\in UT\C$ such that $\pr(P)=u$. Let $\te{\eta}$ be the lift of $\eta$ to $\C$ with initial frame $P$; let $Q$ be its final frame and $\theta_1$ its total turning. By \lref{L:nonempty}\?(b), $\sr C_{\kappa_1}^{\kappa_2}(P,Q;\theta_1)$ is nonempty. Let $\te{\ga}$ be one of its elements. Then the projection $\ga$ of $\te{\ga}$ on $S$ satisfies the conclusion of the corollary because of \pref{P:flat} and the fact that $\te\eta$, $\te\ga$ lie in the same component of $\sr C_{-\infty}^{+\infty}(P,Q)$. 
\end{proof}

Again, the analogue of this result does not hold for a general surface $S$, e.g., for $S=\mathbf{H}^2$. It is also false for a flat surface if $\ka_1\ka_2\geq 0$. To see this, let $P,Q\in UT\C$ satisfy $\pr(P)=u$, $\pr(Q)=v$, choose $\te\eta\in \sr C_{-\infty}^{+\infty}(P,Q)$ to have a total turning which is unattainable for curves in $\sr C_{\kappa_1}^{\kappa_2}(P,Q)$ and let $\eta$ be the projection of $\te\eta$ on $S$.


\section{Final remarks}\label{S:final}
\subsection*{Spaces of curves with curvature in a closed interval}
In both \cite{Dubins} and \cite{Dubins1}, Dubins worked with the set $\hat{\sr L}_{-\ka_0}^{+\ka_0}(Q)$ of definition \dref{D:main} (but with the $C^1$ topology), where the curvatures are restricted to lie in a \tsl{closed} interval. This choice is motivated by the fact that these spaces, unlike those of the form $\sr L_{-\ka_0}^{+\ka_0}(Q)$, always contain curves of minimal length (see Proposition 1 in \cite{Dubins}).  All of the main results in our paper concerning the topology of $\sr L_{\ka_1}^{\ka_2}(P,Q)$ have analogues for $\hat{\sr L}_{\ka_1}^{\ka_2}(P,Q)$.  We shall now briefly indicate the modifications which are necessary. 

Notice that $\hat{\sr L}_{\ka_1}^{\ka_2}(P,Q)$ is \tsl{not} a Banach manifold, and that the analogue of \lref{L:submersion} is false for these spaces, as shown by example  \eref{E:closed}. In contrast, \lref{L:reparametrize} and \cref{C:reparametrize} still hold when $\sr M=\hat{\sr L}_{\ka_1}^{\ka_2}(P,Q)$.  The important corollary \cref{C:normalized} has the following analogue, whose proof is essentially the same as that of \tref{T:normalized}, see remark \rref{R:obvious}.

 \begin{prop}\label{T:normalizedbar}
 	Let $P,\,Q\in UT\C$ and $\ka_1<\ka_2$ be finite. Then there exists a homeomorphism between $\hat{\sr L}_{\kappa_1}^{\kappa_2}(P,Q)$ and a space of type $\hat{\sr L}_{-1}^{+1}(Q_0)$, $\hat{\sr L}_{0}^{1}(Q_0)$ or $\hat{\sr L}_{1}^{2}(Q_0)$, according as $\ka_1\ka_2<0$, $\ka_1\ka_2=0$ or $\ka_1\ka_2>0$, respectively. Moreover, this homeomorphism preserves the total turning of curves unless $\ka_1<\ka_2\leq 0$, in which case it reverses the sign. \qed
 \end{prop}
 
 In case $\ka_1\ka_2<0$, we actually have $\hat{\sr L}_{\kappa_1}^{\kappa_2}(P,Q)\home \hat{\sr L}_{-1}^{+1}(Q_1)$ with $Q_1$ as in the statement of \tref{T:normalized}. We leave the task of determining $Q_0$ in the other two cases to the interested reader. 

 Let us denote by $\hat{\sr U}_c$,  $\hat{\sr U}_d$ and $\hat{\sr T}\subs \hat{\sr L}_{-1}^{+1}(Q;\theta_1)$ the subspaces consisting of all condensed, diffuse and critical curves, where $Q=(q,z)$, $e^{i\theta_1}=z$ and $\hat{\sr L}_{-1}^{+1}(Q;\theta_1)$ consists of those curves in $\hat{\sr L}_{-1}^{+1}(Q)$ which have total turning $\theta_1$. The analogue of \tref{T:condensed}, stating that $\hat{\sr U}_c$ is either empty or contractible, is, naturally, \pref{P:excavator}, which was used to prove \tref{T:condensed}. The results and proofs  in \S\ref{S:diffuse} all need minimal or no modifications. In particular, $\hat{\sr U}_d$ is always nonempty and weakly contractible. 
 
 The proof that $\bd \hat{\sr U}_d=\hat{\sr T}$ is the same as the one given in \pref{P:boundary}.  The proof that $\bd\hat{\sr U}_c=\hat {\sr T}$, however, needs to be modified, since we have relied on \lref{L:submersion}. The idea is again to apply  construction \cref{C:excavator}, but to all of $\ga$, not just to some of its arcs as in the proof of \pref{P:boundary}. If $\ga$ is a critical curve, then the corresponding function $f$ (see Figure \ref{F:excavator}) will attain the values $\pm \infty$, and at these points we need to assign weights, corresponding to the lengths of the line segments where $\theta_\ga$ attains its maximum and minimum. Then we redefine $A(\mu_-,\mu_+)$ as the sum of the area under the graph of $f^{(\mu_-,\mu_+)}$ plus the weight at $+\infty$ minus the weight at $-\infty$. The process described in \cref{C:excavator} will transform $f$ into a bounded function of the same area, that is, it will decrease the amplitude of $\ga$, making it a condensed curve.
 
The proofs of \pref{P:condensedlocation}, \pref{P:criticallocation} and \cref{C:criticallocation}, which deal with the existence of condensed and critical curves, go through unchanged; the only difference in the conclusions is that the corresponding regions $R_{\hat{\sr U}_c}$, $R_\sig$ and $R_{\hat{\sr T}}$ of the plane are now closed, instead of open. Thus, in the analogue of \tref{T:criterion}, the region of Figure 1 should contain the two circles of radius 2, but not the circle of radius 4, and we cannot assert that $\hat{\sr U}_c$ and $\hat{\sr U}_d$ are homeomorphic to $\E$, only that they are weakly contractible. The rest of the statement and the proof hold without modifications. 

Similarly, the version of \tref{T:convex} for $\hat{\sr L}_{\ka_1}^{\ka_2}(P,Q)$ states that this space has one contractible connected component for each realizable total turning when $\ka_1\ka_2\geq 0$. The proof is the same as that of \tref{T:convex} if $\ka_1\ka_2>0$. If $\ka_1=0$, then we cannot really parametrize $\ga\in \hat{\sr L}_{\ka_1}^{\ka_2}(P,Q)$ by argument. Nevertheless, the proof still works if we replace $\rho(\theta)\?d\theta$ by a measure $\mu(\theta)$ on the Borel subsets of $[0,\theta_1]$ which has an atom at $\theta$ if the curvature of $\ga$ vanishes at $\ga(\theta)$; note that the convex combination of two measures is again a measure. The case where $\ka_2=0$ can be deduced from this one by reversing orientations. 


\subsection*{A few conjectures of Dubins}
All of the results in the next proposition were conjectured by Dubins in \S6 of \cite{Dubins1}.

\begin{prop}
Let $q\in \C$, $\theta_1\in \R$, $z=e^{i\theta_1}$ and $Q=(q,z)$. Then:
	\begin{enumerate}
		\item [(a)] The set of all $(q,\theta_1)\in \C\times \R$ such that $\hat{\sr L}_{-1}^{+1}(Q;\theta_1)$ is disconnected is a bounded subset of $\C\times \R$, neither open nor closed.\footnote{Actually, Dubins had guessed that this set would be bounded and open in $\C\times \R$.}
		\item [(b)] $\hat{\sr L}_{-1}^{+1}(Q;\theta_1)$ has at most two components.
		\item [(c)] If $\hat{\sr L}_{-1}^{+1}(Q;\theta_1)$ is disconnected, then one component $(\hat{\sr U}_d)$ contains curves of arbitrarily large length, while the supremum of the lengths of curves in the other component $(\hat{\sr U}_c)$ is finite.
		\item [(d)] Every point of $\C$ lies in the image of some $\ga\in \hat{\sr U}_d$, while the images of curves in $\hat{\sr U}_c$ are contained in a bounded subset of $\C$.
	\end{enumerate}
\end{prop}
\begin{proof}
	Parts (a) and (b) are immediate from the analogue of \tref{T:criterion} for $\hat{\sr L}$. As discussed above, $\hat{\sr L}_{-1}^{+1}(Q;\theta_1)$ is disconnected if and only if $\abs{\theta_1}<\pi$ and $q$ lies in the region in Figure \ref{F:S0} including the circles of radius 2 but not the circle of radius 4. Suppose that $q$ does lie in this region. Choose $\hat\om\in (\theta_1,\pi)$ such that
	\begin{equation*}
		\abs{q-\sign(\theta_1)i(z-1)}< 4\sin \Big(\frac{\hat\om}{2}\Big).
	\end{equation*}
	 Then \lref{L:omega} implies that there does not exist any curve in $\hat{\sr L}_{-1}^{+1}(Q;\theta_1)$ having amplitude in $[\hat\om,\pi]$. The assertions about $\hat{\sr U}_c$ in (c) and (d) now follow from \lref{L:bounded}. The assertions about $\hat{\sr U}_d$ are obvious, because, by (the version for $\hat{\sr L}$ of) \pref{P:diffuse}, this subspace always contains curves of amplitude $\geq 2\pi$, and onto such a curve we may graft line segments of any direction and arbitrary length.
\end{proof}


As expected, there is a version of the foregoing proposition for $\sr L_{-1}^{+1}(Q;\theta_1)$. The corresponding assertions in (a) and (b) are immediate from \tref{T:criterion}, and the assertions about $\sr U_d$ are again obvious. A curve in $\sr U_c$ parametrized by arc-length can also be considered as an element of $\hat{\sr U}_c$, so the properties stated in (c) and (d) for $\sr U_c$ follow from those for $\hat{\sr U}_c$ unless $q$ lies on the circle of radius 4 in Figure \ref{F:S0}. In this case, $\sr L_{-1}^{+1}(Q;\theta_1)$ is disconnected, but $\hat{\sr L}_{-1}^{+1}(Q;\theta_1)$ is not. One can prove directly that the length of any $\ga\in \sr U_c$ must be smaller than that of the ``canonical'' critical curves of type $\ty{+-}$ or $\ty{-+}$ that were constructed in the proof of \pref{P:criticallocation}.

\subsection*{Conjectures on minimal length}
	Let $L(\ga)$ denote the length of $\ga$ and suppose that $\hat{\sr L}_{-1}^{+1}(Q;\theta_1)$ is disconnected. We believe that the results developed here may be used to prove that if $m=\sup_{\ga\in \hat{\sr U}_c}L(\ga)$ and $M=\inf_{\ga\in \hat{\sr U}_d}L(\ga)$, then $m<M$; this is another conjecture of Dubins. It would be interesting, and probably useful for applications, to find the values corresponding to $m$ and $M$ for the more general spaces $\hat{\sr L}_{\kappa_1}^{\kappa_2}(Q)$. 
	
	We observed in \rref{R:respect} that normal translations, and hence the homeomorphisms of \tref{T:normalizedbar} need not preserve inequalities between lengths. Since they do map circles to circles and lines to lines, it could still be expected that the image of a curve which minimizes length under these homeomorphisms is likewise of minimal length. Unfortunately, this is false. Suppose for instance that we apply the homeomorphism $\hat{\sr L}_{-1}^{+1}(Q)\to \hat{\sr L}_{-1}^{100}(Q_0)$ to the Dubins path in Figure \ref{F:faster}\?(b). It should be clear that its image, which again consists of a line segment and two arcs of circles of opposite orientation with the same amplitude as before, does not minimize length in $\hat{\sr L}_{-1}^{100}(Q_0)$, since in the latter space it is generally much more efficient to curve to the left than to the right, even if this yields a path of greater total turning. 

In spite of this difficulty, we conjecture that Dubins' theorem that any shortest path in $\hat{\sr L}_{-\ka_0}^{+\ka_0}(P,Q)$ must be a concatenation of three pieces, each of which is either an arc of circle or a line segment, still holds for the spaces $\hat{\sr L}_{\kappa_1}^{\kappa_2}(P,Q)$, $\ka_1\ka_2<0$. For $\ka_1\ka_2> 0$, we conjecture that a curve of minimal length is  a concatenation of $n$ arcs of circles of curvature $\ka_1$ and $\ka_2$. However, for fixed $P,Q\in UT\C$, we must have $\lim_{\ka_1,\ka_2\to +\infty}n=\infty$. Indeed, a curve of this type has total turning at most $2n\pi$, and the minimal total turning of a curve in $\hat{\sr L}_{\ka_1}^{\ka_2}(Q)$ increases to infinity as $\ka_1>0$ increases to infinity (for fixed $Q=(q,z)$ with $q\neq 0$).

\subsection*{Acknowledgements}
The first author is partially supported by grants from \tsc{cnpq}, \tsc{capes} and \tsc{faperj} (Brazil). The second author gratefully acknowledges the financial support of \tsc{cnpq} and \tsc{fapesp}. Both authors would like to thank the anonymous referee, whose comments have helped to make the paper clearer.


\providecommand{\bysame}{\leavevmode\hbox to3em{\hrulefill}\thinspace}
\providecommand{\MR}{\relax\ifhmode\unskip\space\fi MR }
\providecommand{\MRhref}[2]{%
  \href{http://www.ams.org/mathscinet-getitem?mr=#1}{#2}
}
\providecommand{\href}[2]{#2}


\begin{thebibliography}{10}

\bibitem{Anisov}
Sergei~S. Anisov, \emph{Convex curves in {$RP^n$}}, Tr. Mat. Inst. Steklova
  \textbf{221} (1998), 9--47.

\bibitem{BurSalTom}
D.~Burghelea, N.~Saldanha, and C.~Tomei, \emph{Results on infinite-dimensional
  topology and applications to the structure of the critical sets of nonlinear
  \textrm{S}turm-\textrm{L}iouville operators}, J.~Differ.~Equations
  \textbf{188} (2003), 569--590.

\bibitem{BurKui}
Dan Burghelea and Nicolaas~H. Kuiper, \emph{Hilbert manifolds}, Ann. of Math.
  \textbf{90} (1969), no.~3, 379--417.

\bibitem{Dubins}
Lester~E. Dubins, \emph{On curves of minimal length with a constraint on
  average curvature, and with prescribed initial and terminal positions and
  tangents}, Amer.~J.~Math. \textbf{79} (1957), no.~3, 497--516.

\bibitem{Dubins1}
\bysame, \emph{On plane curves with curvature}, Pacific J.~Math. \textbf{11}
  (1961), no.~2, 471--481.

\bibitem{Feldman}
E.~A. Feldman, \emph{Deformations of closed space curves}, J. Diff. Geom.
  \textbf{2} (1968), no.~1, 67--75.

\bibitem{Feldman1}
\bysame, \emph{Nondegenerate curves on a riemannian manifold}, J. Diff. Geom.
  \textbf{5} (1971), no.~1, 187--210.

\bibitem{Henderson}
David Henderson, \emph{Infinite-dimensional manifolds are open subsets of
  {H}ilbert space}, Bull. Amer. Math. Soc. \textbf{75} (1969), no.~4, 759--762.

\bibitem{Hopf}
Heinz Hopf, \emph{Zum {C}lifford-{K}leinschen {R}aumproblem}, Math. Ann.
  \textbf{95} (1926), 313--339.

\bibitem{KheSha}
Boris~A. Khesin and Boris~Z. Shapiro, \emph{Nondegenerate curves on
  $\mathit{S}^2$ and orbit classification of the \textrm{Z}amolodchikov
  algebra}, Commun. Math. Phys. \textbf{145} (1992), 357--362.

\bibitem{KheSha1}
\bysame, \emph{Homotopy classification of nondegenerate quasiperiodic curves on
  the 2-sphere}, Publ. Inst. Math. (Beograd) \textbf{66 (80)} (1999), 127--156.

\bibitem{Little}
John~A. Little, \emph{Nondegenerate homotopies of curves on the unit 2-sphere},
  J.~Differential Geom. \textbf{4} (1970), 339--348.

\bibitem{Little1}
\bysame, \emph{Third order nondegenerate homotopies of space curves},
  J.~Differential Geom. \textbf{5} (1971), 503--515.

\bibitem{MosSad}
Jacob Mostovoy and Rustam Sadykov, \emph{The space of non-degenerate closed
  curves in a {R}iemannian manifold}, arXiv:1209.4109, 2012.

\bibitem{Palais}
Richard Palais, \emph{Homotopy theory of infinite dimensional manifolds},
  Topology \textbf{5} (1966), 1--16.

\bibitem{Saldanha3}
Nicolau~C. Saldanha, \emph{The homotopy type of spaces of locally convex curves
  in the sphere}, Geom. Topol. \textbf{19} (2015), 1155--1203.

\bibitem{SalSha}
Nicolau~C. Saldanha and Boris~Z. Shapiro, \emph{Spaces of locally convex curves
  in $\mathit{S}^n$ and combinatorics of the group $\mathit{B}_{n+1}^{+}$},
  J.~Singul. \textbf{4} (2012), 1--22.

\bibitem{SalZueh}
Nicolau~C. Saldanha and Pedro Z\"{u}hlke, \emph{On the components of spaces of
  curves on the 2-sphere with geodesic curvature in a prescribed interval},
  Internat. J. Math. \textbf{24} (2013), no.~14, 1--78.

\bibitem{SalZuehold}
\bysame, \emph{Spaces of curves with constrained curvature on flat surfaces,
  {I}}, preprint available at \texttt{arxiv.org/abs/1312.1675v3}, 2013.

\bibitem{SalZueh2}
\bysame, \emph{Homotopy type of spaces of curves with constrained curvature on
  flat surfaces}, preprint available at \texttt{arxiv.org/abs/1410.8590}, 2014.

\bibitem{ShaSha}
Boris~Z. Shapiro and Michael~Z. Shapiro, \emph{On the number of connected
  components in the space of closed nondegenerate curves on $\mathit{S}^n$},
  Bull. Amer. Math. Soc. \textbf{25} (1991), no.~1, 75--79.

\bibitem{ShaTop}
Michael~Z. Shapiro, \emph{Topology of the space of nondegenerate curves}, Math.
  USSR \textbf{57} (1993), 106--126.

\bibitem{Smale}
Stephen Smale, \emph{Regular curves on {R}iemannian manifolds}, Trans. Amer.
  Math. Soc. \textbf{87} (1956), no.~2, 492--512.

\bibitem{WhiGra}
Hassler Whitney, \emph{On regular closed curves in the plane}, Compos.~Math.
  \textbf{4} (1937), no.~1, 276--284.

\bibitem{Younes}
Laurent Younes, \emph{Shapes and diffeomorphisms}, Springer-Verlag, 2010.

\end{thebibliography}

\vspace{12pt} \noindent{\small \tsc{Departamento de Matem\'atica, Pontif\'icia
	Universidade Cat\'olica do Rio de Janeiro (PUC-Rio)\\ Rua
Marqu\^es de S\~ao Vicente 225, G\'avea -- Rio de Janeiro, RJ 22453-900, Brazil}}\\
\noindent{\ttt{nicolau@mat.puc-rio.br}}

\vspace{12pt} \noindent{\small \tsc{Instituto de Matem\'atica e Estat\'istica,
	Universidade de S\~ao Paulo (IME-USP) \\ Rua~do Mat\~ao 1010, Cidade
	Universit\'aria -- S\~ao Paulo, SP
05508-090, Brazil}}\\
\noindent{\ttt{pedroz@ime.usp.br}}

\vspace{2pt}

\end{document}